\documentclass{article}
\usepackage{fullpage,tikz,calc}
\usepackage[utf8]{inputenc}
\usepackage{color,amsfonts,amsmath,amssymb,amsthm,cases,wasysym}
\usepackage[normalem]{ulem}
\usepackage{svg}
\usepackage{float}
\usepackage{mathrsfs}

\title{Boundary Algebras of Positroids}
\author{Jonah Berggren and Khrystyna Serhiyenko}
\date{}

\newtheorem{thm}{Theorem}[section]
\newtheorem{prop}[thm]{Proposition}
\newtheorem{lemma}[thm]{Lemma}
\newtheorem{cor}[thm]{Corollary}

\newtheorem{thmIntro}{Theorem}    
    
\theoremstyle{definition}
\newtheorem{exIntro}[thmIntro]{Example}    
\newtheorem{defn}[thm]{Definition}
\newtheorem{remk}[thm]{Remark}
\newtheorem{example}[thm]{Example}

\newcommand{\reach}{\text{reach}}
\newcommand{\Le}{\text{Le}}

\newcommand{\A}{\mathscr A}
\newcommand{\CL}{\textup{CL}}
\newcommand{\CC}{\textup{CC}}
\renewcommand{\k}{\mathbb C}
\renewcommand{\AA}{\mathscr A}
\renewcommand{\P}{\mathcal P}
\newcommand{\GP}{\textup{GP}}
\newcommand{\Gr}{\textup{Gr}}
\newcommand{\X}{X}
\newcommand{\Y}{Y}

\begin{document}

\maketitle

\begin{abstract}
	A dimer model is a quiver with faces embedded into a disk. A consistent dimer model gives rise to a strand diagram, and hence to a positroid.
	The Gorenstein-projective module category over the completed boundary algebra of a dimer model was shown by Pressland to categorify a cluster structure 
	 on the corresponding positroid variety.
	Outside of the Grassmannian case, boundary algebras of dimer models are not well understood. We give an explicit description of the boundary algebra of a consistent dimer model as a quiver with relations calculated only from the data of the decorated permutation or, equivalently, Grassmann necklace of its positroid.
\end{abstract}

\section{Introduction}

A dimer model is a quiver with faces embedded into a surface. {Dimer models on tori and other closed surfaces were introduced to study phase transitions in physics and are connected to numerous topics in combinatorics, representation theory, and algebraic geometry; see the survey~\cite{ZBocklandt2015} and the references therein. In this paper, we focus on dimer models on disks and their connections to cluster structures on positroid varieties.} In the following, all dimer models are assumed to be on disks. Associated to a dimer model $Q$ is a \emph{dimer algebra} $A_Q$, obtained as a quotient of the path algebra $\k Q$. The \emph{boundary algebra} $B_Q$ of $Q$ is the subalgebra of $A_Q$ induced by (equivalence classes of) paths starting and ending at boundary vertices.
A dimer model gives rise to a \emph{strand diagram}, which  is a collection of strands in a disk; following these strands gives a decorated permutation, and hence a positroid $\mathcal P_Q$.
Of particular interest are \emph{consistent} dimer models, which (among other things) give connected decorated permutations and positroids.

Let $\mathcal P$ be a positroid with basis system $\mathcal B_\P\subseteq{[n]\choose k}$ and Grassmann necklace~\cite[Definition 16.1]{ZPostnikov} $\mathcal I_\P$. 
The \emph{(open) positroid variety} $\Pi^\circ(\mathcal P)$, introduced and connected with the Grassmannian and its stratifications by Knutson-Lam-Speyer~\cite{KLSZ}, is the subvariety of the Grassmannian $\Gr(k,n)$ defined by the vanishing of Pl\"ucker coordinates not in $\mathcal B_\P$ and the nonvanishing of Pl\"ucker coordinates in $\mathcal I_\P$. We denote the affine cone over this variety by $\widetilde\Pi^\circ(\mathcal P)$.
Recently, a number of works have developed cluster structures on positroid varieties and connected them to consistent dimer models.

When $\mathcal P$ is the uniform $(k,n)$-positroid, {the positroid variety $\Pi^\circ(\mathcal P)$ is the Grassmannian $\text{Gr}(k,n)$.} This special case received attention first.
Scott~\cite{ZScott} showed that the homogeneous coordinate ring of the Grassmannian $\text{Gr}(k,n)$ is a cluster algebra, in which certain seeds are indexed by $(k,n)$-Postnikov diagrams (i.e., strand diagrams with $n$ marked boundary points and no bad configurations whose decorated permutations send $j\mapsto j-k$).
Gei\ss-Leclerc-Schroer~\cite{GLSZ} gave a categorification for this cluster structure, and Jensen-King-Su~\cite{ZJKS} phrased this categorification as the Gorenstein-projective module category over a bound path algebra $\k Q_{k,n}/I_{k,n}$ of the \emph{circle quiver} $Q_{k,n}$.
\nocite{DLZ}
Baur-King-Marsh~\cite{ZBKM} connected this story back to strand diagrams and dimer models by defining the \emph{boundary algebra} $B_Q=eA_Qe$ of a dimer model $Q$, where $A_Q$ is the dimer algebra and $e$ is the idempotent given by the sum of all boundary vertices. They showed that the bound path algebra $\k Q_{k,n}/I_{k,n}$ is isomorphic to the boundary algebra $B_Q$ of any consistent dimer model whose decorated permutation sends $j\mapsto j-k$.

This story was then extended to the case where $\mathcal P$ is a general connected $(k,n)$-positroid. For convenience, we make the additional assumption that $n\geq3$ (the case $n=2$ includes a small number of degenerate examples~\cite[Remark 2.2]{ZPressland2015}).
Galashin-Lam~\cite{GLZ}, generalizing work~\cite{SSWZ} on Schubert varieties, showed that the open positroid variety $\widetilde\Pi^\circ(\mathcal P)$ has the structure of a cluster algebra $\AA_Q$. Both of these works relied heavily on results of Leclerc~\cite{ZZL} for Richardson varieties.
Pressland~\cite[Theorem 6.11]{ZPressland2019} showed that this cluster structure may be obtained as the Gorenstein-projective module category $\GP(\widehat B_Q)$. The category $\GP(\widehat B_Q)$ is an additive categorification of the cluster algebra $\AA_Q$ whose initial seed is given by the cluster-tilting object $eA_Q$.
The Fu-Keller cluster character~\cite{FKZ} provides a bijection between the reachable cluster-tilting objects of $(\textup{GP}(\widehat B_Q),eA_Q)$ and the cluster variables of $\A_Q$.

This categorification may be used to better understand the Galashin-Lam cluster structure on positroid varieties.
Recently, Pressland used it to prove the conjecture of Muller and Speyer~\cite{ZZMS} that two complementary cluster structures on positroid varieties, coming from labelling the regions of the strand diagram using sources or targets of strands, quasi-coincide~\cite{ZZP}. A key ingredient in this proof comes from the categorical interpretation of perfect matchings and twists developed in~\cite{CKP}.
{Other recent work focuses on the Grassmannian case; in particular, a series of works by Baur, Bogdanic, Elsener, and Li~\cite{BBEL,ZZBBE,ZZBBL} describe the Gorenstein-projective modules over the circle algebra $\k Q_{k,n}/I_{k,n}$ corresponding to rank 2 and 3 cluster variables.}

On the other hand, one limitation to using this categorification for general positroid varieties is that outside of the case where $\mathcal P_Q$ is uniform, the boundary algebras $B_Q$ are not well understood.
In this paper, we give a description of the boundary algebra $B_Q$ as a quiver with relations calculated from the decorated permutation or Grassmann necklace of $Q$. This is an important step in understanding the additive categorification given by $\GP(\widehat B_Q)$.

Let $Q$ be a consistent dimer model with decorated permutation $\pi$. A path $p$ of $Q$ is \emph{minimal} if no path equivalent to $p$ contains a cycle of $Q$. We number the boundary vertices of $Q$ from 1 to $n$ increasing clockwise. We let $x_i:i\to i+1$ and $y_i:i+1\to i$ be the minimal clockwise and counter-clockwise paths between cyclically adjacent vertices. For convenience, we notate $x_i^m:=x_ix_{i+1}\dots x_{i+m-1}$ and $y_i^m:=y_iy_{i-1}\dots y_{i-m+1}$. If $p$ is a path of $Q$ starting and ending at boundary vertices $t(p)$ and $h(p)$, then we define $\reach_y(p)\in[n]$ so that $y_{t(p)-1}^{\reach_y(p)}$ is a path from $t(p)$ to $h(p)$. We define the \emph{relation number} ${\Y}(p)$ of a minimal path $p$ such that $[p(xy)^{{\Y}(p)}]=[y_{t(p)-1}^{\reach_y(p)}]$. Intuitively, the relation number ${\Y}(p)$ measures how far $p$ is from being equivalent to a composition of $y_i$'s. One may similarly define ${\X}(p)$ so that $[p(xy)^{{\X}(p)}]=[x_{t(p)}^{\reach_x(p)}]$.

We label the marked boundary points of the strand diagram of $Q$ from 1 to $n$ such that the strand-vertex $i$ is immediately clockwise of the vertex $i\in Q_0$; see the left of Figure~\ref{fig:intro}. If $v\in Q_0$ is a boundary vertex, we write $v^{cl}$ and $v^{cc}$ for the strand-vertices immediately clockwise and counter-clockwise of $v$, respectively. In the below theorem, all intervals are modulo $[n]$, so that $(v_2,v_1)=\{v_2+1,v_2+2,\dots,v_1-1\}$ (modulo $n$).

The following result gives a purely combinatorial characterization of the arrow-defining paths of $Q$ (i.e., the paths of $Q$ giving rise to arrows of the Gabriel quiver of the boundary algebra $B_Q$) and gives us a relation for each one. We will then obtain a representation of the boundary algebra in terms of only this information.

\begin{thmIntro}[{Theorem~\ref{thm:main-perm}}]\label{thm:0}
	Let $Q$ be a consistent dimer model with decorated permutation $\pi$. Let $v_1$ and $v_2$ be distinct elements of $[n]$ (considered as boundary vertices of $Q$). The minimal path $p$ from $v_1$ to $v_2$ is arrow-defining if and only if
	\begin{enumerate}
		\item\label{IGF1E} for every $w\in(v_2,v_1)$, there is a number $j\in[v_2^{cl},w^{cc}]$ such that $\pi(j)\in[w^{cl},v_1^{cc}]$, and
		\item\label{IGF2E} for every $w\in(v_1,v_2)$, there is a number $j\in[w^{cl},v_2^{cc}]$ such that $\pi(j)\in[v_1^{cl},w^{cc}]$.
	\end{enumerate}
	In this case, ${\Y}(p)=\#\{i\in[n]\ :\ v_2^{cl}\leq i<\pi(i)\leq v_1^{cc}\textup{ is a clockwise ordering}\}$ and 
	${\X}(p)=\#\{j\in[n]\ :\ v_1^{cl}\leq\pi(j)<j\leq v_2^{cc}\textup{ is a clockwise ordering}\}$.
\end{thmIntro}

{Theorem~\ref{thm:0} is proven by first characterizing minimal rightmost and leftmost paths of $Q$ in terms of the strand diagram, and then observing that we may understand when these paths factor through extra boundary vertices by looking only at the underlying strand permutation.
}
Condition~\eqref{IGF1E} of Theorem~\ref{thm:0} states precisely that for any vertex $w$ clockwise of $v_2$ and counter-clockwise of $v_1$, there is some strand $z$ such that $v_1$ and $v_2$ are to the right of $z$, but $w$ is to the left of $z$. Similarly, condition~\eqref{IGF2E} states that for a vertex $w\in (v_1,v_2)$, there is some strand $z$ such that $v_1$ and $v_2$ are to the left of $z$, but $w$ is to the right of $z$. This view lends to a nice interpretation of Theorem~\ref{thm:0} in terms of Grassmann necklaces:
\begin{thmIntro}[{Corollary~\ref{cor:main-neck}}]\label{thm:0.5}
	Let $Q$ be a consistent dimer model with Grassmann necklace $\mathcal I_Q=(I_1,\dots,I_n)$.
	Let $v_1$ and ${v_2}$ be distinct elements of $[n]$. There is an arrow-defining path $p$ from $v_1$ to $v_2$ if and only if
	\begin{enumerate}
		\item\label{FG1intro} for any $w\in(v_2,v_1)$, we have $I_w\not\supseteq I_{v_1}\cap I_{v_2}$, and
		\item\label{FG2intro} for any $w\in(v_1,v_2)$, we have $I_w\not\subseteq I_{v_1}\cup I_{v_2}$.
	\end{enumerate}
	The relation numbers of $p$ may be calculated as ${\Y}(p)=\big|[v_2,v_1)\cap I_{v_1}\cap I_{v_2}\big|$ and ${\X}(p)=\big|[v_1,v_2)\backslash\big(I_{v_1}\cup I_{v_w}\big)\big|$.
\end{thmIntro}

Theorem~\ref{thm:0} or~\ref{thm:0.5} may be used to find all of the arrows of the Gabriel quiver of $B_Q$ and their relation numbers using the decorated permutation or Grassmann necklace of $Q$.
We then show that the information of the \emph{nonadjacent} arrows of $B_Q$ and their relation numbers is sufficient to calculate the boundary algebra.

Dimer algebras and boundary algebras of consistent dimer models are \emph{cancellative}, meaning that
for paths $p,q,a,b$ of $Q$ which compose appropriately, we have $[ap]=[aq]\iff [p]=[q]$ and $[pb]=[qb]\iff [p]=[q]$. Given an ideal $I$ of a path algebra $\k Q$, we say that the \emph{cancellative closure} $I^\bullet$ of $I$ is the smallest ideal $I^\bullet\supseteq I$ such that $\k Q/I^\bullet$ is cancellative.

We now give a presentation of the boundary algebra $B_Q$ by defining a new path algebra with relations $\k Q_\circ^\pi/I_\circ^\pi$.
Let $Q_{k,n}$ be the \emph{circle quiver} on vertex set $[n]$ such that for each $i$ there is an arrow $x_i:i\to i+1$ and an arrow $y_i:i+1\to i$ (we are abusing notation by also using, for example, $x_i$ to refer to the minimal path from $i$ to $i+1$ in the dimer algebra $A_Q$). In the path algebra $\k Q_{k,n}$, we write $x=\sum_{i\in[n]}x_i$ and $y=\sum_{i\in[n]}y_i$. Baur, King, and Marsh showed in~\cite[Corollary 10.4]{ZBKM} that the boundary algebra of a consistent dimer model whose decorated permutation sends $j\mapsto j-k$ is isomorphic to $\k Q_{k,n}/I_{k,n}$, where $I_{k,n}$ is the ideal generated by the relations $[xy]-[yx]$ and $[x^k]-[y^{n-k}]$.
We concern ourselves with the general case. Let $k$ be the number of noninversions of the decorated permutation $\pi$ of $Q$.
Let $Q^{\pi}_{\circ}$ be the quiver obtained by starting with the circle quiver $Q_{k,n}$ and drawing an arrow $\alpha_p:t(p)\to h(p)$ for every nonadjacent arrow-defining path $[p]$ of $Q$, up to path-equivalence. Let $I^{\pi}_{\circ}$ be the cancellative closure of the ideal generated by the Grassmannian relations $[xy]-[yx]$ and $[x^k]-[y^{n-k}]$ and the {nonadjacent relations}:
\[I^{\pi}_{\circ}=(\{[xy]-[yx],\ [x^k]-[y^{n-k}],\ \sum_{p\in\textup{NBP}(Q)}[y_{t(p)-1}^{\reach_y(p)}]-[p(xy)^{{\Y}(p)}]\})^\bullet,\]
where $\textup{NBP}(Q)$ is the set of nonadjacent arrow-defining paths of $Q$ up to equivalence. One may symmetrically define $I_\circ^\pi$ in terms of the relations given by ${\X}(P)$.
\begin{thmIntro}\label{thm:1}
	The boundary algebra $B_Q$ is isomorphic to $B^\pi_\circ:=\k Q_\circ^\pi/I_\circ^\pi$.
\end{thmIntro}
Note that Theorem~\ref{thm:0} gives a way to calculate elements $p\in\textup{NBP}(Q)$ and their relation numbers ${\Y}(p)$ based only on the permutation $\pi$, hence we may calculate $Q_\circ^\pi$ and $I_\circ^\pi$ based only on the decorated permutation $\pi$, justifying the notation.
The ideal $I^\pi_\circ$ may not be admissible, as certain adjacent paths $x_i$ and $y_i$ may fail to be arrows of $Q_\circ^\pi$. This issue may be detected through the strand diagram using Theorem~\ref{thm:0}. Hence, to obtain an admissible bound path algebra one may remove the adjacent arrows $x_i$ and $y_i$ of $Q^\pi_\circ$ which are not arrows of the Gabriel quiver of boundary algebra $B_Q$ and adjust the ideal $I^\pi_\circ$ accordingly.

We point out that Theorem~\ref{thm:1} implies that the pair $(k,n)$ along with the set of nonadjacent arrows of the Gabriel quiver of $B_Q$ and their relation numbers is sufficient to recover the boundary algebra.
We will study this further in future work.
The usage of a cancellative closure in the definition of $I_\circ^\pi$ means that we do not in general have a minimal generating set of the relations of $I_\circ^\pi$. This will also be addressed in future work.

\begin{exIntro}
	\begin{figure}[H]
		\centering
		\def\svgscale{0.21}
		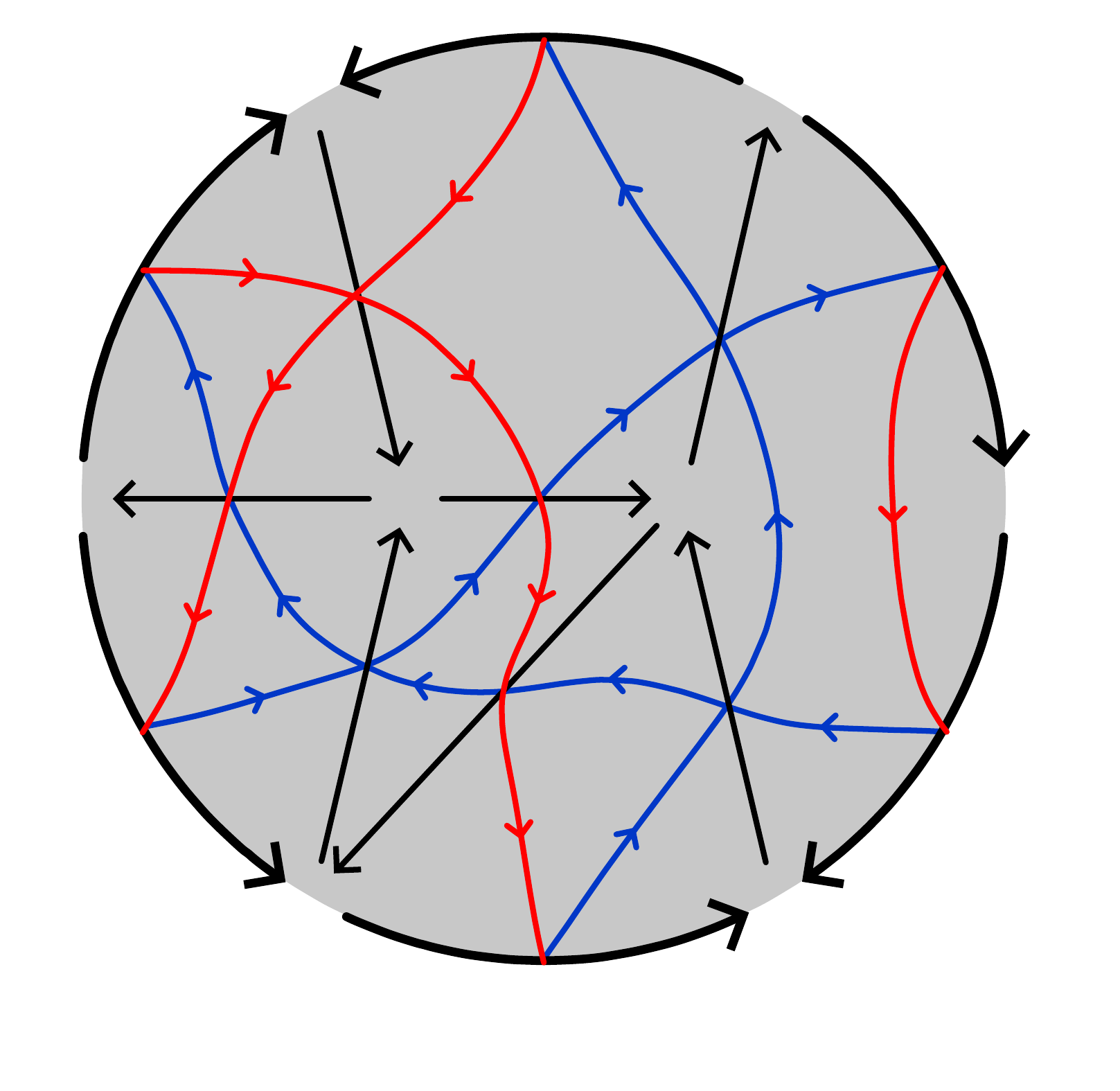
		\hspace{1cm}
		\def\svgscale{0.21}
		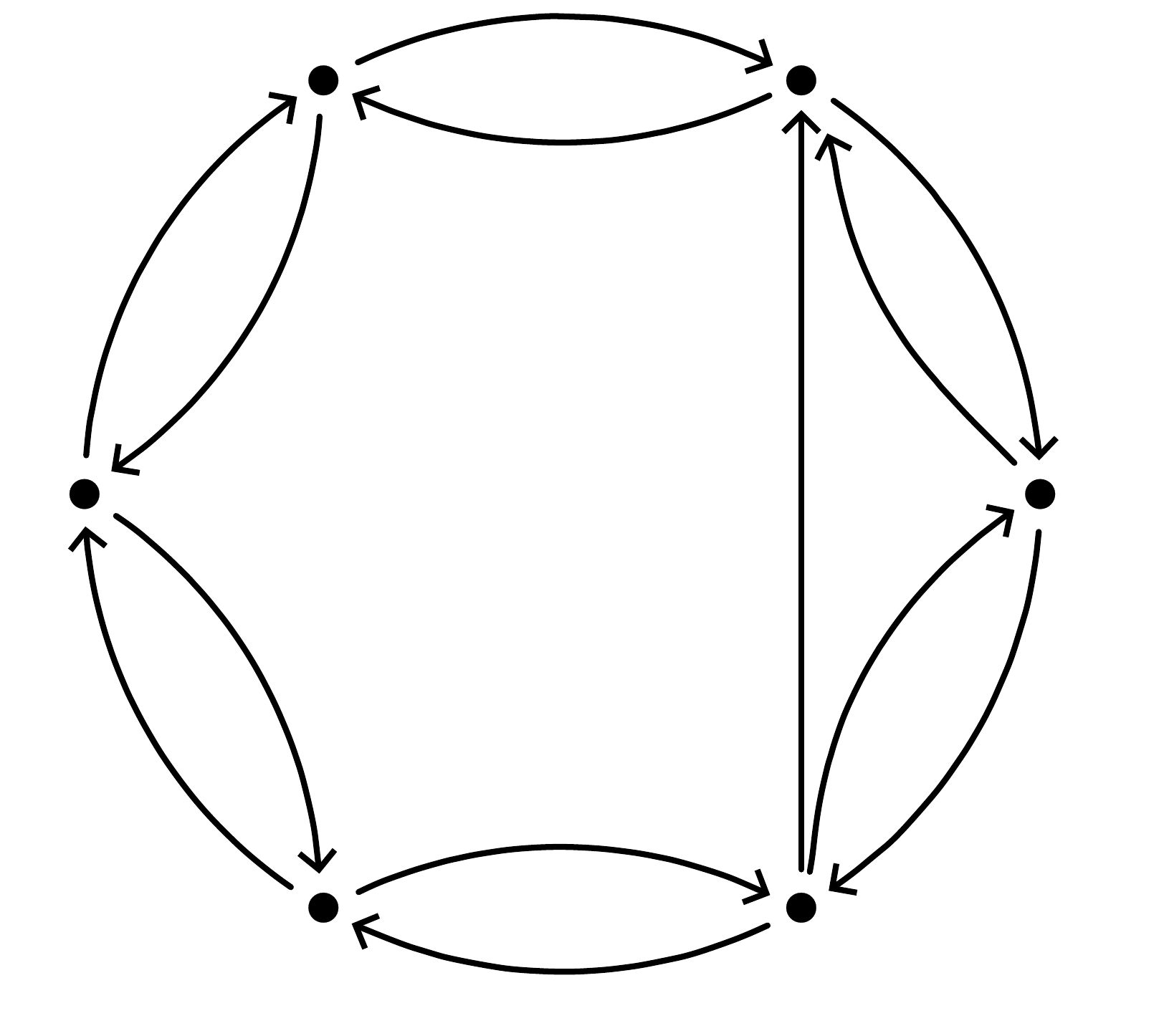
		\caption{On the left is a consistent dimer model $Q$ and on the right is the quiver $Q_\circ^\pi$.}
		\label{fig:intro}
	\end{figure}
	Consider the dimer model in Figure~\ref{fig:intro} with permutation $\pi:=256134$ in one-line notation and Grassmann necklace $\left(123,\ 234,\ 134,\ 145,\ 156,\ 126\right)$. We use, for example, 123 to denote the set $\{1,2,3\}$ for readability.
	First, we will use Theorem~\ref{thm:0.5} to understand the arrow-defining paths and their relation numbers.
	Setting $(v_1,v_2)=(3,1)$, we see that $I_{v_1}\cap I_{v_2}=134\cap 123=13$ is not contained in $I_w$ for any $w\in(v_2,v_1)=(1,3)=\{2\}$ (since $I_2=234\not\supseteq 13$), showing that Theorem~\ref{thm:0.5}~\eqref{FG1intro} is satisfied.
	We now treat~\eqref{FG2intro}. Note that $(v_1,v_2)=(3,1)=\{4,5,6\}$ and $I_{v_1}\cup I_{v_2}=123\cup 134=1234$. Since neither $I_4=145$, nor $I_5=156$, nor $I_6=126$ are contained in $1234$, we see that~\eqref{FG2intro} is satisfied.
	Hence, there is an arrow-defining path $p$ from 1 to 3 in the dimer algebra $A_Q$. Its relation numbers are ${\Y}(p)=|[1,3)\cap 123\cup 134|=|2|=1$ and ${\X}(p)=|[3,1)\backslash(134\cup 123)|=|56|=2$.
	On the other hand, set $(v_1,v_2)=(4,1)$. Then $I_3=134$ contains $I_1\cap I_4=123\cap 145=1$, hence~\eqref{FG1intro} fails and there is no arrow-defining path from 4 to 1 in $Q$.
	One may thus check that the only nonadjacent arrow-defining path is from 3 to 1.
	All of the above calculations may instead be phrased in terms of Theorem~\ref{thm:0}. 
	On the right of Figure~\ref{fig:intro} is shown the quiver $Q_\circ^\pi$, where the nonadjacent arrow $\alpha$ has relation number ${\Y}(\alpha)=1$. Theorem~\ref{thm:1} shows that $B_Q=Q_\circ^\pi/I_\circ^\pi$, where $I_\circ^\pi$ is the cancellative closure of
	\[\{[xy]-[yx],\ [x^3]-[y^3],\ [\alpha(xy)]-[y_2y_1]\}.\]
\end{exIntro}

The structure of the paper is as follows.
In Section~\ref{sec:dmb}, we give background on dimer models. We define consistency and connect it to strand diagrams and cancellativity.
Next, in Section~\ref{sec:TDM}, we show that the Gabriel quiver of the boundary algebra $B_Q$ may be understood in terms of rightmost and leftmost paths in the dimer model. In Section~\ref{sec:carn}, we show how rightmost and leftmost minimal paths and their relation numbers may be obtained from the strand diagram. We use this result to find these paths and relation numbers from the permutation or Grassmann necklace. 
Finally, in Section~\ref{sec:dba}, we show how to use the calculations of Section~\ref{sec:carn} to describe the boundary algebra of a consistent dimer model with a given permutation.

\subsection*{Acknowledgments}

The authors thank Jonathan Boretsky for useful discussions.
The authors were supported by the NSF grant DMS-2054255.

\section{Dimer Model Background}\label{sec:first-section}
\label{sec:dmb}

In this section, we give necessary background. First, we focus on dimer models, in particular on \emph{thin} (or \emph{consistent}) dimer models in disks. Thin dimer models may be defined using algebraic conditions or conditions on the associated strand diagram, and are connected with positroids through strand permutations. We also define dimer submodels, which will be an important tool in later proofs.

\subsection{Dimer Model Definition}

We define dimer models, following the exposition of~\cite{BS}.
A \textit{quiver} is a directed graph. A \textit{cycle} of $Q$ is a nonconstant oriented path of $Q$ which starts and ends at the same vertex.
If $Q$ is a quiver, we write $Q_{cyc}$ for the set of cycles in $Q$ up to cyclic equivalence. An element of $Q_{cyc}$ is the set of arrows in some cycle.

\begin{defn}
    A \textit{quiver with faces} is a triple $Q=(Q_0,Q_1,Q_2)$, where $(Q_0,Q_1)$ are the vertices and arrows of a quiver and $Q_2\subseteq Q_{cyc}$ is a set of \textit{faces} of $Q$.
\end{defn}

Given a vertex $i\in Q_0$, we define the \textit{incidence graph} of $Q$ at $i$ to be the graph whose vertices are given by the arrows incident to $i$ and whose arrows $\alpha\to\beta$ correspond to paths
\[\xrightarrow{\alpha}i\xrightarrow{\beta}\]
which occur in faces of $Q$.

\begin{defn}\label{defn:dimer-model}
	A (finite, oriented) \textit{dimer model with boundary} is given by a finite quiver with faces $Q=(Q_0,Q_1,Q_2)$, where $Q_2$ is written as a disjoint union $Q_2=Q_2^{cc}\cup Q_2^{cl}$, satisfying the following properties:
    \begin{enumerate}
        \item The quiver $Q$ has no loops.
        \item Each arrow of $Q_1$ is in either one face or two faces of $Q$. An arrow which is in one face is called a \textit{boundary arrow} and an arrow which is in two faces is called an \textit{internal arrow}.
	\item Each internal arrow lies in a cycle bounding a face in $Q_2^{cc}$ and in a cycle bounding a face in $Q_2^{cl}$.
	{\item\label{ddm:4} The incidence graph of $Q$ at each vertex is connected.}
    \end{enumerate}
\end{defn}

See the left of Figure~\ref{fig:ex2} for an example of a dimer model.

Given a dimer model with boundary $Q$ we may associate each face $F$ of $Q$ with a polygon whose edges are labeled by the arrows in $F$ and glue the edges of these polygons together as indicated by the directions of the arrows to form a surface with boundary $S(Q)$ into which $Q$ may be embedded. The surface $S(Q)$ is oriented such that the cycles of faces in $Q_2^{cc}$ are oriented positive (or counter-clockwise) and the cycles of faces in $Q_2^{cl}$ are oriented negative (or clockwise). The boundary of $S(Q)$ runs along the boundary arrows of $Q$. If $S(Q)$ is a disk, then we say that $Q$ is a \textit{dimer model on a disk}.

We may also move in the other direction. Suppose that $Q$ is a finite quiver with no loops such that every vertex has finite degree. Suppose further that $Q$ has an embedding into an oriented surface $\Sigma$ with boundary such that the complement of $Q$ in $\Sigma$ is a disjoint union of discs, each of which is bounded by a cycle of $Q$. We may then view $Q$ as a dimer model with boundary by declaring $Q_2^{cc}$ (respectively $Q_2^{cl}$) to be the set of positively (respectively, negatively) oriented cycles of $Q$ which bound a connected component of the complement of $Q$ in $\Sigma$. All dimer models may be obtained in this way.

Let $Q$ be a dimer model and let $p$ be a path in $Q$. We write $t(p)$ and $h(p)$ for the start and end vertex of $p$, respectively. If a path $q$ can be factored in the form $q=q_1pq_2$, where $h(q_1)=t(p)$ and $t(q_2)=h(p)$, we say that $p$ is in $q$ or that $q$ contains $p$ as a subpath and we write $p\in q$. Corresponding to any vertex $v$ is a \textit{constant path} $e_v$ from $v$ to itself which has no arrows. A {cycle} is then a nonconstant path $l$ such that $t(l)=h(l)$. A path $p$ \emph{cycleless} if no subpath of $p$ (including $p$ itself) is a cycle. 

\begin{defn}\label{defn:dimer-algebra}
	Given a dimer model with boundary $Q$, the \textit{dimer algebra} $A_Q$ is defined as the quotient of the path algebra $\k Q$ by (the ideal generated by) the relations
	\[[R_\alpha^{cc}]-[R_\alpha^{cl}]\]
    for every internal arrow $\alpha\in Q_1$.
\end{defn}

We now make more definitions. We say that two paths $p$ and $q$ in $Q$ are \textit{path-equivalent} if their associated elements in the dimer algebra $A_Q$ are equal. If $p$ is a path in $Q$, we write $[p]$ for the path-equivalence class of $p$ under these relations.
If $p'$ is a path obtained from $p$ by replacing a single subpath $R_\alpha^{cl}$ with $R_\alpha^{cc}$ ($R_\alpha^{cc}$ with $R_\alpha^{cl}$, respectively), then we say that $p'$ is a \textit{basic right-morph} (respectively, \textit{basic left-morph}) of $p$.
Since the ideal $I$ is generated by the relations $\{[R_\alpha^{cc}]-[R_\alpha^{cl}]\ :\ \alpha\textup{ is an internal arrow of }Q\}$, two paths $p$ and $q$ are path-equivalent if and only if 
there is a sequence of paths $p=r_1,\dots,r_m=q$ such that $r_{i+1}$ is a basic (left- or right-) morph of $r_i$ for $i\in[m-1]$.
A path is \textit{leftmost} (respectively \textit{rightmost}) if it has no basic left-morphs (respectively basic right-morphs).

Suppose $p$ is a cycle in $Q$ which starts and ends at some vertex $v$ and travels around a face of $Q$ once. Then we say that $p$ is a \textit{face-path} of $Q$ starting at $v$. The terminology is justified by the following observation which follows from the defining relations.

\begin{remk}\label{faces-are-equivalent}
    Any two face-paths of $Q$ starting at $v$ are path-equivalent.
\end{remk}

\begin{defn}
    For all $v\in Q_0$, fix some face-path $f_v$ at $v$. Then define
    \begin{equation}
	    f:=\sum_{v\in Q_0}f_v.
    \end{equation}
\end{defn}
It follows from Remark~\ref{faces-are-equivalent} that the path-equivalence class $[f]$ is independent of the choice of $f_v$ for all $v\in Q_0$. Moreover, the dimer algebra relations imply that $[f]$ commutes with every arrow, hence $[f]$ is in the center of $A_Q$. 

The \textit{completed path algebra} $\langle\langle\k Q\rangle\rangle$ has as its underlying set the \textit{possibly infinite} linear combinations of finite paths in $Q$, with multiplication induced by composition. See~\cite[Definition 2.6]{ZPressland2020}.

\begin{defn}\label{defn:completed-dimer-algebra}
	The \textit{completed dimer algebra} $\widehat A_Q$ is the quotient of the completed path algebra $\k \langle\langle Q\rangle\rangle$ by the closure of the ideal generated by the relations $[R_\alpha^{cc}]-[R_\alpha^{cl}]$ for each internal arrow $\alpha$ with respect to the arrow ideal.
\end{defn}
Elements of $\widehat A_Q$ are possibly infinite linear combinations of (finite) paths of $Q$, with multiplication induced by composition.

\begin{defn}\label{defn:ba}
	Let the \emph{boundary idempotent} $e$ be the sum of all boundary vertices of $Q$. A \emph{boundary path} of $Q$ is a path beginning and ending at a boundary vertex. Define the \emph{boundary algebra} $B_Q:=eA_Qe$. Elements of the boundary algebra are linear combinations of boundary paths of $Q$ under the relations of the dimer algebra.
	The \textit{completed boundary algebra} is $\widehat B_Q\cong e\widehat A_Qe$.
\end{defn}

\subsection{Thin Dimer Models}\label{sec:hc}

For the rest of this paper, we will assume that any dimer model is on a disk, unless otherwise specified.
We define \textit{thin} dimer models on disks and give two additional equivalent consistency conditions.
We prove some short lemmas about thin models.

\begin{defn}
	A path $p$ in a dimer model $Q$ is \textit{minimal} if we may not write $[p]=[qf^m]$ for any $m\geq1$.
\end{defn}

\begin{defn}\label{defn:algebraic-consistency}
	A dimer model $Q=(Q_0,Q_1,Q_2)$ in a disk is \textit{thin} if it satisfies the following \textit{thinness condition}:
	 For any vertices $v_1$ and $v_2$ of $Q$, there is a minimal path $r$ from $v_1$ to $v_2$ with the property that any other path $p$ from $v_1$ to $v_2$ satisfies
	 \[[p]=[rf^{C(p)}]\]
	 for some unique integer $C(p)\geq0$, which we call the \textit{c-value} of $p$.
\end{defn}

In particular, there is a unique equivalence class of minimal paths between any two vertices of a thin dimer model.

\begin{remk}
	In~\cite{BS}, consistent dimer models on arbitrary surfaces are defined by requiring that there is a unique minimal path-equivalence class in every homotopy class, and that an arbitrary path is equivalent to a composition of a minimal path with some number of face-paths. A thin dimer model is then a consistent dimer model in a disk.
\end{remk}

\begin{lemma}\label{lem:c-values-compose}
	If $p$ and $q$ are paths in a thin dimer model $Q$ with $h(p)=t(q)$, then $C(pq)\geq C(p)+C(q)$.
\end{lemma}
\begin{proof}
    If $p$ and $q$ are paths in a thin dimer model $Q$ with $h(p)=t(q)$, then we may write $[p]=[f^{C(p)}r_p]$ and $[q]=[f^{C(q)}r_q]$ for some minimal paths $r_p$ and $r_q$. Then, using the fact that $[f]$ is central, we calculate
    \[[pq]=[f^{C(p)}r_pf^{C(q)}r_q]=[f^{C(p)+C(q)}r_pr_q].\]
	We have shown that $[f^{C(p)+C(q)}]$ may be factored out of $[pq]$, hence $C(pq)\geq C(p)+C(q)$.
\end{proof}

Recall that two paths are equivalent if and only if there is a sequence of basic morphs taking one to the other. Since a basic morph cannot remove some arrows without replacing them with other arrows, the constant path is the unique minimal path from a vertex to itself. This leads to the following remark.

\begin{remk}\label{rem:cycle-cant-be-constant}
	If $Q$ is thin and $p$ is a nonconstant cycle, then the c-value $C(p)>0$.
\end{remk}

\begin{lemma}\label{lem:subpath-of-facepath-is-thin}
    Let $Q$ be a thin dimer model. Any proper subpath of a face-path of $Q$ is minimal.
\end{lemma}
\begin{proof}
	Suppose $p$ is a proper subpath of a face-path $f_v$ starting at $v:=t(p)$. Let $p'$ be the subpath of $f_v$ such that $pp'=f_v$. If $p$ is not minimal, then by definition of thinness, $[p]=[rf^{C(p)}]$ for some minimal path $r$ from $v$ to $h(p)$ and some integer $C(p)>0$.
	Then $[pp']=[f^{C(p)}rp']=[f_v^{C(p)}rp']$. The path $rp'$ is a nonconstant cycle, hence has a positive c-value $C(rp')>0$ by Remark~\ref{rem:cycle-cant-be-constant}. By definition of thinness $[rp']=[f_v^{C(rp')}]$. It follows that 
	\[[f_v]=[pp']=[f_v^{C(p)}rp']=[f_v^{C(p)}f_v^{C(rp')}]=[f_v^{C(p)+C(rp')}],\] which is a contradiction since $C(p)+C(rp')\geq1+1=2$ but all face-paths trivially have a c-value of 1. It follows that $p$ is minimal.
\end{proof}

\subsection{Cancellativity}

We now see that a dimer model in a disk is thin if and only its dimer algebra is cancellative.
We define the cancellative closure of an ideal, which is used in the statement of our main theorem.

\begin{defn}
	Let $A=\k Q/I$ be a path algebra with relations. We say that $A$ is a \emph{cancellation algebra}, or that $A$ (or $I$) is \emph{cancellative}, if the following property is satisfied:
	For any choice of $m\in\mathbb Z_{>0}$ and paths $a,b,\{p_i\ :\ i\in[m]\}$ and weights $\{c_i\in\k\ :\ i\in[m]\}$ satisfying $h(a)=t(p_i)$ and $h(p_i)=t(b)$ for all $i\in[m]$, we have $[a\left(\sum_{i=1}^mc_ip_i\right)b]\in I\iff[\sum_{i=1}^mc_ip_i]\in I$.
	We call this the \emph{cancellation property}.
\end{defn}

When $A=\k Q/I$ is a dimer algebra, the ideal $I$ is generated by \emph{commutation relations} of the form $[p]-[q]$, where $p$ and $q$ are paths of $Q$ with $t(p)=t(q)$ and $h(p)=h(q)$. In this case, the condition of being cancellative amounts to the condition that, 
for paths $p,q,a,b$ of $Q$ with $h(a)=t(p)=t(q)$ and $t(b)=h(p)=h(q)$, we have $[apb]-[aqb]\in I\iff [p]-[q]\in I$.
It is not hard to see that the thinness condition implies cancellativity; in fact, the reverse implication also holds.

\begin{thm}[{\cite[Theorem A]{BS}}]\label{thm:cancellative-closure}
	A dimer model $Q$ is thin if and only if $A_Q$ is cancellative. 
\end{thm}

\begin{defn}\label{def:cancellative-closure}
    Let $J$ be a subset of a path algebra $\k Q$. The \textit{cancellative closure} $J^\bullet$ of $J$ is defined to be the intersection of all cancellative ideals containing $J$:
	\[J^\bullet=\bigcap_{\substack{I\supseteq J \textup{ a cancellative} \\ \textup{ideal of }\k Q}}I.\]
\end{defn}

\begin{lemma}
	If $J\subseteq\k Q$, then $J^\bullet$ is a cancellative ideal.
\end{lemma}
Hence, the cancellative closure $J^\bullet$ is the smallest cancellative ideal containing $J$.
\begin{proof}
	Since $J^\bullet$ is an intersection of ideals, it is itself an ideal.
	We now show that it is cancellative. 
	Take $m\in\mathbb Z_{>0}$ and paths $a,b,\{p_i\ :\ i\in[m]\}$ and weights $\{c_i\in\k\ :\ i\in[m]\}$ satisfying $h(a)=t(p_i)$ and $h(p_i)=t(b)$ for all $i\in[m]$.
	Then
	\begin{align*}
		[a(\sum_{i=1}^mc_ip_i)b]\in J^\bullet&\iff  [a(\sum_{i=1}^mc_ip_i)b]\in I\text{ for any cancellative ideal }I\supseteq J \\
		&\iff [\sum_{i=1}^mc_ip_i]\in I\text{ for any cancellative ideal }I\supseteq J\\
		&\iff [\sum_{i=1}^mc_ip_i]\in J^\bullet.
	\end{align*}
	This shows that $J^\bullet$ is cancellative.
\end{proof}

It is immediate that if $J$ is a cancellative ideal, then $J^\bullet=J$.

\subsection{Strand Diagrams and Bad Configurations}

We define strand diagrams and connect them to dimer models. We show that a dimer model is thin if and only if its strand diagram avoids certain bad configurations. The below definition is a reformulation of~\cite[Definition 1.10]{ZBocklandt2015}.

\begin{defn}\label{defn:postnikov-diagram}
	Let $\Sigma$ be a disk with a discrete set of $n$ marked points on its boundary. A \textit{strand diagram} $D$ in $\Sigma$ consists of a finite collection of oriented strands subject to the following conditions.
    \begin{enumerate}
        {\item Each boundary marked point is the start point of exactly one strand, and the end point of exactly one strand.\label{pd:1}}
        {\item Any two strands intersect in finitely many points, and each intersection involves only two strands. Each intersection not at a marked boundary point is transversal.\label{pd:2}}
	{\item\label{pd:3} Moving along a strand, the signs of its crossings with other strands alternate. This includes intersections at a marked boundary point. We call this the \emph{alternating intersection property}. See the figure below, where the bold segment is boundary.
		\begin{figure}[H]			\centering
			\def\svgscale{.21}
			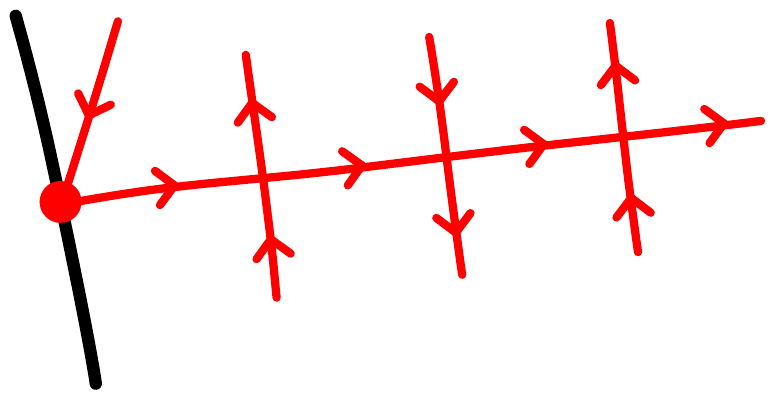
		\end{figure}}
	{\item\label{dpd:4} 
		Any connected component $C$ of the complement of $D$ in the interior of $\Sigma$ is an open disk. The boundary of $C$ may contain a number of one-dimensional ``boundary segments'' of the boundary of $\Sigma$, in which case $C$ is a \textit{boundary region}. Otherwise, $C$ is an \textit{internal region}.
It follows from~\eqref{pd:3} that any internal region $C$ is either an \textit{oriented region} (i.e., all strands on the boundary of the component are oriented in the same direction) or an \textit{alternating region} (i.e., the strands on the boundary of the component alternate directions). See Figure~\ref{fig:oriented-alternating}.
		    {It follows from the above conditions that adjacent strands on the boundary of a boundary region must alternate in direction. We consider all boundary regions to be alternating, even if they have only one strand on the boundary.}

		\begin{figure}			\centering
			\def\svgscale{0.21}
			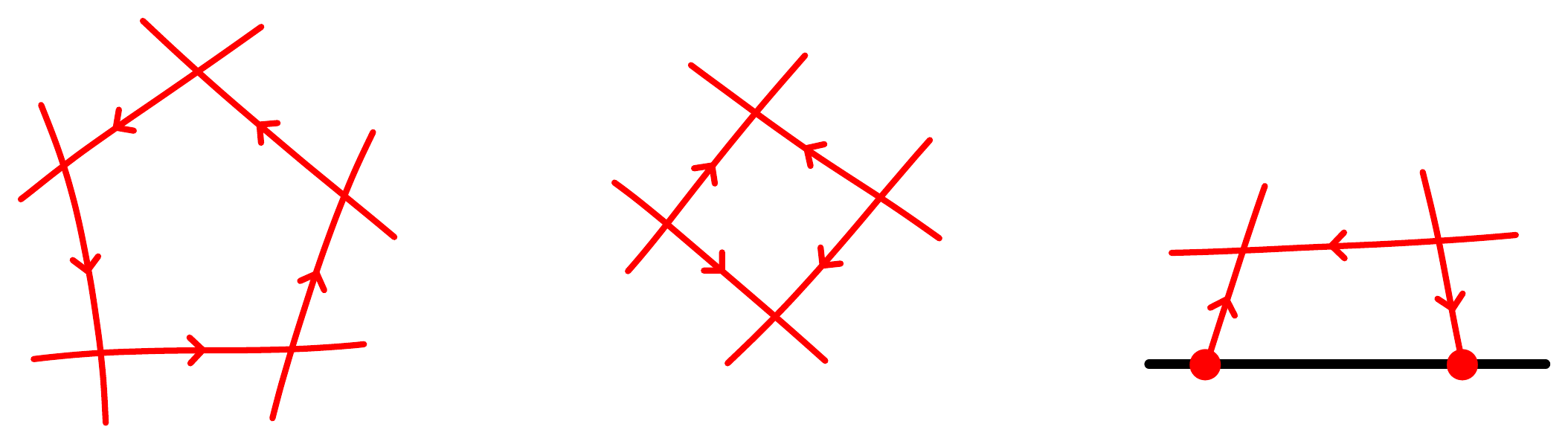
			\caption{One oriented (left) and two alternating (right) regions. The bold segment is boundary.}
			\label{fig:oriented-alternating}
		\end{figure}}
    \end{enumerate}
    The diagram $D$ is called a \textit{Postnikov diagram} if in addition it satisfies the following conditions.
    \begin{enumerate}
        {\item No subpath of a strand is a closed cycle.\label{pd:6nointeriorcycles}}
	{\item\label{spd:2} 
		No two strand segments intersect twice and are oriented in the same direction between these intersection points.\label{pd:5badlens}}
    \end{enumerate}
	In other words, \textit{bad configurations} shown in Figure~\ref{fig:bad-configurations} and described below must not appear in order for $D$ to be a Postnikov diagram:
	\begin{figure}		\centering
		\def\svgscale{0.21}
		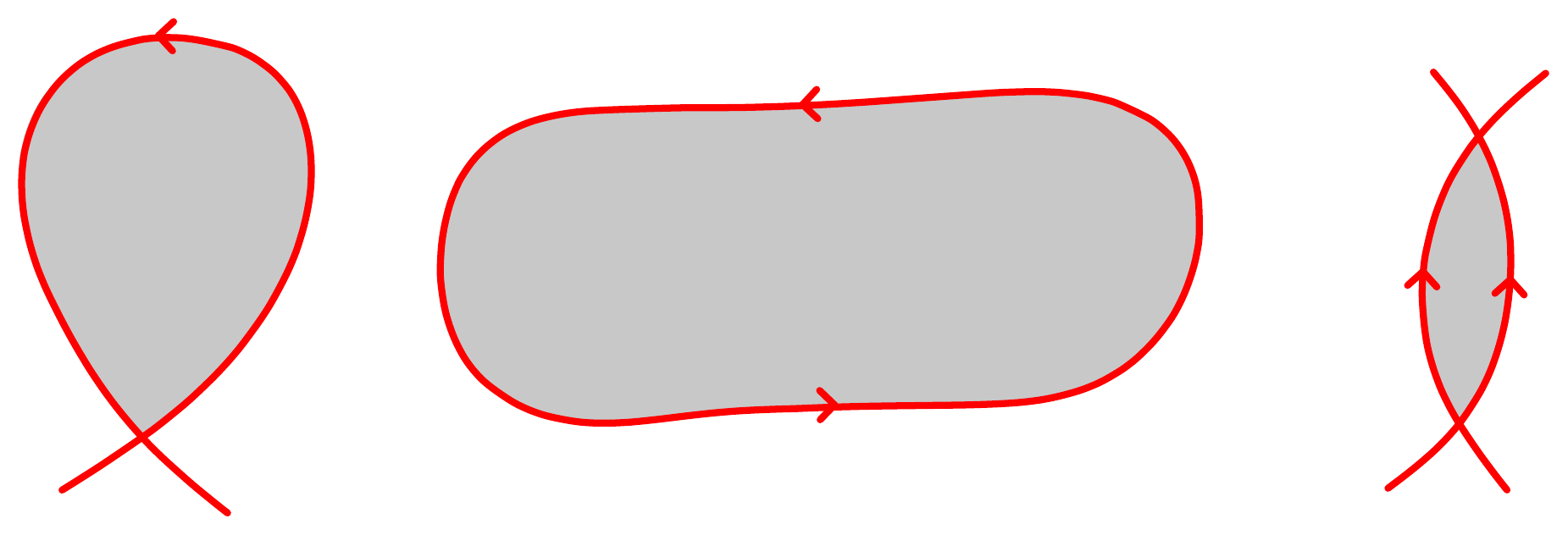
		\caption{The three bad configurations.}
		\label{fig:bad-configurations}
	\end{figure}
	\begin{enumerate}
		\item\label{gg1} A strand which intersects itself through a cycle as forbidden in~\eqref{pd:6nointeriorcycles}, called a \textit{self-intersecting strand}. 
		\item\label{gg2} A strand in the interior as forbidden in~\eqref{pd:6nointeriorcycles}, called a \textit{closed cycle}. 
		\item\label{gg3} Two strand segments which intersect in the same direction as forbidden in~\eqref{pd:5badlens}, called a \textit{bad lens}.
	\end{enumerate}
	The diagram $D$ is \emph{connected} if the set of strands is connected. Equivalently, $D$ is connected if every boundary region has exactly one boundary segment of the disk $\Sigma$.
\end{defn}

See the left of Figure~\ref{fig:ex2} for an example of a strand diagram (overlayed onto a dimer model).
A \emph{$(k,n)$-Postnikov diagram} is a Postnikov diagram such that the strand starting at vertex $j$ ends at vertex $j-k$ for all $j\in[n]$.

\begin{defn}~\label{defn:dimer-model-from-strand-diagram}
	Let $D$ be a connected strand diagram. We associate to $D$ a dimer model $Q_D$ on a disk as follows. The vertices of $Q_D$ are the alternating regions of $D$. When the closures of two different alternating regions $v_1$ and $v_2$ meet in a crossing point between strands of $D$, or at one of the marked boundary points, we draw an arrow between $v_1$ and $v_2$, oriented in a way consistent with these strands, as shown in Figure~\ref{fig:strand-arrows}.
	The counter-clockwise (respectively clockwise) faces of $Q_D$ are the arrows around a counter-clockwise (respectively clockwise) region of $D$.
	\begin{figure}		\centering
		\def\svgscale{.21}
		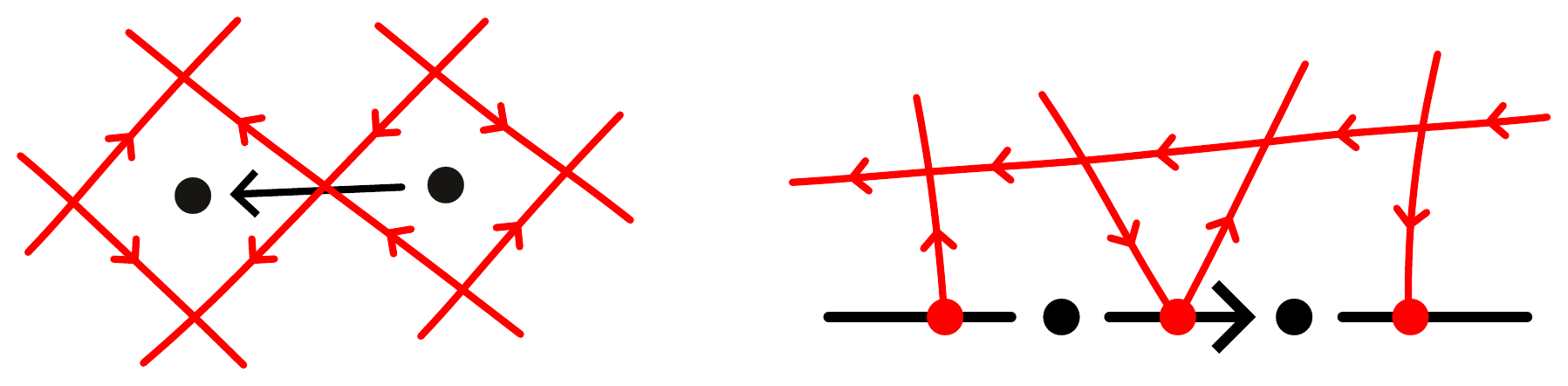
		\caption{The arrows between two alternating faces. The bold arrow is a boundary arrow.}
		\label{fig:strand-arrows}
	\end{figure}
	The result is a thin dimer model in a disk~\cite[Proposition 2.11]{ZPressland2019}.
\end{defn}
We may also go in the other direction.
\begin{defn}\label{defn:strand-diagram-from-dimer-model}
	Let $Q$ be a dimer model. We associate a connected strand diagram $D_Q$ to $Q$ as follows.
	For any arrow $\alpha$ of $Q$, let $v_\alpha$ be the point in the center of $\alpha$ in the embedding of $Q$ into $S(Q)$.
	For any two arrows $\alpha$ and $\beta$ of $Q$ such that $\beta\alpha$ is a subpath of a face-path, we draw a path from $v_\alpha$ to $v_\beta$ along the interior of the face containing $\beta\alpha$.
	Label the endpoints of strands along the boundary of $D$ from 1 to $n$ such that the $j$th point lies along the blossom arrow immediately clockwise of the $j$th boundary vertex of $D$.
	See the left of Figure~\ref{fig:ex2} for a dimer model with its strand diagram overlayed.
\end{defn}

The above constructions are mutual inverses, and hence establish a correspondence between connected strand diagrams and dimer models. 
The following was shown for dimer models on the disk corresponding to $(k,n)$-diagrams in~\cite{ZBKM} and for general dimer models on the disk in~\cite{BS}.
\begin{thm}[{\cite{BS}}]\label{thm:thin-bad-conf}
	A dimer model in a disk is thin if and only if its strand diagram has no bad configurations (i.e., if it is a $(k,n)$-Postnikov diagram).
\end{thm}

\subsection{Decorated Permutations}

We now connect dimer models with positroids through decorated permutations.
A \emph{decorated permutation} is a permutation $\pi$ on $[n]$ along with a coloring of the fixed points in two colors. If $i$ is a fixed point of $\pi$, we write $\pi(i)=\overline i$ or $\pi(i)=\underline i$ to denote this coloring. Decorated permutations are a positroid cryptomorphism. In other words, a decorated permutation determines a positroid, and vice versa.
The positroid associated to a decorated permutation $\pi$ is \emph{connected} if $[n]$ cannot be broken into two cyclically contiguous subsets $S$ and $T$ such that $\pi(S)\subseteq S$ and $\pi(T)\subseteq T$. We also say that $\pi$ is \emph{connected} in this case; such permutations are also called \emph{stable-interval-free} in the literature.
In~\cite{ZPostnikov}, strand diagrams were related to decorated permutations and it was shown that up to certain moves, a strand diagram is determined by its decorated permutation.

\begin{defn}
	Let $D$ be a strand diagram with no bad configurations. For a marked boundary point $i$ of $D$, let $z_i$ be the strand beginning at $i$. The \emph{decorated permutation $\pi_D$ of $D$} is the decorated permutation of $[n]$ such that if $i\neq h(z_i)$ then
	$\pi(i)=h(z_i)$, if $i=h(z_i)$ and $z_i$ winds clockwise then $\pi(z_i)=\overline{h(z_i)}$, and if $i=h(z_i)$ and $z_i$ winds counter-clockwise then $\pi(i)=\underline{h(z_i)}$.
	If $Q$ is a thin dimer model, then the \textit{decorated permutation $\pi_Q$ of $Q$} is the decorated permutation of its associated strand diagram.
\end{defn}

The permutation of the dimer model in the left of Figure~\ref{fig:ex2} is $2 5 6 1 3 4$ in one-line notation.

The following lemma is implicit in many works surrounding strand diagrams, but the authors are unable to find a direct proof. We include one here for the convenience of the reader.

\begin{lemma}\label{lem:con}
	Let $D$ be a strand diagram with no bad configurations. Then the decorated permutation $\pi_D$ of $D$ is connected if and only if $D$ is connected.
\end{lemma}
\begin{proof}
	If $D$ is not connected, then it is immediate that $\pi_D$ is not connected. We must show the reverse implication.
	\begin{figure}[H]
		\centering
\def\svgscale{0.21}
\begingroup%
  \makeatletter%
  \providecommand\color[2][]{%
    \errmessage{(Inkscape) Color is used for the text in Inkscape, but the package 'color.sty' is not loaded}%
    \renewcommand\color[2][]{}%
  }%
  \providecommand\transparent[1]{%
    \errmessage{(Inkscape) Transparency is used (non-zero) for the text in Inkscape, but the package 'transparent.sty' is not loaded}%
    \renewcommand\transparent[1]{}%
  }%
  \providecommand\rotatebox[2]{#2}%
  \newcommand*\fsize{\dimexpr\f@size pt\relax}%
  \newcommand*\lineheight[1]{\fontsize{\fsize}{#1\fsize}\selectfont}%
  \ifx\svgwidth\undefined%
    \setlength{\unitlength}{1034.51599121bp}%
    \ifx\svgscale\undefined%
      \relax%
    \else%
      \setlength{\unitlength}{\unitlength * \real{\svgscale}}%
    \fi%
  \else%
    \setlength{\unitlength}{\svgwidth}%
  \fi%
  \global\let\svgwidth\undefined%
  \global\let\svgscale\undefined%
  \makeatother%
  \begin{picture}(1,0.40260471)%
    \lineheight{1}%
    \setlength\tabcolsep{0pt}%
    \put(0,0){\includegraphics[width=\unitlength,page=1]{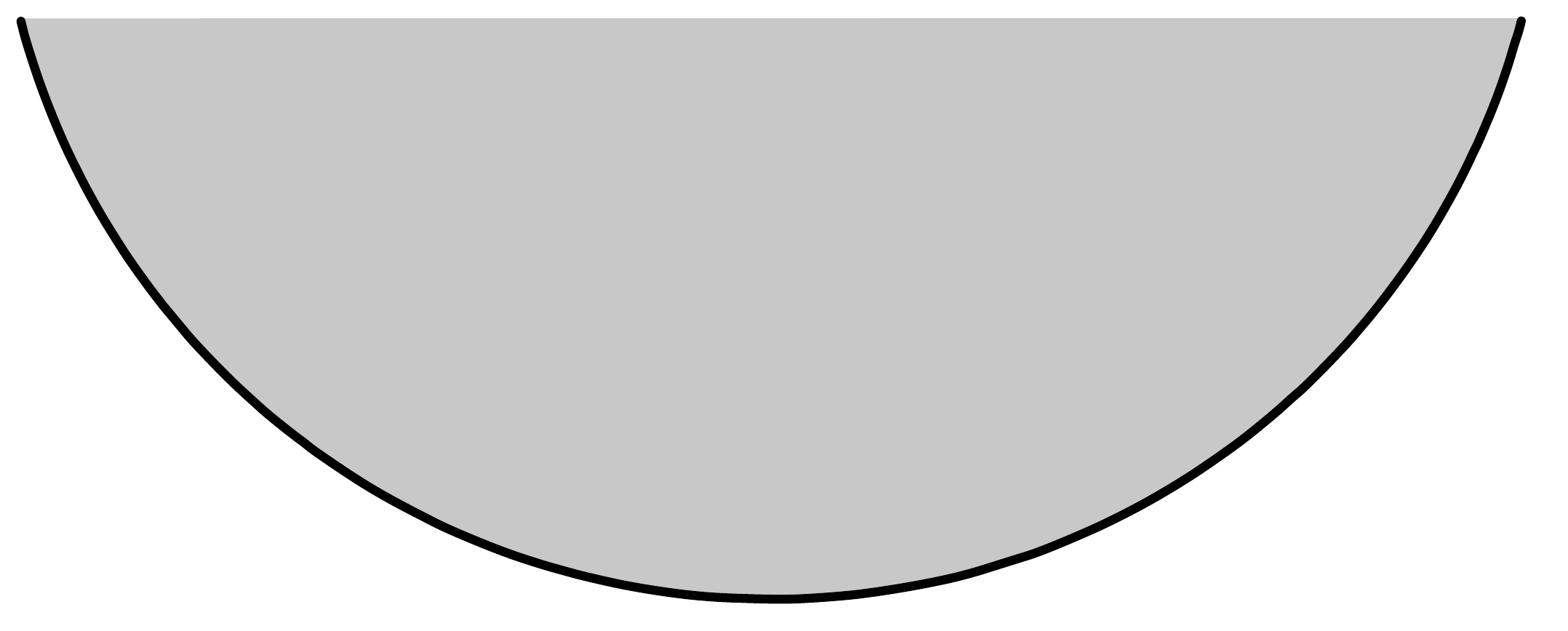}}%
    \put(0.86882659,0.13626151){\color[rgb]{0.09019608,0.08627451,0.07058824}\makebox(0,0)[lt]{\lineheight{1.25}\smash{\begin{tabular}[t]{l}$j$\end{tabular}}}}%
    \put(0.93413056,0.22510731){\color[rgb]{0.09019608,0.08627451,0.07058824}\makebox(0,0)[lt]{\lineheight{1.25}\smash{\begin{tabular}[t]{l}$j-1$\end{tabular}}}}%
    \put(0.71848479,0.30488606){\color[rgb]{0.09019608,0.08627451,0.07058824}\makebox(0,0)[lt]{\lineheight{1.25}\smash{\begin{tabular}[t]{l}$z_1$\end{tabular}}}}%
    \put(0.42079098,0.36555608){\color[rgb]{0.09019608,0.08627451,0.07058824}\makebox(0,0)[lt]{\lineheight{1.25}\smash{\begin{tabular}[t]{l}$z_2$\end{tabular}}}}%
    \put(0.26394855,0.33173388){\color[rgb]{0.09019608,0.08627451,0.07058824}\makebox(0,0)[lt]{\lineheight{1.25}\smash{\begin{tabular}[t]{l}$z_3$\end{tabular}}}}%
    \put(0.17687402,0.37103651){\color[rgb]{0.09019608,0.08627451,0.07058824}\makebox(0,0)[lt]{\lineheight{1.25}\smash{\begin{tabular}[t]{l}$w$\end{tabular}}}}%
    \put(0,0){\includegraphics[width=\unitlength,page=2]{connected.pdf}}%
  \end{picture}%
\endgroup%

		\caption{Proof of Lemma~\ref{lem:con} with $m=3$.}
		\label{fig:DLA}
	\end{figure}

	Suppose $\pi$ is not a connected positroid. Then there are nonempty cyclically contiguous subsets $A$ and $B$ of $[n]$ whose disjoint union is $[n]$ such that $\pi(A)=A$ and $\pi(B)=B$. Choose these subsets such that $B$ is a connected component of $\pi$ (i.e., there is no $B'\subsetneq B$ such that $\pi(B')=B$). Say $j\in B$ and $j-1\in A$. We say that a strand of $D$ is a strand of $A$ if it starts and ends in $A$, and otherwise is a strand of $B$. Suppose that the two strands at vertex $j$ are oriented towards $j-1$ as in Figure~\ref{fig:DLA}; the other case is symmetric.

	Let $z_1$ be the strand coming out of vertex $j$. If the first strand to intersect $z_1$ after the boundary is a strand of $B$, then let this strand be $z_2$ and let $z'_1$ be the substrand of $z_1$ from its start to this intersection. If the first strand to intersect $z_2$ before this intersection with $z_1$ is a strand of $B$, then let this strand be $z_m$ and let $z'_2$ be the substrand of $z_2$ between these two intersection points. Continue on like this; suppose that at some point we get some strand $z_m$ such that the first strand to intersect $z_m$ after its intersection with $z_{m-1}$, if $m$ is odd, or before its intersection with $z_{m-1}$, if $m$ is even, is a strand $w$ of $A$. Let $z'_m$ be the substrand of $z_m$ between its intersection with $z'_{m-1}$ and $w$, and let $z''_m$ be the substrand of $z_m$ after its intersection with $w$.

	Suppose that $m$ is odd, as in Figure~\ref{fig:DLA}; the even case is similar.
	Note that the substrands $z'_1,\dots,z'_m,z''_m$ form a segment which cuts out a smaller disk $\Sigma'$ containing the boundary vertices between $j,j+1,\dots,h(z_m)$.
	By the alternating intersection property, $w$ enters $\Sigma'$ through its intersection with $z'_m$. Since $w$ is a strand of $A$, it must leave this area after its intersection with $z'_m$. It cannot do this through an intersection with $z''_m$ without creating a bad lens and it cannot do this through an intersection with $z'_1,\dots,z'_m$ without contradicting our choice of $w$. It follows that no such $w$ may exist.

	Then this process of choosing $z_i$'s must terminate through the choice of some $z'_m$ which reaches the boundary of the disk. By the alternating intersection property, the substrands $z'_1,\dots,z'_m$ form a barrier which separates the strands of $B$ from the strands of $A$, hence no strand of $B$ intersects any strand of $A$. This shows that $D$ is disconnected.
\end{proof}

In~\cite[Corollary 14.7]{ZPostnikov}, it was shown that for any decorated permutation $\pi$ there exists a strand diagram $D_\pi$ with no bad configurations whose decorated permutation is $\pi$.
Since connected strand diagrams with no bad configurations are in bijection with thin dimer models by Theorem~\ref{thm:thin-bad-conf}, we obtain the following lemma.

\begin{lemma}\label{lem:perm-to-dimer}
	Let $\pi$ be a connected decorated permutation. There exists a thin dimer model $Q_\pi$ whose decorated permutation is $\pi$.
\end{lemma}

The following result was proven first for $(k,n)$-Postnikov diagrams in~\cite{ZBKM}, and for general dimer models on disks in~\cite{ZPressland2015}. It is proven by showing that certain moves on strand diagrams preserve the boundary algebra.
While these sources do not explicitly state that the isomorphism preserves the idempotents corresponding to boundary vertices, this is clear by the proof.
\begin{prop}[{\cite[Corollary 4.7]{ZPressland2015}}]\label{prop:ba-from-perm}
	If $Q$ and $Q'$ are thin dimer models such that $\pi_Q=\pi_{Q'}$, then there is an isomorphism $\phi:B_Q\to B_{Q'}$ which, for any boundary vertex $j$, sends the idempotent $e_j\in B_Q$ to the idempotent $e_j\in B_{Q'}$.
\end{prop}

In light of Proposition~\ref{prop:ba-from-perm}, one may define the \emph{boundary algebra $B^\pi$ of a connected positroid $\pi$} to be the algebra $B_Q$ where $Q$ is any thin dimer model with decorated permutation $\pi$. The goal of this paper is to describe $B^\pi$ in terms of $\pi$.

\subsection{Dimer Submodels}

We now define dimer submodels, allowing us to pass to a smaller disk within a dimer model.

\begin{defn}\label{defn:disk-submodel}
        Let $Q=(Q_0,Q_1,Q_2)$ be a dimer model in a disk. A \textit{dimer submodel} of $Q$ is a dimer model in a disk $Q'=(Q'_0,Q'_1,Q'_2)$ such that $Q'_0\subseteq Q_0$ and $(Q'_1,Q'_2)$ are all of the edges and faces in $(Q_1,Q_2)$ containing only vertices in $Q'_0$.
\end{defn}

If $Q$ is a dimer model, then there is an embedding of $Q$ into a disk. By choosing a set of faces of $Q$ in this embedding which themselves make up a disk, we get a dimer submodel of $Q$. All dimer submodels of $Q$ are obtained in this way.

The following result is a special case of~\cite[Theorem 5.3]{BS}. We provide a proof here for the convenience of the reader.

\begin{thm}[{\cite[Theorem 5.3]{BS}}]\label{thm:submodel}
	If $Q$ is thin, then any dimer submodel $Q'$ of $Q$ is thin. Moreover, paths $p$ and $q$ of $Q'$ are equivalent in $Q'$ if and only if they are equivalent as paths of $Q$.
\end{thm}
\begin{proof}
	By Theorem~\ref{thm:thin-bad-conf}, 
	thinness is equivalent to the absence of bad configurations of the associated strand diagram. Restricting a dimer model to a submodel simply restricts the strand diagram, and hence may not create bad configurations. This shows the first statement.
	We now address the latter statement.
	Take paths $p$ and $q$ of $Q$ with the same start and end vertices. Without loss of generality we may suppose that $[p]=[qf^m]$ for some $m\geq0$ in $Q'$. Then there is a sequence of basic morphs taking $p$ to $qf^m$ in $Q'$; this is also a sequence of basic morphs in $Q$, so $[p]=[qf^m]$ in $Q$ as well.
	Then $[p]=[q]$ in $Q'$ if and only if $m=0$ if and only if $[p]=[q]$ in $Q$.
\end{proof}

\section{Arrow-Defining Paths and Relation Numbers}\label{sec:section-thin-finite-dimer-model}
\label{sec:TDM}

We now prove some lemmas about thin dimer models using the algebraic thinness condition. 
We define and give a preliminary characterization for the paths of the dimer algebra which give rise to arrows of the boundary algebra.
For any minimal path $p$, we define \emph{relation numbers} comparing it to the paths from $t(p)$ to $h(p)$ which follow along the boundary of the dimer model.

\subsection{Arrow-Defining Paths and Direct Paths}

In the following, let $Q$ be a thin dimer model.
Label the boundary vertices of $Q$ from 1 to $n$ in clockwise order. See the dimer model on the left of Figure~\ref{fig:ex2}. For any $i$, there is a unique boundary arrow $\alpha_i$ between vertices $i$ and $i+1$. Moreover, the arrow $\alpha_i$ has a unique return path $R_{\alpha_i}$. If $\alpha_i:i\to i+1$ then we set $x_i:=\alpha_i$ and $y_i:=R_{\alpha_i}$. If $\alpha_i:i+1\to i$ then we set $y_i:=\alpha_i$ and $x_i:=R_{\alpha_i}$. We set $x:=\sum_{i=1}^nx_i$ and $y:=\sum_{i=1}^ny_j$.
For convenience, we introduce the notation $x_i^m:=x_ix_{i+1}\dots x_{i+m-1}$ and $y_i^m:=y_{i}y_{i-1}\dots y_{i-m+1}$. Then $[x_i^m]=[e_ix^m]$ and $[y_{i-1}^m]=[e_{i}y^m]$.
We always consider indices of boundary vertices modulo $n$.

\begin{lemma}\label{x-y-paths-thin}
	The following facts hold for all $i\in[n]$.
	\begin{enumerate}
		\item\label{61} The compositions $x_iy_i$ and $y_ix_i$ are face-paths of $Q$.
		\item\label{60} The paths $x_i$ and $y_i$ are minimal. 
		\item\label{62} There is an equality $[xy]=[yx]$ in $A_Q$.
	\end{enumerate}
\end{lemma}
\begin{proof}
	It is clear by definition that $x_iy_i$ and $y_{i}x_{i}$ are face-paths of $Q$, showing~\eqref{61}. Then by Lemma~\ref{lem:subpath-of-facepath-is-thin}, $x_i$ and $y_i$ are minimal for all $i$, showing~\eqref{60}. Moreover, $x_iy_i$ and $y_{i-1}x_{i-1}$ are face-paths of $Q$ starting at the same vertex $e_i$, hence $[x_iy_i]=[y_{i-1}x_{i-1}]$ for all $i\in[n]$ by Remark~\ref{faces-are-equivalent}, showing~\eqref{62}.
\end{proof}

\begin{cor}\label{xy-is-c-on-boundary}
	Let $p$ be a boundary path and let $r$ be a minimal path from $t(p)$ to $h(p)$. Then $[p]=[(xy)^{C(p)}r]$, where $C(p)$ is the c-value of $p$.
\end{cor}
\begin{proof}
	This is the thinness property of Definition~\ref{defn:algebraic-consistency} where the face-path of $f$ starting at any boundary vertex $i$ is chosen to be $x_iy_i$, which is a face-path by Lemma~\ref{x-y-paths-thin}~\eqref{61}.
\end{proof}

\begin{defn}
	Let $p$ be a cycleless boundary path in $Q$. Then $p$ subdivides $Q$ into a left side and a right side. Let $L(p)\subseteq Q_2$ consist of the set of faces which are to the left of $p$. Let $R(p)\subseteq Q_2$ consist of the set of faces which are to the right of $p$.
Note that if $q$ is a left-morph of $p$ and both are cycleless, then $L(q)\subsetneq L(p)$.
	For any cycleless boundary paths $p$ and $q$ between the same vertices, we say that $q$ is \emph{to the left of $p$} (or $p$ is \emph{to the right of $q$}) if $L(q)\subseteq L(p)$ (equivalently, $R(q)\supseteq R(p)$).
\end{defn}

The following two results appear in~\cite{BS}, but we prove them again here to be as self-contained as possible.

\begin{prop}\label{thm:leftmost-c-valueZ}
	Let $p$ be any minimal boundary path and let $q$ be any rightmost path from $t(p)$ to $h(p)$. Then $p$ is to the left of $q$.
\end{prop}
\begin{proof}
	Suppose to the contrary that $p$ is not to the left of $q$. 
		Then we may choose subpaths $p'$ and $q'$ of $p$ and $q$, respectively, such that $p'(q')^{-1}$ is a simple counter-clockwise loop (as a path in the disk $\Sigma$). Since $p$ is minimal, its subpath $p'$ must also be minimal.
	Let $Q'$ be the dimer submodel given by restricting $Q$ to the area bounded by $p'(q')^{-1}$.
	The submodel $Q'$ is thin by Theorem~\ref{thm:submodel}. The path $q$ is rightmost in $Q$ so its subpath $q'$ is rightmost in $Q'$. Since $q'$ winds clockwise around the boundary of $Q'$, it cannot have any left-morphs, hence $q'$ is the only path in its equivalence class in $Q'$. In particular, a face-path cannot be factored out of $q'$, so $q'$ is minimal in $Q'$ and $[p']=[q'f^m]$ in $Q'$ for some $m>0$. By Theorem~\ref{thm:submodel}, $[p']=[q'f^m]$ in $Q$, hence $p'$ is not minimal in $Q$, a contradiction.
\end{proof}

\begin{cor}\label{cor:leftmost-unique}
	Let $v_1$ and $v_2$ be vertices of a thin dimer model $Q$. There is a unique minimal rightmost path from $v_1$ to $v_2$, and a unique minimal leftmost path from $v_1$ to $v_2$.
\end{cor}
\begin{proof}
	We prove the existence of a unique minimal rightmost path; the leftmost result is dual.
	Let $p$ be any minimal path from $v_1$ to $v_2$. We repeatedly take basic right-morphs to get a sequence $p=p_0,p_1,\dots$ such that $p_{i+1}$ is a basic right-morph of $p_i$. This process must terminate in a rightmost path $p_m$ since $R(p_{i+1})\subsetneq R(p_i)$ holds at each step.
	Moreover, if $q$ is any rightmost minimal boundary path from $v_1$ to $v_2$, then Proposition~\ref{thm:leftmost-c-valueZ} shows that $p_m$ is to the left of $q$ and, separately, that $q$ is to the left of $p_m$, hence we must have $p_m=q$. Then $p_m$ is the \emph{unique} minimal rightmost path from $v_1$ to $v_2$.
\end{proof}

\begin{defn}\label{def:arrow-defining-paths}\label{defn:right-left-pair}
	A cycleless boundary path $p$ of $q$ is
    \begin{itemize}
	    \item[(a)] 
 \textit{direct} if it passes through no boundary vertices other than its start and end vertices, and it is
        \item[(b)] \textit{arrow-defining} if every path in its equivalence class is direct.
    \end{itemize}
\end{defn}

It is immediate that any arrow-defining path is minimal, since for any boundary vertex $i$ the face-path $x_iy_i$ factors through $i+1$ and $y_{i-1}x_{i-1}$ factors through $i-1$. We now show that the (path-equivalence classes of) arrow-defining paths are precisely those which give rise to arrows in the Gabriel quiver of the boundary algebra. 

\begin{lemma}\label{lem:arrow-defining-defines-arrow}
	Let $Q$ be a thin dimer model. There is an arrow from $v_1$ to $v_2$ in the boundary algebra $B_Q$ if and only if the minimal path from $v_1$ to $v_2$ in $A_Q$ is arrow-defining.
\end{lemma}
\begin{proof}
	The path $p$ gives an arrow of the boundary algebra if and only if no equivalent path factors through any boundary vertices other than its start and end vertices. This is precisely the definition of an arrow-defining path.
\end{proof}

We now relate arrow-defining paths and minimal leftmost and rightmost paths.

\begin{prop}\label{thm:direct-pair-iff-achinko}
    Let $v_1$ and $v_2$ be distinct boundary vertices of a thin dimer model $Q$. The following are equivalent.
	\begin{enumerate}
		\item\label{h1} There is an arrow-defining path from $v_1$ to $v_2$.
		\item\label{h2} The minimal leftmost and rightmost paths from $v_1$ to $v_2$ are both direct.
	\end{enumerate}
\end{prop}
\begin{proof}
	If $p$ is an arrow-defining path from $v_1$ to $v_2$, then $p$ is minimal and every path in its equivalence class is direct. In particular, the minimal leftmost and rightmost paths are direct, showing that~\eqref{h1}$\implies$~\eqref{h2}. On the other hand, suppose that the minimal leftmost path $p_l$ and the minimal rightmost path $p_r$ are both direct and let $p$ be any minimal path from $v_1$ to $v_2$. Proposition~\ref{thm:leftmost-c-valueZ} shows that $p$ is to the left of $p_r$, and the dual of Proposition~\ref{thm:leftmost-c-valueZ} shows that $p$ is to the right of $p_l$.
	Hence, the path $p$ lies in the area bounded by $p_l$ and $p_r$. In particular, $p$ does not contain any boundary vertices aside from $v_1$ and $v_2$, so $p$ is direct. Since this holds for any minimal path from $v_1$ to $v_2$, we have shown that any minimal path from $v_1$ to $v_2$ is arrow-defining. This ends the proof.
\end{proof}

A priori, checking if a path is arrow-defining requires computing every vertex through which a minimal path may factor; Proposition~\ref{thm:direct-pair-iff-achinko} argues that it is possible to do this by only looking at the leftmost and rightmost minimal paths. 
In the following, we will develop tools to compute rightmost and leftmost paths, which will in turn let us compute arrows of the boundary algebra $B_Q$ via Lemma~\ref{lem:arrow-defining-defines-arrow}.

\subsection{Relation Numbers of Paths}

We use the structure of the arrows $x_i$ and $y_i$ between adjacent boundary vertices to define the relation numbers of a minimal path, which compare it to a composition of $x$'s or a composition of $y$'s.

\begin{defn}\label{defn:reach}
    Let $p$ be a boundary path. Let $\reach_x(p)$ be the number in $\{0,\dots,n-1\}$ which is equivalent to $h(p)-t(p)$ modulo $n$. Similarly, let $\reach_y(p)$ be the number in $\{0,\dots,n-1\}$ which is equivalent to $t(p)-h(p)$ modulo $n$. 
\end{defn}

Note that $y_{t(p)-1}^{\reach_y(p)}$ and $x_{t(p)}^{\reach_x(p)}$ are paths from $t(p)$ to $h(p)$ and $\reach_x(p)= n-\reach_y(p)$.

\begin{defn}
	Let $p$ be a nonconstant boundary path in $Q$. Define the \emph{right relation number ${\Y}(p)$ of $p$} as ${\Y}(p):=C(y_{t(p)-1}^{\reach_y(p)})-C(p)$. If ${\Y}(p)\geq0$, then ${\Y}(p)$ is the unique integer such that $[p(xy)^{{\Y}(p)}]=[y_{t(p)-1}^{\reach_y(p)}]$; on the other hand, if ${\Y}(p)\leq0$, then ${\Y}(p)$ is the unique integer such that $[p]=[y_{t(p)-1}^{\reach_y(p)}(xy)^{-{\Y}(p)}]$. Similarly, define the \emph{left relation number ${\X}(p)$ of $p$} to be ${\X}(p):=C(x_{t(p)}^{\reach_x(p)})-C(p)$.
\end{defn}

Intuitively, ${\Y}(p)$ measures how far the path $p$ is from being equivalent to a composition of $y$'s.
When $p$ is minimal, we may move and cancel $x$'s and $y$'s from the equations
$[p(xy)^{{\Y}(p)}]=[y_{t(p)-1}^{\reach_y(p)}]$
and $[p(xy)^{{\X}(p)}]=[x_{t(p)}^{\reach_x(p)}]$
to get more relations of the dimer algebra.

\begin{prop}\label{thm:achinko-relations-from-rep}\label{thm:cyc-set}
	Suppose $p$ is a minimal boundary path. Then 
	\begin{enumerate}
		\item\label{X1} ${\Y}(p)<\reach_y(p)$ and  $[y_{t(p)-1-m}^{\reach_y(p)-{\Y}(p)}]=[x^{m}px^{{\Y}(p)-m}]$ holds in $A_Q$ for any $0\leq m\leq {\Y}(p)$, and
		\item\label{X2} ${\X}(p)<\reach_x(p)$ and 
			$[x_{t(p)+m}^{\reach_x(p)-{\X}(p)}]=[y^mpy^{{\X}(p)-m}]$ holds in $A_Q$ for any $0\leq m\leq {\X}(p)$.
	\end{enumerate}
\end{prop}
\begin{proof}
	We only show~\eqref{X1}, since~\eqref{X2} is symmetric.
	If ${\Y}(p)\geq\reach_y(p)$, then we can cancel (using that $A_Q$ is cancellative by Theorem~\ref{thm:cancellative-closure}) the $y$'s from the right side of the equation $[p(xy)^{{\Y}(p)}]=[y_{t(p)-1}^{\reach_y(p)}]$ to get $[px^{{\Y}(p)}y^{{\Y}(p)-\reach_y(p)}]=[e_{t(p)}]$, which is a contradiction as a constant path may not be equivalent to a nonconstant path. Hence, ${\Y}(p)<\reach_y(p)$.
	Fix $m$ such that $0\leq m\leq {\Y}(p)$. We may write
	\begin{align*}
		[y_{t(p)-1}^{\reach_y(p)}]=[p(xy)^{{\Y}(p)}]&=[(xy)^{m}p(xy)^{{\Y}(p)-m}]=[y^{m}x^{m}px^{{\Y}(p)-m}y^{{\Y}(p)-m}],
	\end{align*}
	where the second equality follows because $[xy]$ is in the center of $A_Q$ and the third follows because $[xy]=[yx]$. Since ${\Y}(p)<\reach_y(p)$, we may cancel out all $y$'s from the left and right sides of the right hand term. The resulting equation is $[y_{t(p)-1-m}^{\reach_y(p)-{\Y}(p)}]=[x^{m}px^{{\Y}(p)-m}]$. 
\end{proof}

We call the relations mentioned in~\eqref{X1} and~\eqref{X2} of Proposition~\ref{thm:achinko-relations-from-rep} the two \textit{cyclic sets of relations} of $p$. We wait until Example~\ref{ex:2} to give an example of these relations.

\section{Calculating Arrows and Relation Numbers}
\label{sec:carn}

We now show that the arrow-defining paths of a dimer model and their relation numbers may be calculated from the associated strand diagram, decorated permutation, or Grassmann necklace. To do this, we will need to understand rightmost paths in terms of strands.

As before, let $Q$ be a thin dimer model and number the boundary vertices of $Q$ from $1$ to $n$, increasing clockwise. We use the term \emph{strand-vertices} to refer to the start and end points of the strands of $Q$. Number the strand-vertices from $1$ to $n$ increasing clockwise, such that strand-vertex $i$ is immediately clockwise of the vertex $i\in Q$ and counter-clockwise of $i+1\in Q$.
See Figure~\ref{fig:ex2}.
	Given a boundary vertex $v$ of the quiver $Q$, we write $v^{cl}:=v$ and $v^{cc}:=v-1$ for the strand-vertices immediately clockwise and counter-clockwise of $v$, respectively.

\begin{defn}\label{defn:cyc-int}
If $i$ and $j$ are distinct elements of $[n]$, then we say that the \emph{(open) clockwise interval $(i,j)$ from $i$ to $j$} is the set $\{i+1,i+2,\dots,j-1\}$, where addition is modulo $n$. We also use the notation $[i,j)$, $(i,j]$, and $[i,j]$ for clockwise intervals which include either or both endpoints.
In particular, we frequently use this notation to refer to clockwise intervals of boundary vertices of $Q$ or strand-vertices of the strand diagram.
\end{defn}

\subsection{Understanding Rightmost Paths through Strand Diagrams}

We first obtain results about rightmost (and leftmost) paths of $Q$ from the information of the strand diagram.

\begin{defn}
	Let $v_1$ and $v_2$ be boundary vertices of $Q$.
	A \textit{clockwise strand} between $v_2$ and $v_1$ is a strand from strand-vertex $w_1$ to strand-vertex $w_2$, such that $v_2^{cl},\ w_1,\ w_2,\ v_1^{cc}$ is a clockwise ordering (we may have $v_2^{cl}=w_1$ and/or $w_2=v_1^{cc}$).
	We write $\CL(v_2,v_1)$ for the set of clockwise strands between $v_2$ and $v_1$. Similarly, we say that a \emph{counter-clockwise strand} between $v_2$ and $v_1$ is a strand from $w_1$ to $w_2$, where $v_1^{cl},\ w_2,\ w_1,\ v_2^{cc}$ is a clockwise ordering. We write $\CC(v_2,v_1)$ for the set of counter-clockwise strands between $v_2$ and $v_1$.
\end{defn}

\begin{defn}\label{defn:need}
	Suppose $v_2+1\neq v_1$.
	Let $S$ be a subset of $\CL(v_2,v_1)$. {Define $R^\cup(S)$ to be the connected component of the complement of the strands $S$ in the disk which contains the boundary vertices $v_1$ and $v_2$. Define the \emph{extremal substrands of $S$} to be the sub-paths of strands of $S$ which run along the boundary of $R^\cup(S)$. See Figure~\ref{fig:rs}. If the extremal substrands are connected, we say that $S$ is \emph{extremally connected}. If, further, the union of the extremal substrands forms a path from $v_2^{cl}$ to $v_1^{cc}$, then we say that $S$ is \emph{$(v_1,v_2)$-extremally connected}, as in the left of Figure~\ref{fig:rs}. If $v_2+1=v_1$, then $\CL(v_2,v_1)$ is empty and we consider it to be $(v_1,v_2)$-extremally connected.} The right of Figure~\ref{fig:rs} gives a non-$(v_1,v_2)$-extremally connected set of strands of $\CL(v_2,v_1)$.

	Suppose $S$ is $(v_1,v_2)$-extremally connected.
	The extremal substrands of $S$ may be pasted together to form a path $z_S$ from $v_2$ to $v_1$, viewed as either a composition of strand segments or a path of $Q$. By definition, if some internal arrow $\alpha$ is in $z_S$, then it must either be preceded by $\alpha'$ or followed by $\alpha''$ in $z_S$, where $\alpha'\alpha$ and $\alpha\alpha''$ are subpaths of the clockwise face-path containing $\alpha$. Hence, the set of arrows of $Q$ which are not in $z_S$ but share a clockwise face with arrows of $z_S$ paste together to form a rightmost path of $Q$ from $v_1$ to $v_2$, which we refer to as the \emph{rightmost return path $r_S$ of $S$}.
	See the left of Figure~\ref{fig:rs} for an example. 

	If $S$ is a collection of strands of $\CL(v_2,v_1)$ which is not $(v_1,v_2)$-extremally connected, then one may break $S$ into its maximal extremally connected components, draw a rightmost return path for each one, and connect them with $y_i$'s to obtain the \emph{rightmost return path} $r_S$ from $v_1$ to $v_2$. See the right of Figure~\ref{fig:rs}.
	\begin{figure}[H]
		\centering
		\def\svgscale{0.21}
		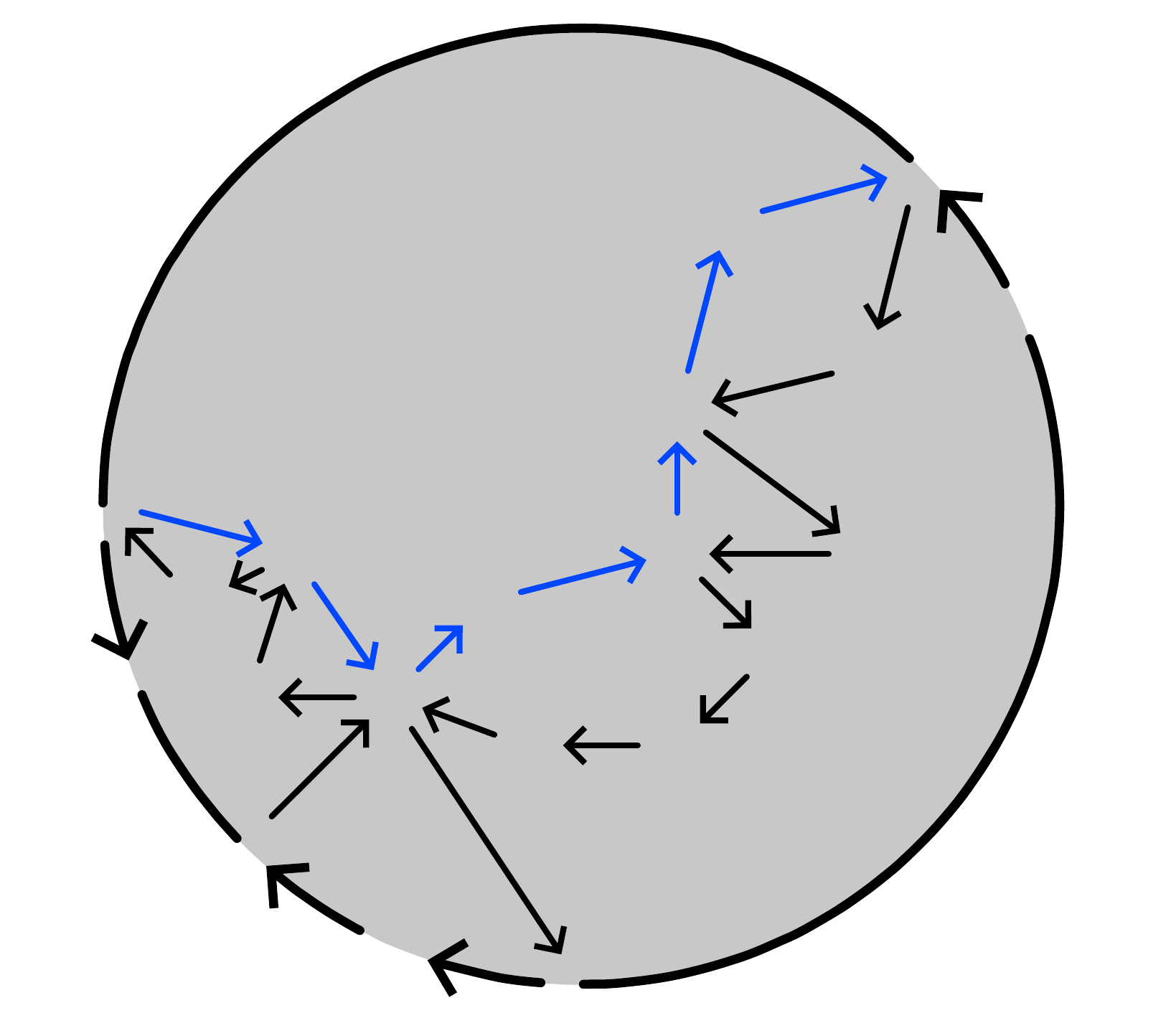
		\hspace{1cm}
		\def\svgscale{0.21}
\begingroup%
  \makeatletter%
  \providecommand\color[2][]{%
    \errmessage{(Inkscape) Color is used for the text in Inkscape, but the package 'color.sty' is not loaded}%
    \renewcommand\color[2][]{}%
  }%
  \providecommand\transparent[1]{%
    \errmessage{(Inkscape) Transparency is used (non-zero) for the text in Inkscape, but the package 'transparent.sty' is not loaded}%
    \renewcommand\transparent[1]{}%
  }%
  \providecommand\rotatebox[2]{#2}%
  \newcommand*\fsize{\dimexpr\f@size pt\relax}%
  \newcommand*\lineheight[1]{\fontsize{\fsize}{#1\fsize}\selectfont}%
  \ifx\svgwidth\undefined%
    \setlength{\unitlength}{786.64501953bp}%
    \ifx\svgscale\undefined%
      \relax%
    \else%
      \setlength{\unitlength}{\unitlength * \real{\svgscale}}%
    \fi%
  \else%
    \setlength{\unitlength}{\svgwidth}%
  \fi%
  \global\let\svgwidth\undefined%
  \global\let\svgscale\undefined%
  \makeatother%
  \begin{picture}(1,0.8611279)%
    \lineheight{1}%
    \setlength\tabcolsep{0pt}%
    \put(0,0){\includegraphics[width=\unitlength,page=1]{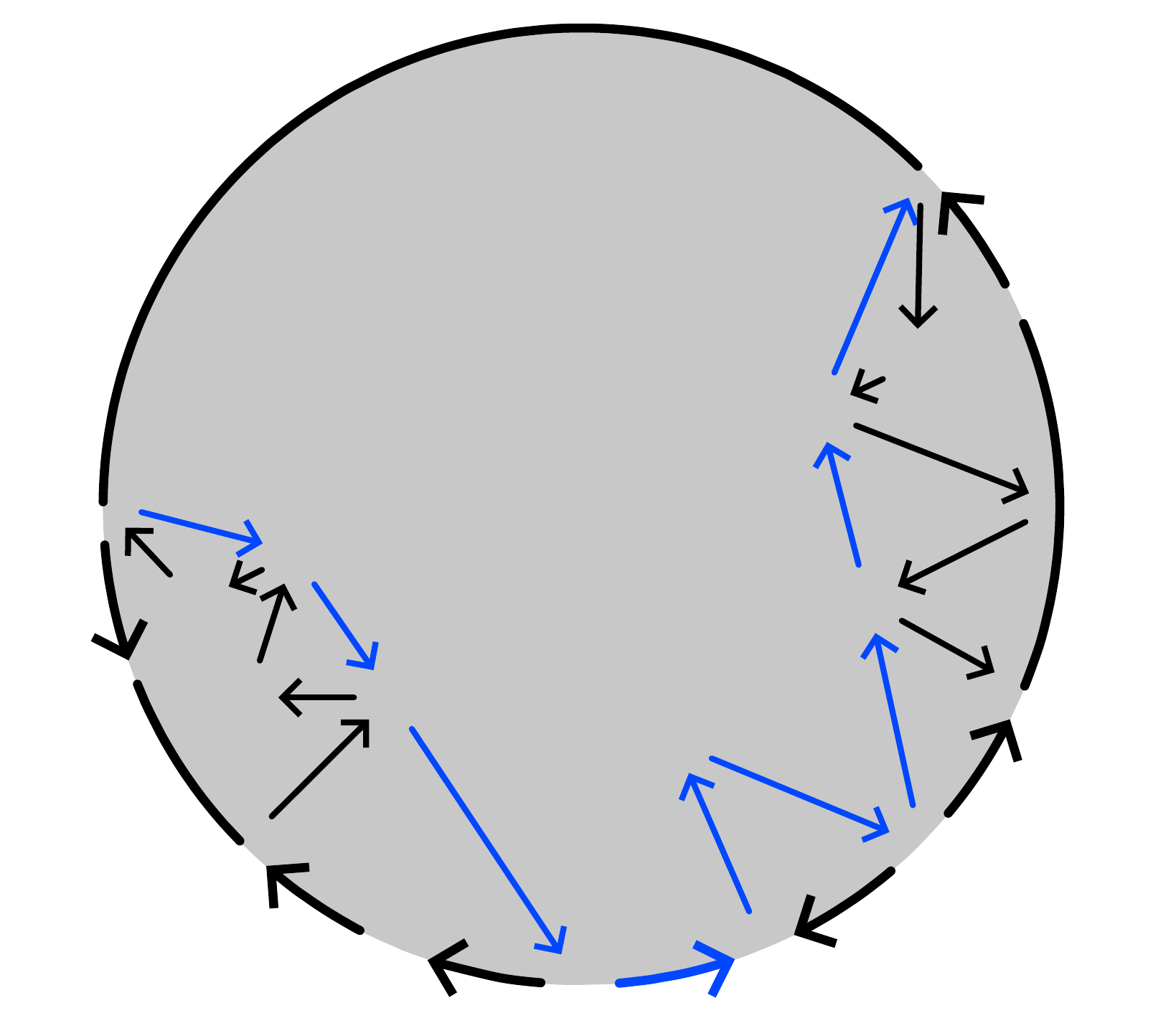}}%
    \put(0.00980493,0.40965298){\color[rgb]{0.09019608,0.08627451,0.07058824}\makebox(0,0)[lt]{\lineheight{1.25}\smash{\begin{tabular}[t]{l}$v_1$\end{tabular}}}}%
    \put(0.80584633,0.71064605){\color[rgb]{0.09019608,0.08627451,0.07058824}\makebox(0,0)[lt]{\lineheight{1.25}\smash{\begin{tabular}[t]{l}$v_2$\end{tabular}}}}%
    \put(0.79006157,0.49815173){\color[rgb]{0.09019608,0.08627451,0.07058824}\makebox(0,0)[lt]{\lineheight{1.25}\smash{\begin{tabular}[t]{l}$z_1$\end{tabular}}}}%
    \put(0.30637962,0.17194119){\color[rgb]{0.09019608,0.08627451,0.07058824}\makebox(0,0)[lt]{\lineheight{1.25}\smash{\begin{tabular}[t]{l}$z_2$\end{tabular}}}}%
    \put(0.16334941,0.30301911){\color[rgb]{0.09019608,0.08627451,0.07058824}\makebox(0,0)[lt]{\lineheight{1.25}\smash{\begin{tabular}[t]{l}$z_3$\end{tabular}}}}%
    \put(0.31336625,0.34964078){\color[rgb]{0.09019608,0.08627451,0.07058824}\makebox(0,0)[lt]{\lineheight{1.25}\smash{\begin{tabular}[t]{l}$\textcolor{blue}{r_T}$\end{tabular}}}}%
    \put(0,0){\includegraphics[width=\unitlength,page=2]{ext2.pdf}}%
  \end{picture}%
\endgroup%

		\caption{Shown are two sets of strands $S$ (left) and $T$ (right) of $\CL(v_2,v_1)$ and their corresponding paths $r_S$ and $r_T$, given by the blue arrows. The paths $z_S$ and $z_T$ consist of the black arrows (as paths in the dimer model $Q$) and are also drawn as the composition of the non-dotted segments of strands. Note that, on the left, $z_S$ as a path in $Q$ contains a face-path.}
		\label{fig:rs}
	\end{figure}
\end{defn}

We have defined a way to start with a set of strands $S\subseteq\CL(v_2,v_1)$ and obtain a rightmost path $r_S$. We now prepare to go in the other direction by realizing any rightmost path $p$ as $r_S$, where $S$ is a set of strands in $\CL(h(p),t(p))$.

\begin{lemma}\label{lem:faces-right}
	Let $p$ be a cycleless rightmost boundary path. Let $F$ be a clockwise face of $Q$ in $R(p)$. If $v_1$ and $v_2$ are distinct vertices of $p$ which belong to $F$, then the segment of $p$ from $v_1$ to $v_2$ is the subpath of the face-path $F$ from $v_1$ to $v_2$.
\end{lemma}
\begin{proof}
	Let $\sigma$ be the subpath of the face $F$ from $v_1$ to $v_2$.
	It suffices to prove the case when $p$ does not pass through any vertex in the segment $\sigma$ other than $v_1$ and $v_2$.

	Let $p'$ be the segment of $p$ from $v_1$ to $v_2$. Since $p$ passes through no other vertex of $\sigma$, the path $\sigma(p')^{-1}$ must be a simple loop, and hence bound a disk which induces a dimer submodel $Q'$ of $Q$. The path $\sigma$ is minimal in $Q'$ by Lemma~\ref{lem:subpath-of-facepath-is-thin}, and $p'$ is rightmost, hence Proposition~\ref{thm:leftmost-c-valueZ} shows that $\sigma$ is to the left of $p'$ in $Q'$. On the other hand, since $F$ is to the right of $p$ in $Q$, the path $\sigma$ must be to the right of $p'$ in $Q'$. Then $\sigma$ is both to the left and to the right of $p'$ in $Q'$, so we must have $\sigma=p'$ (in $Q'$ and in $Q$).
\end{proof}

\begin{defn}
	A cycleless boundary path $p$ of $Q$ is \emph{right-direct} if $p$ factors through no vertex of $(h(p),t(p))$. Symmetrically, $p$ is \emph{left-direct} if it factors through no vertex of $(t(p),h(p))$.
\end{defn}
A boundary path is direct if and only if it is both left-direct and right-direct.
We say that a \emph{right-direct component} of a path $p$ is a subpath $p'$ of $p$ which starts and ends at vertices of $(h(p),t(p))$, but does not otherwise use any arrows of $(h(p),t(p))$. Then any cycleless boundary path $p$ factors uniquely as $p=p_1\dots p_m$, where $p_i$ is the  set of right-direct components of $p$. One my similarly factor a boundary path as a composition of its \emph{left-direct components}.

\begin{defn}\label{defn:rss}
	Let $p$ be a {right-direct} rightmost path such that $h(p)+1\neq t(p)$.
	We say that the \textit{right-supporting faces} of $p$ are those clockwise faces in $R(p)$ which share a vertex with $p$.

	Let $F$ be a right-supporting face of $p$. By Lemma~\ref{lem:faces-right}, the arrows of $F$ which are not in $p$ form a path $\gamma_1\dots\gamma_m$, where $m\geq1$. {If $m=1$, then $p$ contains $\sigma:=\delta_1\dots\delta_{m'}$ such that $\gamma_1\sigma$ is a clockwise face-path by Lemma~\ref{lem:faces-right} -- since $p$ is rightmost, it must be true that $\gamma_1$ is a boundary arrow. Since $\gamma_1$ is part of a clockwise face, it is the arrow $x_j$ from $j$ to $j+1$ for some $j$. Then $\sigma$ is the path $y_j$. Since $p$ is right-direct, it must be the case that $\sigma=p$; this contradicts the assumption that $h(p)+1\neq t(p)$. This shows that we must have $m\geq2$.}

	For any $i\in[m-1]$, we say that the path $\gamma_i\gamma_{i+1}$ is a \emph{right-supporting pair}, or \emph{RSP}, of $p$. 
	If $i>1$, then $\gamma_i\gamma_{i+1}$ is an \emph{initial RSP}. If $i<m-1$, then $\gamma_i\gamma_{i+1}$ is a \emph{terminal RSP}.
	If there is a boundary arrow $\gamma$ in $R(p)$ beginning at $t(p)$,  then we consider $\gamma$ to be a terminal RSP. This occurs in the left of Figure~\ref{fig:rs}. Dually, if there is a boundary arrow $\delta$ in $R(p)$ ending at $h(p)$, then we consider $\delta$ to be an initial RSP.
	An RSP which is neither initial nor terminal is a \emph{middle RSP}.
	{Note that an RSP is middle if and only if $m=2$ (and, hence, $i=1$).}
	A \emph{right-supporting strand} of $p$ is a strand $z$ containing an RSP of $p$.

	{If $\gamma_i\gamma_{i+1}$ is a non-terminal RSP of $p$ contained in a right-supporting strand $z$, then the two arrows $\delta\delta'$ following $\gamma_{i+1}$ in $z$ lie in a right-supporting face $F'$ of $p$, and $t(\delta)$ is a vertex of $p$. Then by Lemma~\ref{lem:faces-right}, $\delta\delta'$ is a non-initial RSP of $p$.}
	Symmetrically, if $\gamma_i\gamma_{i+1}$ is a non-initial RSP of $p$ in a right-supporting strand $z$, then the two arrows preceding $\gamma_i$ in $z$ form a non-terminal RSP of $p$.
	We then say that a \emph{right-supporting substrand} of a right-supporting strand $z$ of $p$ is a substrand of $z$ of the form $B_1\dots B_m$, where $B_1$  is an initial RSP, $B_m$ is a terminal RSP, and $B_i$ is a middle RSP for $i\in\{2,\dots,m-1\}$. 

	In Figure~\ref{fig:rs}, the strands $z_i$ form the right-supporting strands of $r_S$ on the left, and $r_T$ on the right, and their non-dotted substrands form the right-supporting substrands.
	See Figure~\ref{fig:rightsups} for an example of a single right-supporting substrand.

	Every right-supporting strand $z$ of $p$ contains a right-supporting substrand. Note that we may have $m=1$, in which case the RSP $B_1$ is both initial and terminal.
\end{defn}

Note that right-supporting substrands are contained entirely within $R(p)$. Intuitively, a right-supporting strand ``hugs $p$'' along its right-supporting substrand before ``detaching from $p$'' by going further into $R(p)$.

	\begin{figure}[H]
		\centering
		\def\svgscale{0.21}
\begingroup%
  \makeatletter%
  \providecommand\color[2][]{%
    \errmessage{(Inkscape) Color is used for the text in Inkscape, but the package 'color.sty' is not loaded}%
    \renewcommand\color[2][]{}%
  }%
  \providecommand\transparent[1]{%
    \errmessage{(Inkscape) Transparency is used (non-zero) for the text in Inkscape, but the package 'transparent.sty' is not loaded}%
    \renewcommand\transparent[1]{}%
  }%
  \providecommand\rotatebox[2]{#2}%
  \newcommand*\fsize{\dimexpr\f@size pt\relax}%
  \newcommand*\lineheight[1]{\fontsize{\fsize}{#1\fsize}\selectfont}%
  \ifx\svgwidth\undefined%
    \setlength{\unitlength}{675.64099121bp}%
    \ifx\svgscale\undefined%
      \relax%
    \else%
      \setlength{\unitlength}{\unitlength * \real{\svgscale}}%
    \fi%
  \else%
    \setlength{\unitlength}{\svgwidth}%
  \fi%
  \global\let\svgwidth\undefined%
  \global\let\svgscale\undefined%
  \makeatother%
  \begin{picture}(1,0.26222657)%
    \lineheight{1}%
    \setlength\tabcolsep{0pt}%
    \put(0,0){\includegraphics[width=\unitlength,page=1]{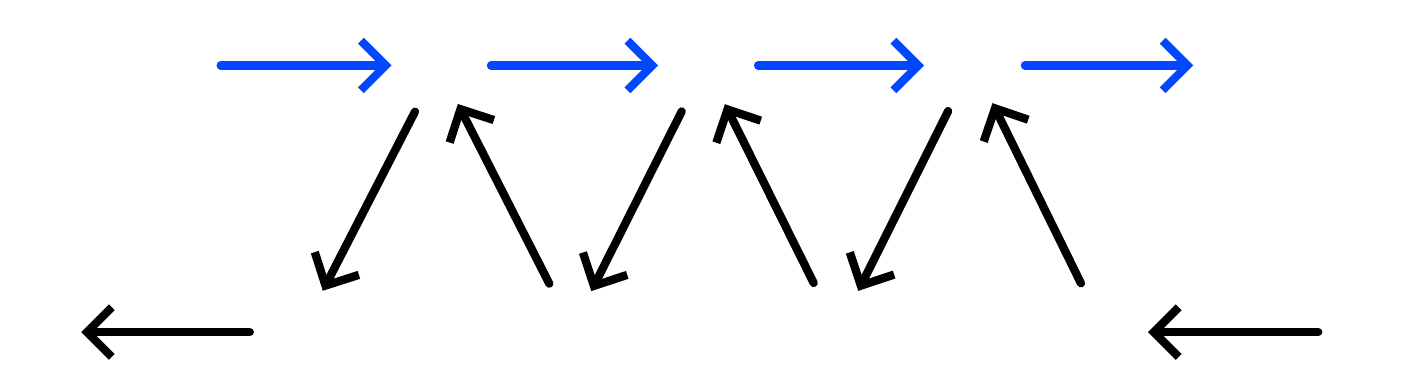}}%
    \put(0.41747023,0.2524227){\color[rgb]{0.09019608,0.08627451,0.07058824}\makebox(0,0)[lt]{\lineheight{1.25}\smash{\begin{tabular}[t]{l}$\textcolor{blue}{p}$\end{tabular}}}}%
    \put(0,0){\includegraphics[width=\unitlength,page=2]{rightsups.pdf}}%
  \end{picture}%
\endgroup%

		\caption{A rightmost path $p$ and a right-supporting substrand.}
		\label{fig:rightsups}
	\end{figure}

If $p$ is a right-direct rightmost path and $h(p)+1=t(p)$, then we must have $p=y_{h(p)}$ and we say that $p$ has no right-supporting strands. If $p$ is an arbitrary rightmost boundary path, we decompose $p=p_1\dots p_m$, where each $p_i$ is a right-direct component of $p$. Then we say that the right-supporting strands of $p$ are the union of the right-supporting strands of $p_i$ for every $i\in[m]$. In this way we may consider the right-supporting strands of an arbitrary rightmost boundary path.

\begin{lemma}\label{lem:rs-contain}
	Let $p$ be a rightmost cycleless boundary path. The right-supporting strands of $p$ are strands of $\CL\big(h(p),t(p)\big)$ which are contained entirely within $R(p)$.
\end{lemma}
\begin{proof}
	Since right-supporting strands are defined in terms of right-direct components, it suffices to consider the case when $p$ is right-direct.
	Let $z_\delta$ be a right-supporting strand of $p$ with right-supporting substrand $\delta$. Let $z_\delta^+$ be the subpath of $z_\delta$ (considered as a path in $Q$) after $\delta$ until either $z_\delta$ ends, or has a shared vertex with $p$. Factor $p$ as $p^-_\delta p_\delta p^+_\delta$, where $p_\delta$ is the segment of $p$ which is contained the right-supporting faces containing arrows of $z_\delta$. We claim that, if $z_\delta^+$ intersects with $p$, then its first intersection with $p$ must be with $p^+_\delta$.

	To show this this, number the right-supporting substrands of $p$ as $\delta_1,\dots,\delta_m$
	so that $\delta_1$ begins at $h(p)^{cl}$, and indices increase moving backwards along $p$ until $\delta_m$ ends at $t(p)^{cc}$.
	It is immediate that the claim holds for $\delta_m$, as $z_{\delta_m}$ does not continue after $\delta_m$. Suppose for some $j>1$ that we have shown the claim for $\delta_{j+1}$; we show that it holds for $j$. Note that $\delta_j$ ends with an intersection with $\delta_{j+1}$; see Figure~\ref{fig:JFG}. Since $z_{\delta_{j+1}}^+$ ends with an intersection with $p_{\delta_{j+1}}^+$ or with a strand-vertex in $(h(p)^{cl},t(p)^{cc})$, it is impossible for $z_{\delta_j}^+$ to reach $p_{\delta_j}^-$ or $p_{\delta_j}$ without creating a bad lens with $z_{\delta_{j+1}}$, a self-intersection with $z_{\delta_j}$, or a preliminary intersection with $p_{\delta_j}^+$. See Figure~\ref{fig:JFG}.
	This completes the proof of the claim.
	\begin{figure}[H]
		\centering
		\def\svgscale{0.21}
\begingroup%
  \makeatletter%
  \providecommand\color[2][]{%
    \errmessage{(Inkscape) Color is used for the text in Inkscape, but the package 'color.sty' is not loaded}%
    \renewcommand\color[2][]{}%
  }%
  \providecommand\transparent[1]{%
    \errmessage{(Inkscape) Transparency is used (non-zero) for the text in Inkscape, but the package 'transparent.sty' is not loaded}%
    \renewcommand\transparent[1]{}%
  }%
  \providecommand\rotatebox[2]{#2}%
  \newcommand*\fsize{\dimexpr\f@size pt\relax}%
  \newcommand*\lineheight[1]{\fontsize{\fsize}{#1\fsize}\selectfont}%
  \ifx\svgwidth\undefined%
    \setlength{\unitlength}{1103.69897461bp}%
    \ifx\svgscale\undefined%
      \relax%
    \else%
      \setlength{\unitlength}{\unitlength * \real{\svgscale}}%
    \fi%
  \else%
    \setlength{\unitlength}{\svgwidth}%
  \fi%
  \global\let\svgwidth\undefined%
  \global\let\svgscale\undefined%
  \makeatother%
  \begin{picture}(1,0.40119363)%
    \lineheight{1}%
    \setlength\tabcolsep{0pt}%
    \put(0,0){\includegraphics[width=\unitlength,page=1]{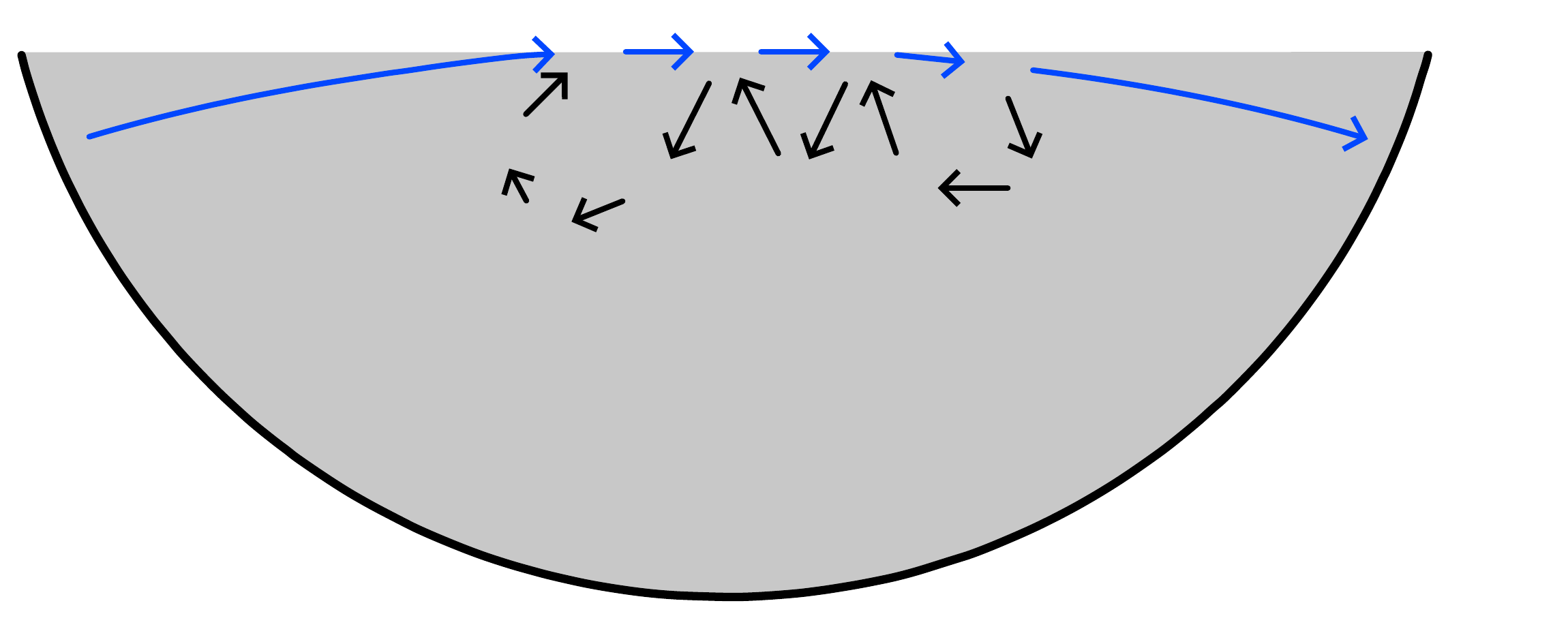}}%
    \put(0.60015005,0.32272404){\color[rgb]{0,0,0}\makebox(0,0)[lt]{\lineheight{1.25}\smash{\begin{tabular}[t]{l}$\delta_j$\end{tabular}}}}%
    \put(0.43454874,0.18032924){\color[rgb]{0,0,0}\makebox(0,0)[lt]{\lineheight{1.25}\smash{\begin{tabular}[t]{l}$z_{\delta_j}$\end{tabular}}}}%
    \put(0.26603812,0.22124094){\color[rgb]{0,0,0}\makebox(0,0)[lt]{\lineheight{1.25}\smash{\begin{tabular}[t]{l}$z_{\delta_{j+1}}$\end{tabular}}}}%
    \put(0.15839464,0.37582675){\color[rgb]{0.09019608,0.08627451,0.07058824}\makebox(0,0)[lt]{\lineheight{1.25}\smash{\begin{tabular}[t]{l}$\textcolor{blue}{p^-_{\delta_j}}$\end{tabular}}}}%
    \put(0.73923146,0.37329426){\color[rgb]{0.09019608,0.08627451,0.07058824}\makebox(0,0)[lt]{\lineheight{1.25}\smash{\begin{tabular}[t]{l}$\textcolor{blue}{p^+_{\delta_j}}$\end{tabular}}}}%
    \put(0.48498641,0.39138994){\color[rgb]{0.09019608,0.08627451,0.07058824}\makebox(0,0)[lt]{\lineheight{1.25}\smash{\begin{tabular}[t]{l}$\textcolor{blue}{p_{\delta_j}}$\end{tabular}}}}%
    \put(0,0){\includegraphics[width=\unitlength,page=2]{prl.pdf}}%
  \end{picture}%
\endgroup%

	\caption{$z_{\delta_j}^+$ cannot leave $R(p)$ through an intersection with $p_{\delta_j}^-$ without forcing $z_{\delta_{j+1}}$ to do the same.}
		\label{fig:JFG}
	\end{figure}

	Symmetrically, one may prove that for any right-supporting substrand $\delta$ of $p$, either the substrand $z_\delta^-$ (considered as a path in $Q$) does not intersect $p$ or its first intersection with $p$ before $\delta_j$ is with $p^-_\delta$.

	Now, we are able to show that neither $z_\delta^+$ nor $z_\delta^-$ intersect with $p$. Indeed, if $z_\delta^+$ intersects with $p$, then by the claim above, its first intersection must be with $p_\delta^+$. To avoid a self-intersection, this forces $z_\delta^-$ to cross out of $R(p)$ through an intersection with $p_\delta^+$, contradicting the symmetrized claim.
	It follows that $z_\delta^+$ does not intersect with $p$; symmetrically, $z_\delta^-$ does not intersect with $p$. This shows that $z_\delta$ is contained entirely within $R(p)$.
	Note now that $z_\delta^+$ cannot reach any strand-vertex in $[h(p)^{cl},t(z_\delta^-)]$ without crossing $z_\delta^-$ or crossing out of $R(p)$, both of which are forbidden, hence $z_\delta^+$ is forced to end between $t(z_\delta^-)$ and $t(p)$. It follows that $z_\delta$ is in $\CL(h(p),t(p))$.
\end{proof}

\begin{prop}\label{lem:rs-enough}
	If $p$ is a cycleless rightmost boundary path from $v_1$ to $v_2$ and $S$ is the set of right-supporting strands of $p$, then $p=r_S$.
	{Moreover, $S$ is $(v_1,v_2)$-extremally connected if and only if $p$ is right-direct.}
\end{prop}
\begin{proof}
	Right-supporting faces and right-supporting substrands are defined such that the arrows of $p$ are precisely the arrows of right-supporting faces of $p$ which are not contained in right-supporting substrands of $p$.
	Lemma~\ref{lem:rs-contain} shows that right-supporting strands of $p$ are contained in $R(p)$, hence that this is the same as the arrows of right-supporting faces of $p$ which are not contained in right-supporting \emph{strands} of $p$. This shows that $p=r_S$.

	It remains to show that $S$ is $(v_1,v_2)$-extremally connected if and only if $p$ is right-direct.
	If $p$ factors through no vertices of $(v_2,v_1)$, then the right-supporting substrands of $p$ paste together to form a path from $v_2^{cl}$ to $v_1^{cc}$, hence the set of right-supporting strands of $p$ is $(v_1,v_2)$-extremally connected. On the other hand, if $p$ is not right-direct, then we may write $p$ as a composition of boundary paths $p_1p_2$, where $h(p_1)=t(p_2)\in(v_2,v_1)$. Then any right-supporting strand of $p$ either starts and ends in $[h(p_2)^{cl},t(p_2)^{cc}]$, or starts and ends in $[h(p_1)^{cl},t(p_1)^{cc}]$, hence $S$ cannot be extremally connected.
\end{proof}

We now argue that the relation number of a non-right-direct (minimal) rightmost boundary path may be calculated from the relation numbers of its right-direct subpaths.

\begin{lemma}\label{lem:workit}
	Let $p$ be a minimal rightmost boundary path. Let $p=p_1\dots p_m$, where each $p_i$ is a right-direct component of $p$.
	Then ${\Y}(p)=\sum_{i=1}^m {\Y}(p_i)$ and any strand of $\CL(h(p),t(p))$ is a strand of $\CL(h(p_i),t(p_i))$ for exactly one $i$.
\end{lemma}
\begin{proof}
	By the definition of relation number, $[y_{t(p_i)-1}^{\reach_y(p_i)}]=[p_if^{{\Y}(p_i)}]$ for each $i\in[m]$. Then
	\[[pf^{\sum_{i=1}^m{\Y}(p_i)}]=[p_1f^{{\Y}(p_1)}]\dots[p_mf^{{\Y}(p_m)}]=[y_{t(p_1)-1}^{\reach_y(p_1)}]\dots[y_{t(p_m)-1}^{\reach_y(p_m)}].\]
	Since $p$ is minimal and hence cycleless, the vertices $h(p)=h(p_m),h(p_{m-1}),\dots,h(p_1),t(p_1)=t(p)$ must be in clockwise order. Then we calculate
		$\sum_{i=1}^m\reach_y(p_i)=\sum_{i=1}^mt(p_i)-h(p_i)=t(p)-h(p)=\reach_y(p)$
		using a telescoping series, where each $t(p_i)-h(p_i)$ and $t(p)-h(p)$ is calculated modulo $n$ so that it lies in $[n]$.
		It then follows that $[pf^{\sum_{i=1}^m{\Y}(p_i)}]=[y_{t(p)-1}^{\reach_y(p)}]$.

	It is immediate that any strand of $\CL(h(p),t(p))$ may be a strand of $\CL(h(p_i),t(p_i))$ for at most one $i$. Suppose that some strand $z$ of $\CL(h(p),t(p))$ is not a strand of any $\CL(h(p_i),t(p_i))$. Then $z$ begins before $t(p_j)$ and ends after $t(p_j)$ for some $j$. Then $r_z$, and hence $p=r_S$, does not factor through $t(p_j)$, a contradiction. This ends the proof.
\end{proof}

We now have the tools to understand minimal rightmost paths, their relation numbers, and their directness by looking at clockwise strands.

\begin{thm}\label{thm:rightmost-relation-num}
	Let $v_1$ and $v_2$ be boundary vertices of $Q$. Then
	\begin{enumerate}
		\item\label{b1} the minimal rightmost path from $v_1$ to $v_2$ is $r_{\CL(v_2,v_1)}$,
		\item\label{b2} the right relation number ${\Y}(r_{\CL(v_2,v_1)})$ is the cardinality of $\CL(v_2,v_1)$, and
		\item\label{b4} the path $r_{\CL(v_2,v_1)}$ is right-direct if and only if $\CL(v_2,v_1)$ is $(v_1,v_2)$-extremally connected.
	\end{enumerate}
\end{thm}

\begin{proof}
	Proposition~\ref{thm:leftmost-c-valueZ} shows that the minimal rightmost path from $v_1$ to $v_2$ must be the (unique) rightmost path which is to the left of all other rightmost paths from $v_1$ to $v_2$. The first statement of Proposition~\ref{lem:rs-enough} states that every rightmost path is of the form $r_S$ for some $S\subseteq\CL(v_2,v_1)$. It follows from the definition that $r_{\CL(v_2,v_1)}$ is to the left of all such paths $r_S$, hence is minimal; this shows~\eqref{b1}.
	The second statement of Proposition~\ref{lem:rs-enough} shows~\eqref{b4}.

	It remains to show~\eqref{b2}.
	We show by induction on $\reach_y(p)$ that ${\Y}(r_{\CL(v_2,v_1)}=|\CL(v_2,v_1)|$. The base case is when $v_1-v_2=1$; in this case, $\CL(v_2,v_1)$ is empty and $r_{\CL(v_2,v_1)}$ is simply $y_{v_2}$, which has a right relation number of zero by definition.
	Now let $p:=r_{\CL(v_2,v_1)}$ such that $v_2+1\neq v_1$ and suppose we have shown the result for rightmost paths $p'$ with $\reach_y(p')<\reach_y(p)$.

	If $p$ factors through a vertex of $(v_2,v_1)$, then we may write $p=p_1p_2$, where $h(p_1)=t(p_2)\in(v_2,v_1)$. It follows from Lemma~\ref{lem:workit} and the induction hypothesis that
	\[|\CL(h(p),t(p))|=\sum_{i=1}^m|\CL(h(p_i),t(p_i))|=\sum_{i=1}^m{\Y}(p_i)={\Y}(p)\]
	and the result is proven. We may now assume that $p$ factors through no vertex of $(v_2,v_1)$.

	We claim that the path $px_{v_2}$ from $v_1$ to $v_2+1$ is minimal.
	If not, then let $r$ be the minimal path from $t(px)=v_1$ to $h(px)=v_2+1$; then $C(r)<C(px)$. Then $C(ry)<C(pxy)=1$, so $ry$ is a minimal path from $v_1$ to $v_2$. On the other hand, the path $ry$ is not to the left of the rightmost path $p$ (since $p$ cannot factor through $v_2+1\in(v_2,v_1)$), which contradicts Proposition~\ref{thm:leftmost-c-valueZ}. This shows that $px$ is minimal.
	Let $q$ be the rightmost minimal path from $t(px)=v_1$ to $h(px)=v_2+1$. Since $q$ is equivalent to $px$, we know that $[qy]=[pxy]$ has a c-value of one. It follows that
	\[C(y_{t(q)-1}^{h(q)-t(q)})+1=C((xy)^{{\Y}(q)}q)+1=C((xy)^{{\Y}(q)}qy)=C(y_{t(q)-1}^{h(p)-t(p)}).\]
	In other words, ${\Y}(q)+1={\Y}(p)$. 

	Let $z$ be the strand beginning at the boundary arrow in $R(p)$ incident to $v_2=h(p)$. Then $z$ is a right-supporting strand of $p$, hence $z$ is in $\CL(v_2,v_1)$ and
$\CL(v_2-1,v_1)\cup\{z\}=\CL(v_2,v_1)$. The induction hypothesis 
	${\Y}(q)=|\CL(v_2+1,v_1)|=|\CL(v_2,v_1)|-1$ then gives the second equality of the calculation
	\[{\Y}(p)={\Y}(q)+1=|\CL(v_2-1,v_1)|+1=|\CL(v_2,v_1)|\]
	and the proof is complete.
\end{proof}
We are now ready to state and prove the main result of this subsection, which characterizes arrow-defining paths of $Q$ and their relation numbers in terms of clockwise and counter-clockwise strands.
\begin{thm}\label{thm:main-strand}
	Let $v_1$ and $v_2$ be vertices of a thin dimer model $Q$.
	Then a minimal path $p$ from $v_1$ to $v_2$ is arrow-defining if and only if $\CL(v_2,v_1)$ and $\CC(v_2,v_1)$ are both $(v_1,v_2)$-extremally connected.
	In this case, ${\X}(p)=|\CC(v_2,v_1)|$ and ${\Y}(p)=|\CL(v_2,v_1)|$.
\end{thm}
\begin{proof}
	By Proposition~\ref{thm:direct-pair-iff-achinko}, there is an arrow from $v_1$ to $v_2$ in the boundary algebra if and only if both the minimal rightmost path and the minimal leftmost path from $v_1$ to $v_2$ are direct.
	Since the minimal rightmost path is to the right of the minimal leftmost path, this is true if and only if the minimal rightmost path is right-direct and the minimal leftmost path is left-direct. By Theorem~\ref{thm:rightmost-relation-num}~\eqref{b4},
the former condition is true if and only if $\CL(v_2,v_1)$ is $(v_1,v_2)$-extremally connected; dually, the latter condition is true if and only if $\CC(v_2,v_1)$ is $(v_1,v_2)$-extremally connected. The first statement of the theorem follows.
	By Theorem~\ref{thm:rightmost-relation-num}~\eqref{b2}, the relation number ${\Y}(p)$ of the minimal path $p$ from $v_1$ to $v_2$ is the cardinality of $\CL(v_2,v_1)$. Dually, the relation number ${\X}(p)$ of $p$ is the cardinality of $\CC(v_2,v_1)$. This shows the final statement.
\end{proof}

We will see that the boundary algebra is determined by the arrows between \textit{nonadjacent} vertices and their relation numbers. When $v_1$ and $v_2$ are adjacent, the process of applying Theorem~\ref{thm:main-strand} is slightly simplified.
If, for example, $v_1+1=v_2$, then $\CL(v_2,v_1)$ is empty and $(v_1,v_2)$-extremally connected by definition, hence we need only verify whether $\CC(v_2,v_1)$ is extremally connected in order to check whether the minimal path $x_{v_1}$ from $v_1$ to $v_1+1=v_2$ is arrow-defining. Either way, $\CL(v_2,v_1)$ is empty, and hence the relation number ${\Y}(x_{v_1})=0$.

\begin{example}\label{ex:2}
	Figure~\ref{fig:ex2} shows a dimer model with strand diagram on the left. If we let $v_1=3$ and $v_2=1$, then the strands of $\CL(v_2,v_1)$ and $\CC(v_2,v_1)$ are highlighted. One can see that both sets are $(v_1,v_2)$-extremally connected, so there is an arrow-defining path $\alpha$ from 3 to 1 
	in $Q$ by Theorem~\ref{thm:main-strand}. This result also gives us relations. Since $|\CL(v_2,v_1)|=1={\X}(\alpha)$ and $|\CC(v_2,v_1)|=2={\Y}(\alpha)$, we have
	\[[\alpha(xy)^2]=[x_3^4]\textup{ and }[\alpha(xy)]=[y_2^2].\]
	Starting from these relations and using that $[xy]$ is in the center of the dimer model and the relation $[xy]=[yx]$, one calculates
	\[\begin{matrix}
		[\alpha(yx)^2]=[x_3^4] & [(xy)\alpha(yx)]=[x_3^4] & [(xy)^2\alpha]=[x_3^4] &  [\alpha(xy)]=[y_2^2] & [(yx)\alpha]=[y_2^2] \\
		[\alpha y^2]=[x_3^2] & [y\alpha y]=[x_4^2] & [y^2\alpha]=[x_5^2] &
		[\alpha x]=[y_2] & [x\alpha]=[y_1],
	\end{matrix} \]
	where the bottom row is obtained by cancelling the $x$'s from the extreme sides of the top row. The relations of the bottom row are the cyclic sets of relations associated to $\alpha$ as in Proposition~\ref{thm:cyc-set}.

	One may check that for any other choice of nonadjacent $w_1$ and $w_2$, either $\CL(w_2,w_1)$ or $\CC(w_2,w_1)$ is not $(w_2,w_1)$-extremally connected, hence $a$ is the only nonadjacent arrow-defining path of $Q$.
	Similarly, one may check using Theorem~\ref{thm:main-strand} that all paths $x_i$ and $y_i$ are arrow-defining, with the exception of $y_1$ and $y_2$. This explains which adjacent arrows appear in the Gabriel quiver.

	\begin{figure}[H]
		\centering
		\def\svgscale{0.21}
		\input{dim.pdf_tex}
		\hspace{1cm}
		\def\svgscale{0.21}
		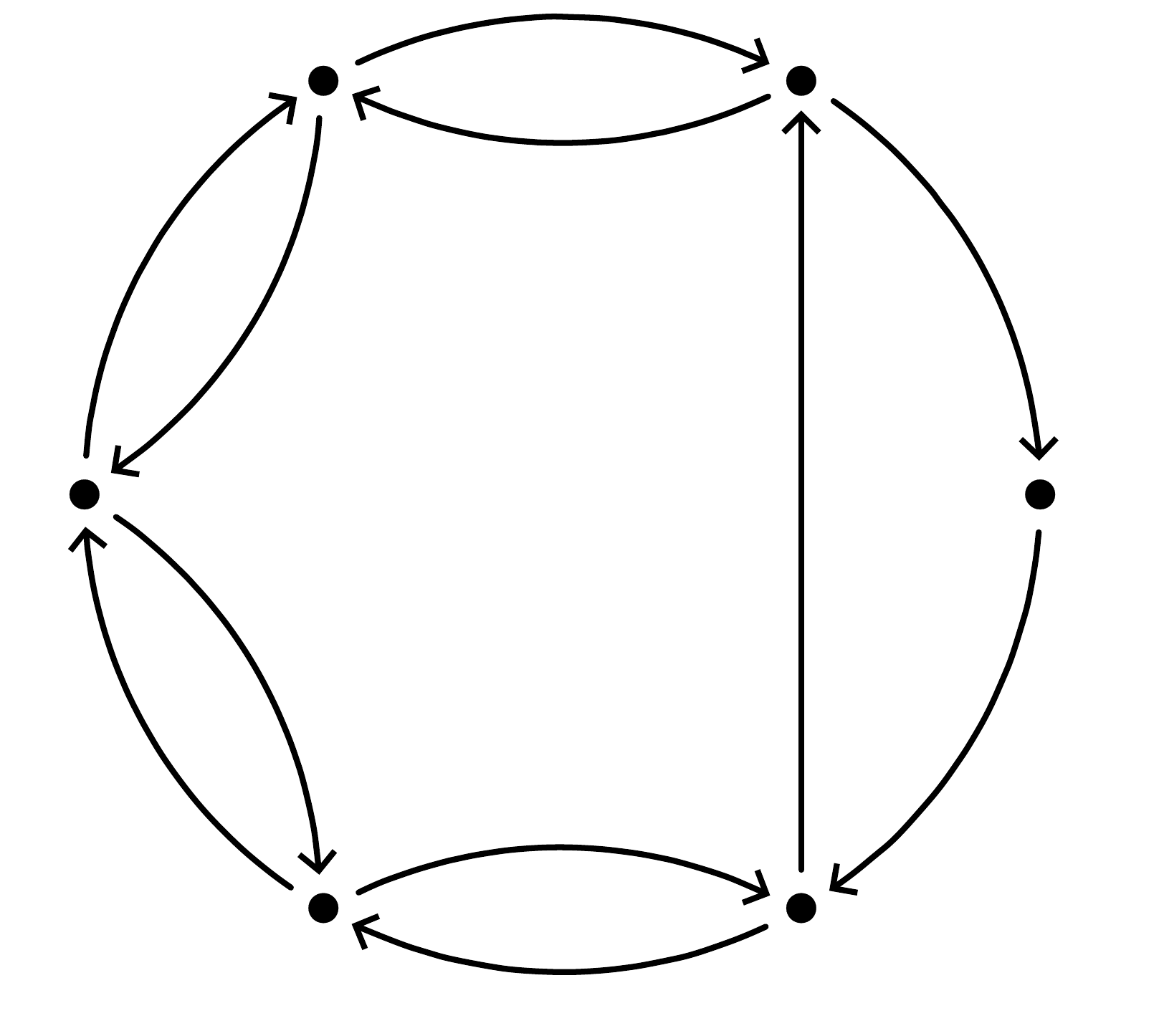
		\caption{On the left is a dimer model with its strand diagram overlayed. Strands of $\CL(3,1)$ and $\CC(3,1)$ are in red, and other strands are in blue. On the right is the Gabriel quiver of its boundary algebra. The nonadjacent arrow $\alpha$ is shown in red and labelled with $2:1$ to denote that ${\X}(\alpha)=2=\big|\CC(3,1)\big|$ and ${\Y}(\alpha)=1=\big|\CL(3,1)\big|$.}
		\label{fig:ex2}
	\end{figure}

\end{example}

\subsection{Decorated Permutations}
\label{ssec:affperm}

Theorem~\ref{thm:main-strand} shows that we may understand minimal paths of $Q$ and their relation numbers through the strand diagram. The definition of extremal connectedness used in this result appears to use information about the strand diagram other than the start and end points of strands. 
On the other hand, we saw in Proposition~\ref{prop:ba-from-perm} that the isomorphism class of the boundary algebra $B_Q$ depends only on the underlying decorated permutation $\pi$ of $Q$. It is then a reasonable desire to obtain information about the boundary algebra directly from this permutation, so that the boundary algebra $B^\pi$ of a positroid $\pi$ may be calculated without choosing a representative dimer model $Q$. We now show that extremal connectedness may be verified only from the information of the start and end points of strands and we use this to obtain a version of Theorem~\ref{thm:main-strand} using the decorated permutation rather than the strand diagram.

Recall that an \textit{inversion} of $\pi$ is an index $i\in[n]$ such that $\pi(i)<i$ and a \textit{noninversion} of $\pi$ is an index $i\in[n]$ such that $\pi(i)\geq i$.

\begin{lemma}\label{lem:fdg}
	Let $\pi$ be a connected decorated permutation of $[n]$ and let $v_1,v_2\in[n]$ be distinct elements. There exists a thin dimer model with decorated permutation $\pi$ such that
	\begin{enumerate}
		\item any two strands of $\CC(v_2,v_1)$ intersect at most once, and
		\item any two strands of $\CL(v_2,v_1)$ intersect at most once.
	\end{enumerate}
\end{lemma}
\begin{proof}
	By cyclically permuting the groundset $[n]$, it suffices to suppose that $v_2=1$.
	Observe that each strand of $\CC(1,v_1)$ corresponds to an inversion of $\pi$ and each strand of $\CL(1,v_1)$ corresponds to a noninversion of $\pi$.
	Let $Q$ be the dimer model of the $\Le$-diagram with decorated permutation $\pi$ as in~\cite{ZPostnikov}.
	The surface of the Le-diagram (homeomorphic to a disk) looks like a Young tableau, with a unique upper-left corner and potentially many lower-right corners. See Figure~\ref{fig:le}. Any strand $z$ corresponding to an inversion of $\pi$ may be drawn as $z^-z^+$, where $z^-$ moves up and to the left, and $z^+$ moves straight right. See Figure~\ref{fig:le} for two examples. By properties of Le-diagrams, an intersection between two inversion strands $z_1^-z_1^+$ and $z_2^-z_2^+$ must occur, without loss of generality, between $z_1^+$ and $z_2^-$. It follows that any two inversion strands may intersect at most once, and hence that two strands of $\CC(1,v_1)$ may intersect at most once.
	\begin{figure}[H]
		\centering
		\def\svgscale{0.21}
\begingroup%
  \makeatletter%
  \providecommand\color[2][]{%
    \errmessage{(Inkscape) Color is used for the text in Inkscape, but the package 'color.sty' is not loaded}%
    \renewcommand\color[2][]{}%
  }%
  \providecommand\transparent[1]{%
    \errmessage{(Inkscape) Transparency is used (non-zero) for the text in Inkscape, but the package 'transparent.sty' is not loaded}%
    \renewcommand\transparent[1]{}%
  }%
  \providecommand\rotatebox[2]{#2}%
  \newcommand*\fsize{\dimexpr\f@size pt\relax}%
  \newcommand*\lineheight[1]{\fontsize{\fsize}{#1\fsize}\selectfont}%
  \ifx\svgwidth\undefined%
    \setlength{\unitlength}{582.76599121bp}%
    \ifx\svgscale\undefined%
      \relax%
    \else%
      \setlength{\unitlength}{\unitlength * \real{\svgscale}}%
    \fi%
  \else%
    \setlength{\unitlength}{\svgwidth}%
  \fi%
  \global\let\svgwidth\undefined%
  \global\let\svgscale\undefined%
  \makeatother%
  \begin{picture}(1,0.71530434)%
    \lineheight{1}%
    \setlength\tabcolsep{0pt}%
    \put(0,0){\includegraphics[width=\unitlength,page=1]{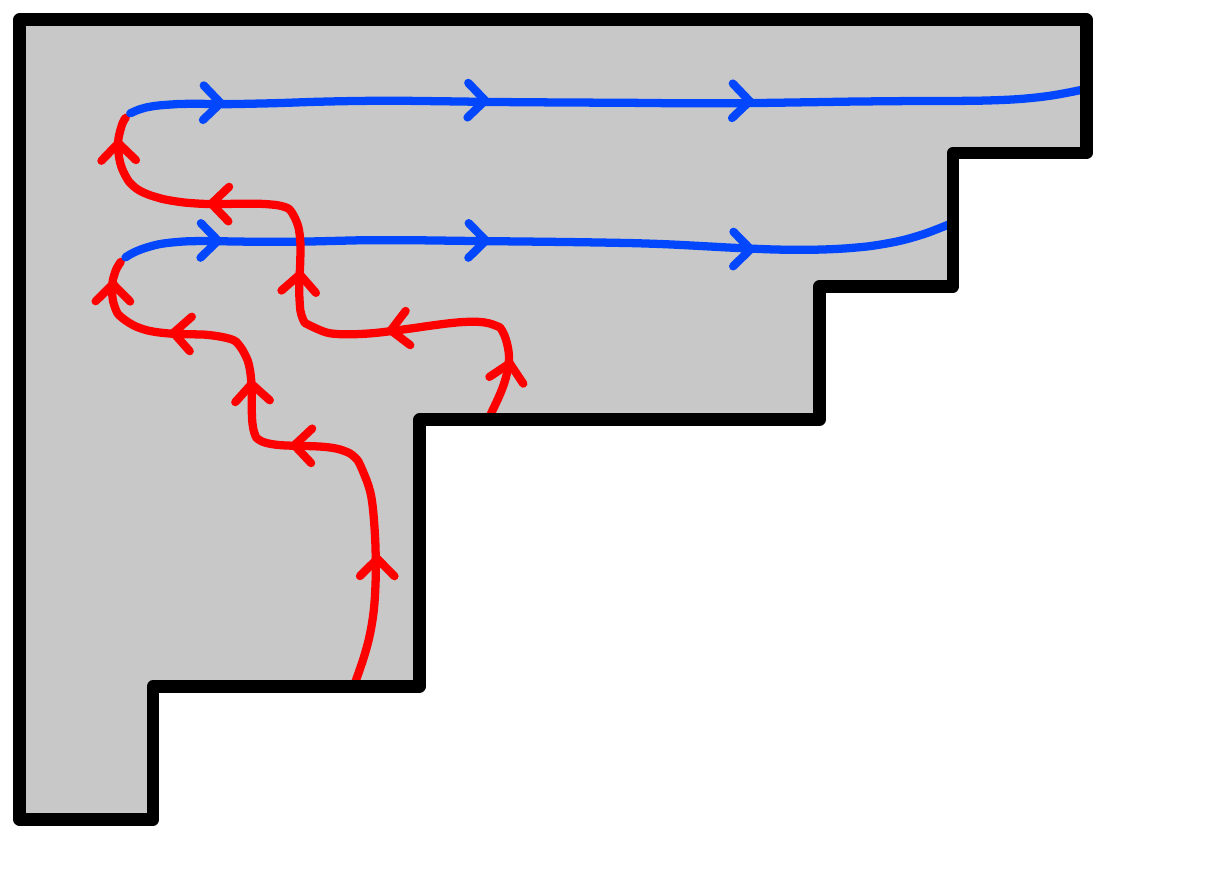}}%
    \put(0.92251265,0.59593618){\color[rgb]{1,0,0}\makebox(0,0)[lt]{\lineheight{1.25}\smash{\begin{tabular}[t]{l}$1$\end{tabular}}}}%
    \put(0.81520886,0.48664892){\color[rgb]{1,0,0}\makebox(0,0)[lt]{\lineheight{1.25}\smash{\begin{tabular}[t]{l}$3$\end{tabular}}}}%
    \put(0.41085616,0.28710382){\color[rgb]{1,0,0}\makebox(0,0)[lt]{\lineheight{1.25}\smash{\begin{tabular}[t]{l}$8$\end{tabular}}}}%
    \put(0.26873051,0.05814537){\color[rgb]{1,0,0}\makebox(0,0)[lt]{\lineheight{1.25}\smash{\begin{tabular}[t]{l}$11$\end{tabular}}}}%
    \put(0,0){\includegraphics[width=\unitlength,page=2]{le.pdf}}%
  \end{picture}%
\endgroup%

		\caption{Two inversion strands drawn in a Le-diagram. The segments $z_i^-$ are in red, and the segments $z_i^+$ are in blue.}
		\label{fig:le}
	\end{figure}

	Similarly, any two strands corresponding to noninversions of $\pi$ may intersect at most once, hence any two strands of $\CL(v_2,v_1)$ may intersect at most once.
\end{proof}

Recall the definition of the clockwise interval $(i,j)$ from Definition~\ref{defn:cyc-int}.

\begin{cor}\label{cor:cr}
	Let $v_1$ and $v_2$ be distinct boundary vertices of $Q$. Then 
	\begin{enumerate}
		\item\label{gf1} $\CL(v_2,v_1)$ is $(v_1,v_2)$-extremally connected if and only if every vertex of $(v_2,v_1)$ is to the left of some strand of $\CL(v_2,v_1)$, and
		\item\label{gf2} $\CC(v_2,v_1)$ is $(v_1,v_2)$-extremally connected if and only if every vertex of $(v_1,v_2)$ is to the right of some strand of $\CC(v_2,v_1)$.
	\end{enumerate}
\end{cor}

\begin{proof}
	We show~\eqref{gf1}; the proof of~\eqref{gf2} is symmetric.
	By Lemma~\ref{lem:fdg} we may choose a dimer model $Q'$ with the same decorated permutation $\pi$ as $Q$ such that two strands of $\CL_{Q'}(v_2,v_1)$ intersect at most once and two strands of $\CC_{Q'}(v_2,v_1)$ intersect at most once.
	
	We first show that $\CL_Q(v_2,v_1)$ is $(v_1,v_2)$-extremally connected if and only if $\CL_{Q'}(v_2,v_1)$ is $(v_1,v_2)$-extremally connected. By Theorem~\ref{thm:rightmost-relation-num}~\eqref{b4}, the set $\CL_Q(v_2,v_1)$ is $(v_1,v_2)$-extremally connected if and only if no minimal path from $v_1$ to $v_2$ in $B_Q$ factors through a vertex of $(v_1,v_2)$. By Proposition~\ref{prop:ba-from-perm}, this is true if and only if no minimal path from $v_1$ to $v_2$ in $B_{Q'}$ factors through a vertex of $(v_1,v_2)$. By Theorem~\ref{thm:rightmost-relation-num}~\eqref{b4} (now applied to $Q'$), this is true if and only if $\CL_{Q'}(v_2,v_1)$ is $(v_1,v_2)$-extremally connected.

	It remains to show that $(v_1,v_2)$-extremal connectedness of $\CL_{Q'}(v_2,v_1)$ is equivalent to the condition that every vertex of $(v_2,v_1)$ is to the left of some strand of $\CL_{Q'}(v_2,v_1)$.
	All of the following calculations will be done in $Q'$.

	First, it is immediate that if the latter condition holds, then $\CL_{Q'}(v_2,v_1)$ is $(v_1,v_2)$-extremally connected. We show the forwards implication.
	Suppose that $\CL_{Q'}(v_2,v_1)$ is extremally connected and order the right-supporting substrands of $p:=r_{\CL_{Q'}(v_2,v_1)}$ as in Definition~\ref{defn:rss} as $\delta_1,\dots,\delta_m$, where $\delta_1$ begins at $h(p)^{cl}$ and indices of substrands increase moving backwards along $p$.
	Since $\CL_{Q'}(v_2,v_1)$ is extremally connected, these substrands paste together to form a path from $v_2^{cl}$ to $v_1^{cc}$. Let $z_1,\dots,z_m$ be the strands containing the substrands $\delta_1,\dots,\delta_m$. All of these are in $\CL_{Q'}(v_2,v_1)$ by Lemma~\ref{lem:rs-contain}, hence by choice of $Q'$ the only intersection between $z_i$ and $z_{i+1}$, for $i\in[m-1]$, must be the intersection occurring between $\delta_i$ and $\delta_{i+1}$. Then $v_2^{cl}\leq t(z_i)<t(z_{i+1})\leq h(z_i)<h(z_{i+1})\leq v_1^{cc}$ is a clockwise ordering for all $i\in[m-1]$. It follows that every vertex in $(v_2,v_1)$ must be to the left of at least one of the strands $z_i$.

	We have shown that $\CL_{Q'}(v_2,v_1)$ is $(v_1,v_2)$-extremally connected if and only if every vertex of $(v_2,v_1)$ is to the left of some strand of $\CL_{Q'}(v_2,v_1)$.
	Since $Q'$ and $Q$ have the same decorated permutation, this is true if and only if every vertex of $(v_2,v_1)$ is to the left of some strand of $\CL_Q(v_2,v_1)$. This completes the proof.
\end{proof}

Consider Figure~\ref{fig:FXFX}. Looking only at the strands $\{z_1,z_2\}$, one sees that $\CL(v_2,v_1)$ is $(v_1,v_2)$-extremally connected. On the other hand, the vertex $w\in (v_2,v_1)$ is not to the left of $z_1$ or $z_2$. Corollary~\ref{cor:cr} implies that there must exist some strand $z_w\in\CC(v_2,v_1)$ which has $w$ to its left, if the dimer model in question is thin.
		\begin{figure}
			\centering
			\def\svgscale{.21}
\begingroup%
  \makeatletter%
  \providecommand\color[2][]{%
    \errmessage{(Inkscape) Color is used for the text in Inkscape, but the package 'color.sty' is not loaded}%
    \renewcommand\color[2][]{}%
  }%
  \providecommand\transparent[1]{%
    \errmessage{(Inkscape) Transparency is used (non-zero) for the text in Inkscape, but the package 'transparent.sty' is not loaded}%
    \renewcommand\transparent[1]{}%
  }%
  \providecommand\rotatebox[2]{#2}%
  \newcommand*\fsize{\dimexpr\f@size pt\relax}%
  \newcommand*\lineheight[1]{\fontsize{\fsize}{#1\fsize}\selectfont}%
  \ifx\svgwidth\undefined%
    \setlength{\unitlength}{1122.5bp}%
    \ifx\svgscale\undefined%
      \relax%
    \else%
      \setlength{\unitlength}{\unitlength * \real{\svgscale}}%
    \fi%
  \else%
    \setlength{\unitlength}{\svgwidth}%
  \fi%
  \global\let\svgwidth\undefined%
  \global\let\svgscale\undefined%
  \makeatother%
  \begin{picture}(1,0.42076258)%
    \lineheight{1}%
    \setlength\tabcolsep{0pt}%
    \put(0,0){\includegraphics[width=\unitlength,page=1]{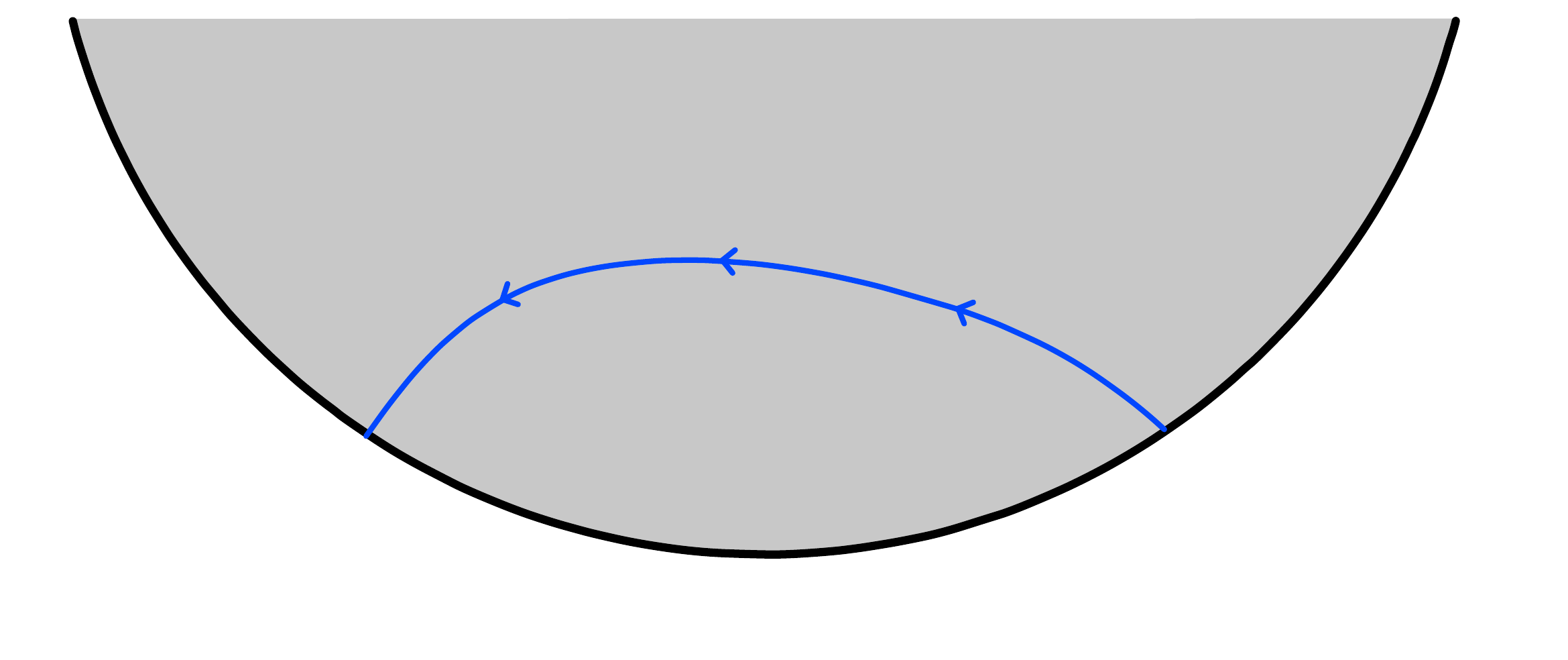}}%
    \put(0.68005969,0.30377764){\color[rgb]{1,0,0}\makebox(0,0)[lt]{\lineheight{1.25}\smash{\begin{tabular}[t]{l}$z_1$\end{tabular}}}}%
    \put(0.30901292,0.32649479){\color[rgb]{1,0,0}\makebox(0,0)[lt]{\lineheight{1.25}\smash{\begin{tabular}[t]{l}$z_2$\end{tabular}}}}%
    \put(0.56613274,0.25163029){\color[rgb]{0.00784314,0.27843137,0.99607843}\makebox(0,0)[lt]{\lineheight{1.25}\smash{\begin{tabular}[t]{l}$\exists z_w$\end{tabular}}}}%
    \put(0.00980401,0.34502129){\color[rgb]{0,0,0}\makebox(0,0)[lt]{\lineheight{1.25}\smash{\begin{tabular}[t]{l}$v_1$\end{tabular}}}}%
    \put(0.92961604,0.34313122){\color[rgb]{0,0,0}\makebox(0,0)[lt]{\lineheight{1.25}\smash{\begin{tabular}[t]{l}$v_2$\end{tabular}}}}%
    \put(0.48114477,0.03110984){\color[rgb]{0,0,0}\makebox(0,0)[lt]{\lineheight{1.25}\smash{\begin{tabular}[t]{l}$w$\end{tabular}}}}%
    \put(0,0){\includegraphics[width=\unitlength,page=2]{FXFX.pdf}}%
  \end{picture}%
\endgroup%

			\caption{Using the existence of the strands $\{z_1,z_2\}$ and Corollary~\ref{cor:cr}, there must exist a strand $z_w\in\CL(v_2,v_1)$ with $w$ to its left.}
			\label{fig:FXFX}
		\end{figure}

We now use Corollary~\ref{cor:cr} to prove a version of Theorem~\ref{thm:main-strand} using decorated permutations.

\begin{thm}\label{thm:main-perm}
	Let $Q$ be a thin dimer model with decorated permutation $\pi$. Let $v_1$ and $v_2$ be distinct boundary vertices of $Q$. The minimal path $p$ from $v_1$ to $v_2$ is arrow-defining if and only if
	\begin{enumerate}
		\item\label{GF1E} for every vertex $w\in(v_2,v_1)$, there is a number $j\in[v_2^{cl},w^{cc}]$ such that $\pi(j)\in[w^{cl},v_1^{cc}]$, and
		\item\label{GF2E} for every vertex $w\in(v_1,v_2)$, there is a number $j\in[w^{cl},v_2^{cc}]$ such that $\pi(j)\in[v_1^{cl},w^{cc}]$.
	\end{enumerate}
	In this case, 
	${{\Y}(p)}=\#\{i\in[n]\ :\ v_2^{cl}\leq i<\pi(i)\leq v_1^{cc}\textup{ is a clockwise ordering}\}$, and
	${\X}(p)=\#\{j\in[n]\ :\ v_1^{cl}\leq\pi(j)<j\leq v_2^{cc}\textup{ is a clockwise ordering}\}$.
\end{thm}
We remark that~\eqref{GF1E} states precisely that for any vertex $w$ clockwise of $v_2$ and counter-clockwise of $v_1$, there is some strand $z$ such that $v_1$ and $v_2$ are to the right of $z$, but $w$ is to the left of $z$. Condition~\eqref{GF2E} may be interpreted symmetrically.
\begin{proof}
	It is immediate that~\eqref{GF1E} is equivalent to Corollary~\ref{cor:cr}~\eqref{gf1} and that~\eqref{GF2E} is equivalent to Corollary~\ref{cor:cr}~\eqref{gf2}.
	Then Corollary~\ref{cor:cr} and the first part of Theorem~\ref{thm:main-strand} shows that the minimal path $p$ from $v_1$ to $v_2$ is arrow-defining if and only if~\eqref{GF1E} and~\eqref{GF2E} both hold.

	Moreover, note that some $i\in[n]$ satisfies $v_2^{cl}\leq i<\pi(i)\leq v_1^{cc}$ if and only if the strand $z_i$ is in $\CL(v_2,v_1)$, and some $j\in[n]$ satisfies $v_1^{cl}\leq\pi(j)<j\leq v_2^{cc}$ if and only if $z_j\in\CC(v_2,v_1)$.
	Then applying the second part of Theorem~\ref{thm:main-strand} yields the desired values of ${\Y}(p)$ and ${\X}(p)$.
\end{proof}

We remark that if $v_2+v_1$, the interval $(v_2,v_1)$ is empty, hence condition~\eqref{GF1E} of Theorem~\ref{thm:main-perm} is vacuous and one must only check condition~\eqref{GF2E} in order to verify that there is an arrow-defining path between $v_1$ and $v_2$.
Similarly, if $v_1+1=v_2$ then condition~\eqref{GF2E} is vacuous and one must only check condition~\eqref{GF1E}.

One may see that the calculations of Example~\ref{ex:2} work based only on the start and end points of the strands, hence may be done by thinking only of the decorated permutation.

\subsection{Grassmann Necklaces}

Grassmann necklaces~\cite{OPSZ,ZPostnikov} are another positroid cryptomorphism.
In this section, we show that the arrow-defining paths and relation numbers of the boundary algebra of a positroid also admit a nice interpretation in terms of the Grassmann necklace of the positroid.

\begin{defn}[{\cite[Definition 16.1]{ZPostnikov}}]
	A \emph{Grassmann necklace} of type $(k,n)$ is a sequence $\mathcal I=(I_1,\dots,I_n)$ of subsets $I_i\in{[n]\choose k}$ such that, for $i\in[n]$, we have $I_{i+1}\supseteq I_i\backslash\{i\}$, where $i+1$ is calculated modulo $n$.
\end{defn}
In other words, either $I_{i+1}=I_i$ or $I_{i+1}$ is obtained from $I_i\ni i$ by deleting $i$ and adding in a different element.
\begin{defn}
	Let $Q$ be a dimer model. Its associated {Grassmann necklace} is the sequence $\mathcal I_Q=(I_1,\dots,I_n)$ of $k$-subsets of $[n]$ defined by
	\[I_i=\{j\in[n]\ :\ \text{the boundary vertex }i\text{ is to the right of the strand }z_j\text{ starting at the strand-vertex }j\}.\]
\end{defn}

If $Q$ is thin, then no strand starts and ends at the same strand-vertex, hence $I_i\neq I_{i+1}$ for all $i$.
The following corollary is a translation of Theorem~\ref{thm:main-perm} to the language of Grassmann necklaces.

\begin{cor}\label{cor:main-neck}
	Let $Q$ be a thin dimer model with Grassmann necklace $\mathcal I_Q=(I_1,\dots,I_n)$.
	Let $v_1$ and ${v_2}$ be distinct boundary vertices of $Q$. There is an arrow-defining path $p$ from $v_1$ to $v_2$ if and only if
	\begin{enumerate}
		\item\label{FG1} for any vertex $w\in(v_2,v_1)$, we have $I_w\not\supseteq I_{v_1}\cap I_{v_2}$, and
		\item\label{FG2} for any vertex $w\in(v_1,v_2)$, we have $I_w\not\subseteq I_{v_1}\cup I_{v_2}$.
	\end{enumerate}
	The relation numbers are calculated as ${\Y}(p)=\big|[v_2,v_1)\cap I_{v_1}\cap I_{v_2}\big|$ and ${\X}(p)=\big|[v_1,v_2)\backslash\big(I_{v_1}\cup I_{v_w}\big)\big|$.
\end{cor}
\begin{proof}
	Fix $w\in(v_2,v_1)$.
	For any $j\in[n]$, we have that $j\in I_{v_1}\cap I_{v_2}$ but $j\not\in I_w$ if and only if $j\in[v_2^{cl},w^{cc}]$ and $\pi(j)\in[w^{cl},v_1^{cc}]$.
	Hence, the condition that $w$ is to the left of some strand starting in $[v_2^{cl},w^{cc}]$ and ending in $[w^{cl},v_1^{cc}]$ (i.e., condition~\eqref{GF1E} of Theorem~\ref{thm:main-perm}) is equivalent to the condition that there is some element of $I_{v_1}\cap I_{v_2}$ which is not an element of $I_w$ (i.e., condition~\eqref{FG1} of Corollary~\ref{cor:main-neck}).
	One may similarly show that condition~\eqref{GF2E} of Theorem~\ref{thm:main-perm} is equivalent to condition~\eqref{FG2} of Corollary~\ref{cor:main-neck}.
	Then applying Theorem~\ref{thm:main-perm} shows that~\eqref{FG1} and~\eqref{FG2} are equivalent to the existence of an arrow-defining path $p$ from $v_1$ to $v_2$.

	The same theorem states that ${\Y}(p)$ is the number of integers $i$ such that $v_2^{cl}\leq i<\pi(i)\leq v_1^{cc}$ is a clockwise order. In fact, this condition on $i$ is equivalent to the condition that $i\in I_{v_1}\cap I_{v_2}\cap[v_2,v_1)$. This shows the statement about ${\Y}(p)$; the statement about ${\X}(p)$ is proved similarly.
\end{proof}

\begin{example}
	We revisit the dimer model of Example~\ref{ex:2}.
	The Grassmann necklace is given by
	\[\left(123,\ 234,\ 134,\ 145,\ 156,\ 126\right),\]
	where we write, for example, $123$ to denote $\{1,2,3\}$ for readability.
	Let $v_1=3$ and $v_2=1$; we have 
	$I_3=134$ and $I_1=123$. The intersection $I_3\cap I_1=13$ is not contained in $I_2=234$.
	The union $I_1\cup I_3=1234$ does not contain $I_4$, $I_5$, or $I_6$. Hence, Corollary~\ref{cor:main-neck} shows that there is an arrow-defining path $p$ from 3 to 1. Its relation number ${\Y}(p)$ is $\big|\{1,2\}\cap(I_1\cap I_3)\big|=\big|\{1\}\big|=1$, hence $[p(xy)]=[y_2^2]$.

	On the other hand, if we take (for example) $v_1=4$ and $v_2=1$, we see that $I_4\cap I_1=\{1\}$ is contained in $I_3$, hence Corollary~\ref{cor:main-neck}~\eqref{FG1} shows that there is no arrow-defining path from 4 to 1.
\end{example}

\subsection{Grassmannian Relations}

Baur, King, and Marsh showed in~\cite{ZBKM} that if $Q$ comes from a Postnikov $(k,n)$-diagram $D$ (i.e., if the decorated permutation of $Q$ sends vertex $i$ to $i-k$ for each $i$), then the relations $[xy]=[yx]$ and $[x^k]=[y^{n-k}]$ hold. In~\cite[Proposition 3.6]{CKP} it was shown that these same relations hold for general thin dimer models in disks. The former relation $[xy]=[yx]$ was explicitly shown in Lemma~\ref{x-y-paths-thin}~\eqref{62}. In this subsection, we provide a more direct proof of the latter relation $[x^k]=[y^{n-k}]$ using the theory of relation numbers and rightmost paths developed above.

\begin{lemma}\label{lem:grass-relations-k}
	If $Q$ is a thin dimer model, then the relation $[x^k]=[y^{n-k}]$ holds in $A_Q$, where $k$ is the number of noninversions of the decorated permutation $\pi_Q$.
\end{lemma}
\begin{proof}
	We wish to show that $[x_j^k]=[y_n^{n-k}]$ for all $j\in[n]$, where $k$ is the number of noninversions of $\pi$.
	In fact, it is sufficient to show that the relation $[x_1^k]=[y_{n}^{n-k}]$ holds, since the number of (non)inversions does not change upon cyclically shifting the vertex labels.
	For readability, we thus perform all calculations for $j=1$, even though they do not fundamentally differ for arbitrary $j\in[n]$.

	Let $V^l$ be the set of strand-vertices $[1^{cl},(k+1)^{cc}]$ and let $V^r$ be the set of strand-vertices $[(k+1)^{cl},1^{cc}]$.
	Let $S_{lr}$ be the set of strands starting at a vertex of $V^l$ and ending at a vertex of $V^r$. Let $S_{rl}$ be the set of strands starting at a vertex of $V^r$ and ending at a vertex of $V^l$.

	By counting the strands beginning in $V^r$ (of which there are $n-k$), one observes that 
		$n-k=|S_{rl}|+|\CC(1,k+1)|+|\CL(k+1,1)|$.
	By noting that the inversions of $\pi$ (of which there are $n-k$) are precisely the beginning strand-vertices of strands of $S_{rl}$, $\CC(k+1,1)$, and $\CC(1,k+1)$, one observes that 
	$n-k=|S_{rl}|+|\CC(1,k+1)|+|\CC(k+1,1)|$.
	Subtracting these descriptions of $n-k$, we see that
	\[0=(n-k)-(n-k)=|\CL(k+1,1)|-|\CC(k+1,1)|.\]
	Let $p$ be the minimal path from $1$ to $k+1$. Theorem~\ref{thm:main-strand} states that ${\X}(p)=|\CC(k+1,1)|$ and ${\Y}(p)=|\CL(k+1,1)|$ -- the above calculation shows that these are the same, hence ${\X}(p)={\Y}(p)$. In other words,
	\[[x_1^k]=[p(xy)^{{\X}(p)}]=[p(xy)^{{\Y}(p)}]=[y_n^{n-k}],\]
	completing the proof.
\end{proof}
\begin{remk}
	As in~\cite[Definition 2.5]{CKP}, the value $k$ (referred also as the \textit{type} of $Q$) may alternatively be calculated as
	\[\#\{\text{counter-clockwise faces}\}-\#\{\text{clockwise faces}\}+\#\{\text{boundary arrows in clockwise faces}\}.\]
	Moreover, when $Q$ comes from a $\Le$-diagram $D$~\cite{ZPostnikov}, $k$ is the height of $D$.
\end{remk}

We call the relations $[xy]=[yx]$ and $[x^k]=[y^{n-k}]$
the \textit{Grassmannian relations} of the dimer algebra $A_Q$.
The permutation $256134$ (in one-line notation) of the model of Example~\ref{ex:2} (Figure~\ref{fig:ex2}) has three noninversions, hence $k=3$ and $[x^3]=[y^3]$.

\section{Describing Boundary Algebras}\label{sec:dba}

We use the results of the previous sections to calculate the boundary algebra of a thin dimer model with a given connected decorated permutation.
Let $\pi$ be a connected decorated permutation with $k$ noninversions and let $Q$ be a thin dimer model whose decorated permutation is $\pi$, which exists by Lemma~\ref{lem:perm-to-dimer}.
Recall the definition of the \textit{boundary algebra} $B_Q$ of $Q$ from Definition~\ref{defn:ba} and the definition of an \textit{arrow-defining path} from
Definition~\ref{def:arrow-defining-paths}.
Since arrow-defining paths are minimal, there is at most one arrow-defining path up to path-equivalence from $v_1$ to $v_2$ for any vertices $v_1,v_2\in Q$. It follows from Lemma~\ref{lem:arrow-defining-defines-arrow} that arrow-defining paths correspond to arrows of the Gabriel quiver of the boundary algebra $B_Q$.
We say that an arrow of the boundary algebra (equivalently, an arrow-defining path) is \textit{nonadjacent} if its source and target vertices are not (cyclically) adjacent to each other.

Let $Q_{k,n}$ be the \emph{circle quiver} on $n$ vertices such that each $i$ has a single arrow $x_i$ to $i+1$ and a single arrow $y_{i-1}$ to $i-1$ modulo $n$. Baur, King, and Marsh showed in~\cite[Corollary 10.4]{ZBKM} that the boundary algebra of a thin $(k,n)$-dimer model (i.e., one whose decorated permutation sends $j\mapsto j-k$) is isomorphic to $Q_{k,n}/I_{k,n}$, for $I_{k,n}$ the ideal generated by the Grassmannian relations $[xy]=[yx]$ and $[x^k]=[y^{n-k}]$ of Lemma~\ref{x-y-paths-thin} and Lemma~\ref{lem:grass-relations-k} is the boundary algebra of the Grassmannian $\Gr({k,n})$.

Let $Q^{\pi}_{\circ}$ be the quiver obtained by starting with $Q_{k,n}$ and drawing an arrow $\alpha_p:t(p)\to h(p)$ for every (equivalence class of a) nonadjacent arrow-defining path $[p]$. Theorem~\ref{thm:main-perm} shows how the (nonadjacent) arrow-defining paths and their relation numbers may be calculated from the strand diagram, justifying the notation. Let $I^{\pi}_{\circ}$ be the cancellative closure of the ideal generated by sum of the Grassmannian relations and the \emph{nonadjacent relations}:
\[I^{\pi}_{\circ}=\big\langle\{[xy]-[yx],\ [x^k]-[y^{n-k}],\ \sum_{p\in\textup{NBP}(Q)}[y_{t(p)-1}^{\reach_y(p)}]-[p(xy)^{{\Y}(p)}]\}\big\rangle^\bullet,\]
where $\textup{NBP}(Q)$ is the set of nonadjacent boundary paths of $Q$ up to equivalence class. Again, Theorem~\ref{thm:main-perm} describes how to calculate the elements of $\textup{NBP}(Q)$ and their relation numbers based only on $\pi$.
By Lemma~\ref{x-y-paths-thin}, Lemma~\ref{lem:grass-relations-k}, and Theorem~\ref{thm:main-perm},
these relations hold in the dimer algebra, hence they must hold in the boundary algebra. Then the boundary algebra $B_Q$ is a quotient of $\k Q^{\pi}_{\circ}/I^{\pi}_{\circ}$. We show that, in fact, equality holds.

\begin{thm}\label{thm:boundary-algebra-main}
	Let $Q$ be a thin dimer model with decorated permutation $\pi$. Then its boundary algebra $B_Q$ is isomorphic to $B^\pi_\circ:=\k Q^{\pi}_{\circ}/I^{\pi}_{\circ}$.
\end{thm}
\begin{proof}
	As explained above the theorem statement, $B_Q$ is some quotient of $\k Q^{\pi}_{\circ}/I^{\pi}_{\circ}$. Then $B_Q\cong \k Q^{\pi}_{\circ}/J$ for some set of relations $J$ containing $I^{\pi}_{\circ}$. If $B_Q\not\cong \k Q^{\pi}_{\circ}/I^{\pi}_{\circ}$, then this containment is strict and we may choose a relation of $J$ which is not in $I^{\pi}_{\circ}$. Since the relations of $A_Q$ are generated by commutation relations $[p]-[q]$ (for paths $p$ and $q$ of $Q$ with $h(p)=h(q)$ and $t(p)=t(q)$), the algebra $B_Q$ is also generated by commutation relations and we may choose a relation of the form $[p]-[q]$ for paths $p$ and $q$ of $Q^{\pi}_{\circ}$.
    
    We perform a series of simplifications. First, we show that $p$ and $q$ may be assumed to consist only of $x$'s and $y$'s. First, the boundary path $p$ may be taken to be of the form \[p=x_{t(p)}^{r_1}y^{s_1}a_1x^{r_2}y^{s_2}a_2\dots a_{t-1}x^{r_t}y^{r_t}a_tx^{r_{t+1}}y^{s_{t+1}}\] for nonadjacent arrow-defining paths $a_i$ and $r_i,s_j\geq0$. We similarly may take $q$ to be of the form 
    \[q=x_{t(p)}^{r'_1}y^{s'_1}a'_1x^{r'_2}y^{s'_2}a'_2\dots a'_{t'-1}x^{r'_{t'}}y^{r'_{t'}}a'_{t'}x^{r'_{t'+1}}y^{s'_{t'+1}}.\]
	Without loss of generality suppose that $M:=\sum_{i=1}^t {\Y}(a_i)$ is greater than or equal to $\sum_{i=1}^{t'}{\Y}(a'_i)$. We now multiply both sides of the relation $[p]-[q]$ by $[(xy)^M]$. 
    Since $J$ and $I^{\pi}_{\circ}$ both obey the cancellation property,
    \begin{align*}
	    [p]-[q]\in J&\iff [px^{M}]-[qx^{M}]\in J\\
        \text{and}\\
	    [p]-[q]\not\in I^{\pi}_{\circ}&\iff [px^M]-[qy^M]\not\in I^{\pi}_{\circ}.
    \end{align*}
	We then see that $[px^M]-[qy^M]$ is a relation in $J\backslash I^{\pi}_{\circ}$.
	Using that $[x]$'s and $[y]$'s commute (by Lemma~\ref{x-y-paths-thin}) and using the relations $[y_{t(a_i)-1}^{\reach_y(a_i)}]-[a_i(xy)^{{\Y}(a_i)}]\in I_\circ^\pi\cap J$, we may substitute out all nonadjacent arrow-defining paths $a_i$ from the left hand side of this relation to get that $[px^M]-[x_{t(p)}^{m_1}y^{m_2}\dots x^{m_t}y^{m_t}]\in J\cap I_\circ^\pi$ for some $m_i\geq0$. Similarly, we see that $[qx^M]-[x_{t(p)}^{m'_1}y^{m'_2}\dots x^{m'_t}y^{m'_{t'}}]\in J\cap I_\circ^\pi$ for some $m'_i\geq0$. Hence, the relation $[x_{t(p)}^{m'_1}y^{m'_2}\dots x^{m'_{t'}}y^{m'_t}]-[x_{t(p)}^{m_1}y^{m_2}\dots x^{m_t}]$ is in $J\backslash I_\circ^\pi$.
	Moreover, since $x$'s and $y$'s commute (Lemma~\ref{x-y-paths-thin}), we may write this relation as $[x_j^{b_1}y^{c_1}]-[x_j^{b_2}y^{c_2}]\in J\backslash I_\circ^\pi$ for some $b_i,c_i\geq0$. By applying the Grassmannian relation $[x^k]-[y^{n-k}]\in J\cap I_\circ^\pi$, we may further assume that $0\leq b_i<k$.
    Without loss of generality suppose $b_1\geq b_2$. 

	If $c_1\geq c_2$, then cancelling $x^{b_2}$ from the left and $y^{c_2}$ from the right via the cancellation property gives that the relation \[[x_j^{b_1-b_2}y^{c_1-c_2}]-[e_j]\] is in $J$ but not $I^\pi_{\circ}$.
	Since this relation is not in $I^{\pi}_{\circ}$, the left hand side must not be constant. On the other hand, a nonconstant path may not be equivalent to a constant path in $B_Q\cong\k Q_\circ^\pi/J$, as this relation would also hold in the dimer algebra $A_Q$, so we have a contradiction.
    
	We may then suppose that $c_1\leq c_2$. The cancellation property (of $J$ and $I^\pi_\circ$) gives that the relation $[x_j^{b_1-b_2}]-[y_{j-1}^{c_2-c_1}]$ is in $J$ but not $I^\pi_\circ$.
	Since a constant path cannot be equivalent to a nonconstant path in $B_Q\cong\k Q_\circ^\pi/J$, we must have $0<b:=b_1-b_2$ and $0<c:=c_2-c_1$.
	Moreover, since $0\leq b_1<k$, we have $0\leq b<k$. We must have $b+c\geq n$ in order for the starts and ends of both sides to match, hence $c>(n-k)$. Multiplying the equation $[x_j^b]-[y_{j-1}^c]$ by $[x^{k-b}]$ gives a relation $[y_{j-1}^{n-k}]=[x_j^k]=[y_{j-1}^{c-k+b}(xy)^{k-b}]$ of $J$. Since $b+c\geq n$, we must have $c-k+b\geq n-k$, so cancelling out $y_{j-1}^{n-k}$ from the left hand side gives $[e_j]=[y_{j-1}^{c-n+b}(xy)^{k-b}]$ in $B_Q\cong\k Q_\circ^\pi/J$, a contradiction.
\end{proof}

\begin{remk}
	Theorem~\ref{thm:boundary-algebra-main} tells us that the combinatorics of the nonadjacent arrows determines the boundary algebra $B^\pi$.
	More concretely, the data of
	\begin{enumerate}
		\item the pair $(k,n)$,
		\item the set $S$ of nonadjacent arrows of $B^\pi$, and 
		\item a right relation number ${\Y}(\alpha)$ for each arrow of $s$,
	\end{enumerate}
	which is obtained in Theorem~\ref{thm:main-strand} from the permutation $\pi$, is enough to determine the boundary algebra $B^\pi$. In future works, we will give a combinatorial characterization of the forms this data may take.
\end{remk}

Theorem~\ref{thm:boundary-algebra-main} gives a presentation $\k Q^\pi_\circ/I^\pi_\circ$ of the boundary algebra of a connected positroid $\pi$. On the other hand, the ideal $I^{\pi}_{\circ}$ may not be contained in the square of the arrow ideal. In other words, the quiver $Q^{\pi}_{\circ}$ appearing in Theorem~\ref{thm:boundary-algebra-main} may not be the Gabriel quiver of $B_Q$ because certain arrows $x_i$ and $y_i$ of $Q_\circ^\pi$ may not appear in the Gabriel quiver of $B_Q$. We now account for this.

Theorem~\ref{thm:main-perm} may be used to find the adjacent paths $x_i$ and $y_i$ which are not arrow-defining based on the strand diagram.
Let $Q^{\pi}$ be the quiver obtained from $Q^{\pi}_{\circ}$ by deleting all arrows $x_i$ and $y_i$ which are not arrow-defining. In $Q^\pi$, define $x'_i$ to be the path from $i$ to $i+1$ as follows: either $x_i$ is arrow-defining and $x'_i:=x_i$, or $x_i$ is not arrow-defining and $x'_i$ is a composition of arrow-defining paths realizing $x_i$.
We define $x'=\sum_{i\in[n]}x'_i$, and we define $y'_i$ and $y'$ analogously.
We may then define an ideal $I^\pi$ of $Q^\pi$ by taking the definition for $I_\circ^\pi$ and replacing $x$ with $x'$ and $y$ with $y'$:
\[I^{\pi}=\langle\{[x'y']-[y'x'],[(x')^k]-[(y')^{n-k}],\sum_{p\in\textup{NBP}(Q)}[(y_{t(p)-1}')^{\reach_y(p)}]-[p(x'y')^{{\Y}(p)}]\}\rangle^\bullet.\]

\begin{cor}\label{cor:boundary-algebra-gabriel-quiver-formula}
    If $Q$ is a thin dimer model with decorated permutation $\pi$, then $I^\pi$ is an admissible ideal of $Q^\pi$ and there is an isomorphism $B_Q\cong \k Q^{\pi}/I^{\pi}$.
\end{cor}

Corollary~\ref{cor:boundary-algebra-gabriel-quiver-formula} then gives an \emph{admissible} presentation $\k Q^\pi/I^\pi$ of the boundary algebra of a connected positroid $\pi$.

\begin{example}\label{ex:22}
	We return to our running Example~\ref{ex:2}. We already calculated the Gabriel quiver based on the strand diagram (and decorated permutation and Grassmann necklace). There is one nonadjacent arrow $\alpha:3\to 1$ with ${\X}(\alpha)=2$ and ${\Y}(\alpha)=1$. This gives the relations
	\[[\alpha(xy)^2]=[x_3^4]\textup{ and }[\alpha(xy)]=[y_2^2].\]
	The decorated permutation given by the strand diagram is $2 5 6 1 3 4$ (in one-line notation) and has 3 noninversions, so $k=3$. Hence, the Grassmannian relations are $[xy]=[yx]$ and $[x^3]=[y^3]$.
	By adding an arrow $\alpha$ from 3 to 1 to the circle quiver $Q_{k,n}$, we obtain $Q^{\pi}_{\circ}$, pictured in the center of Figure~\ref{fig:25}. By Theorem~\ref{thm:boundary-algebra-main}, taking the cancellative closure of the ideal generated by
	\[\big([xy]-[yx]\big),\ \big([x^3]-[y^3]\big),\ \big([\alpha(xy)]-[y_2^2]\big)\]
	gives the ideal $I_\circ^\pi$ such that $B_Q\cong\k Q_\circ^\pi/I_\circ^\pi$.

	In this case, we may give a more concrete description of the relations of $B_Q$.
	As calculated in Example~\ref{ex:2}, one may start from the relations $[\alpha(xy)^2]=[x_3^4]$ and $[\alpha(xy)]=[y_2^2]$ and use that $[xy]$ commutes with everything and the relation $[xy]=[yx]$ to calculate
	\[\begin{matrix}
		[\alpha(yx)^2]=[x_3^4] & [(xy)\alpha(yx)]=[x_3^4] & [(xy)^2\alpha]=[x_3^4] &  [\alpha(xy)]=[y_2^2] & [(yx)\alpha]=[y_2^2] \\
		[\alpha y^2]=[x_3^2] & [y\alpha y]=[x_4^2] & [y^2\alpha]=[x_5^2] &
		[\alpha x]=[y_2] & [x\alpha]=[y_1],
	\end{matrix} \]
	where the bottom row is obtained by cancelling the $x$'s from the extreme sides of the top row. The relations of the bottom row are the cyclic sets of relations associated to $\alpha$ as in Proposition~\ref{thm:cyc-set}. In this case, these relations (along with the Grassmannian relations) generate the ideal $I^{\pi}_{\circ}$; this is not true for arbitrary thin dimer models, as we will see in Example~\ref{ex:fin}. 

	Note that $I^{\pi}_{\circ}$ is not admissible as it contains the relations $[\alpha x]=[y_2]$ and $[x\alpha]=[y_1]$. Hence, to get $Q^{\pi}$ one observes as in Example~\ref{ex:2} that $y_1$ and $y_2$ are not arrow-defining and removes them from $Q^{\pi}_{\circ}$. Now, to get the (admissible) ideal $I^{\pi}$, we set 

	\[x'_i:=x_i \text{ for }i\in[6], \text{ and } y'_i=\begin{cases}
		y_i & i\in\{3,4,5,6\} \\
		x_2\alpha & i=2 \\
		\alpha x_1 & i=3,
	\end{cases}\] 
	and we take the ideal generated by the remaining relations:
	\[I^{\pi}:=\left\langle\big([x'y']-[y'x']\big),\ \big([(x')^3]-[(y')^3]\big),\ 
	\big([\alpha (y')^2]-[(x'_3)^2]\big) + \big([y'\alpha y']-[(x'_4)^2]\big) + \big([(y')^2\alpha]-[(x'_5)^2]\big)\right\rangle.\]
	\begin{figure}[H]
		\centering
		\def\svgscale{0.21}
		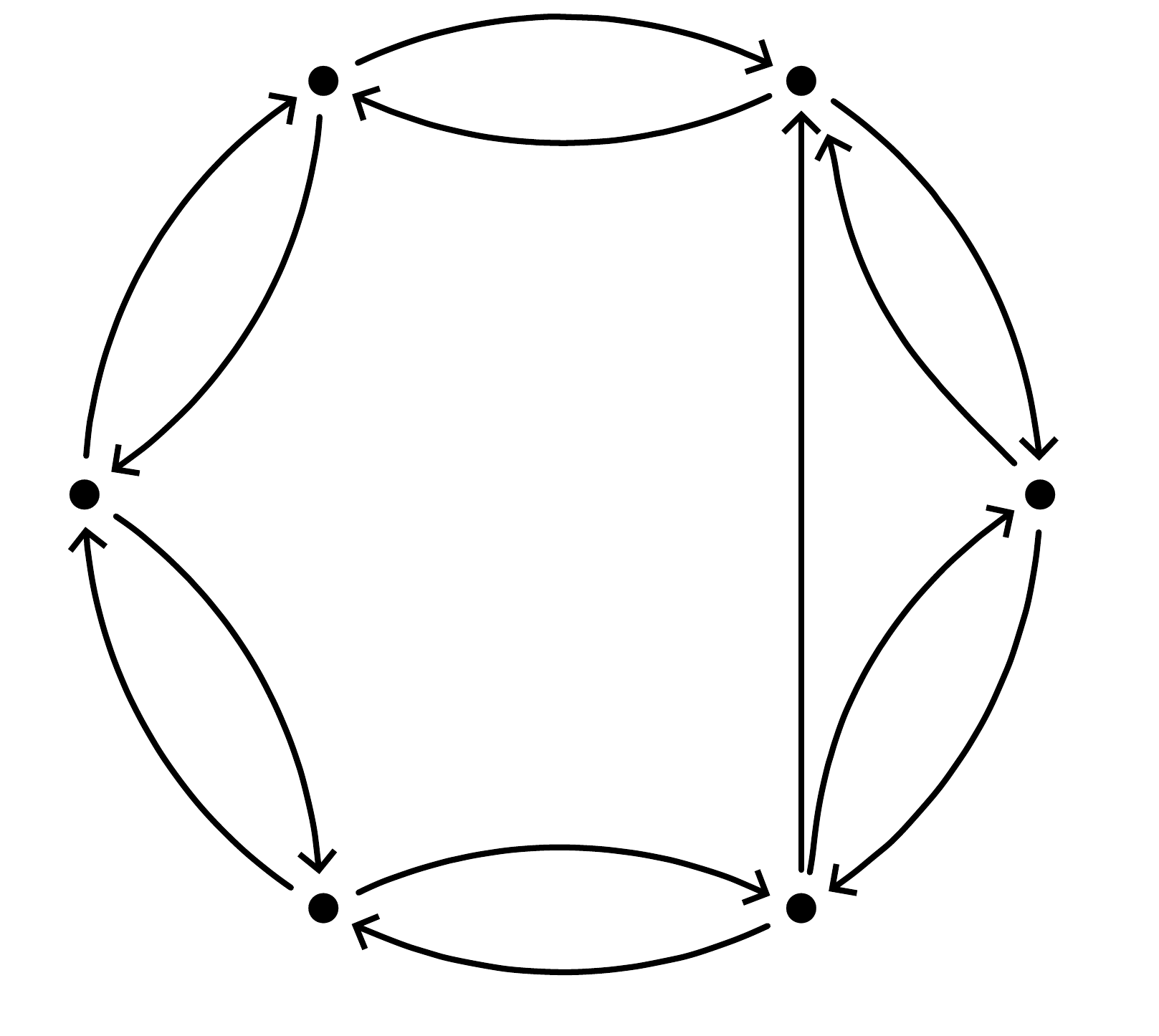
		\hspace{1cm}
		\def\svgscale{0.21}
		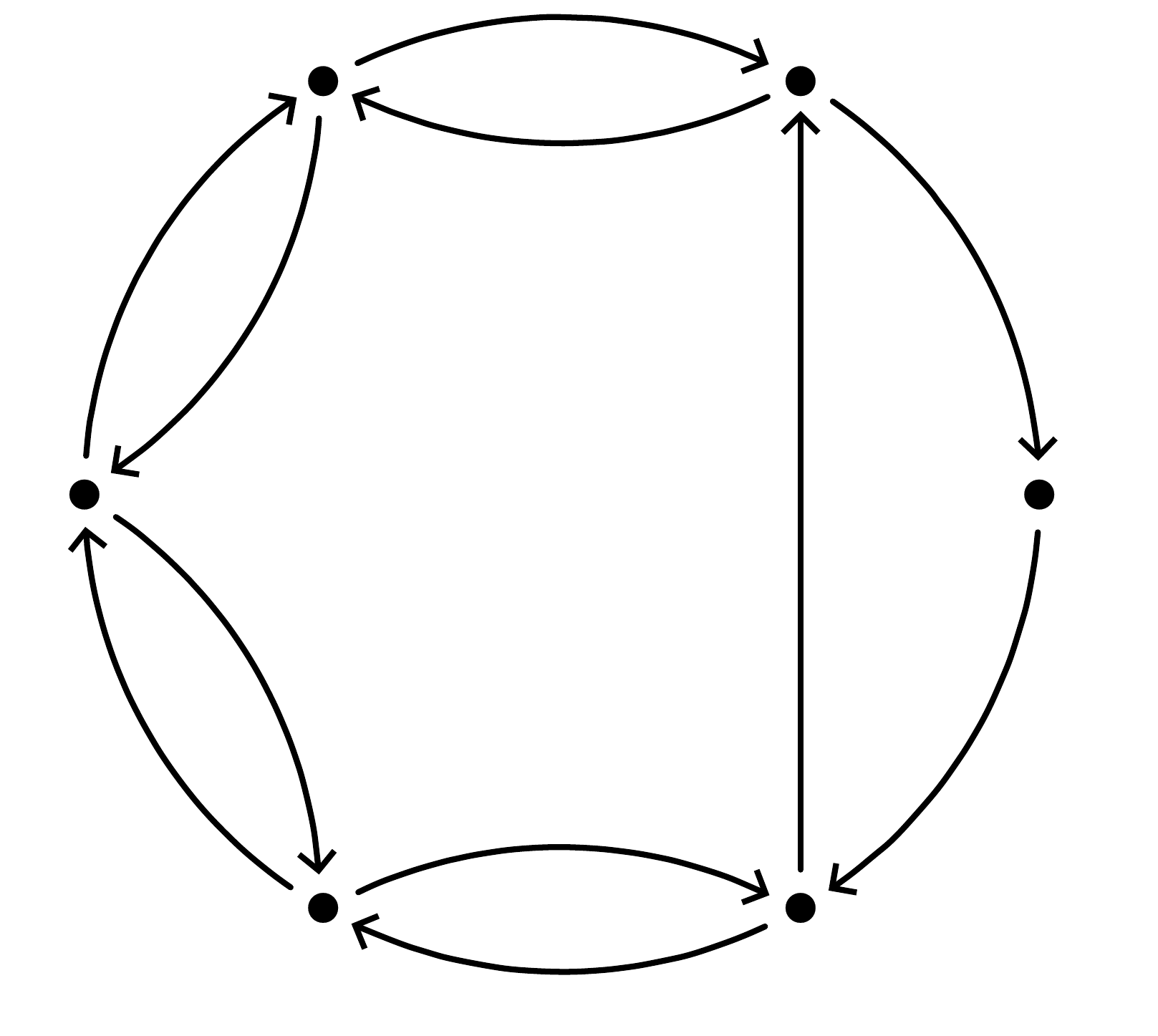
		\caption{Shown are the quivers $Q^{\pi}_{\circ}$ and $Q^{\pi}$ of the dimer model of Example~\ref{ex:2} (Figure~\ref{fig:ex2}).}
		\label{fig:25}
	\end{figure}

\end{example}

\begin{example}\label{ex:fin}
	Let $\pi$ be the decorated permutation of $[n]$ given by $4 5 8 2 9 1 6 7 3$ in one-line notation. This permutation has $5$ inversions and $k=4$ noninversions.
	On the left of Figure~\ref{fig:fin} is shown a visual representation of this decorated permutation. The strands of $\CL(9,6)$ and $\CC(9,6)$ are highlighted, and one may check that the conditions of Theorem~\ref{thm:main-perm} are satisfied for there to be an arrow $\alpha:6\to 9$ with relation number ${\Y}(\alpha)=|\CL(9,6)|=3$. One may similarly check the existence of an arrow $\beta:2\to 5$ with relation number ${\Y}(\beta)=2$. There are no other nonadjacent arrows of the boundary algebra. 
	By adding these arrows to the circle quiver, we obtain $Q^{\pi}_{\circ}$, pictured in the middle of Figure~\ref{fig:fin}. By Theorem~\ref{thm:boundary-algebra-main}, when we take $I^{\pi}_{\circ}$ to be the cancellative closure of the ideal generated by
\[\big([xy]-[yx]\big)+\big([x^4]-[y^5]\big)+\big([\alpha(xy)]-[y_2^2]\big),\]
	we obtain the boundary algebra $B_Q\cong\k Q^{\pi}_{\circ}/I^{\pi}_{\circ}$.

	Theorem~\ref{thm:main-perm} shows that $x_6$, $x_7$, and $x_8$ are the only adjacent arrows absent from the Gabriel quiver of the boundary algebra. By removing these arrows and adjusting the relations accordingly, we obtain an admissible bound path algebra $B_Q\cong\k Q^{\pi}/I^{\pi}$.
	\begin{figure}[H]
		\centering
		\def\svgscale{0.21} 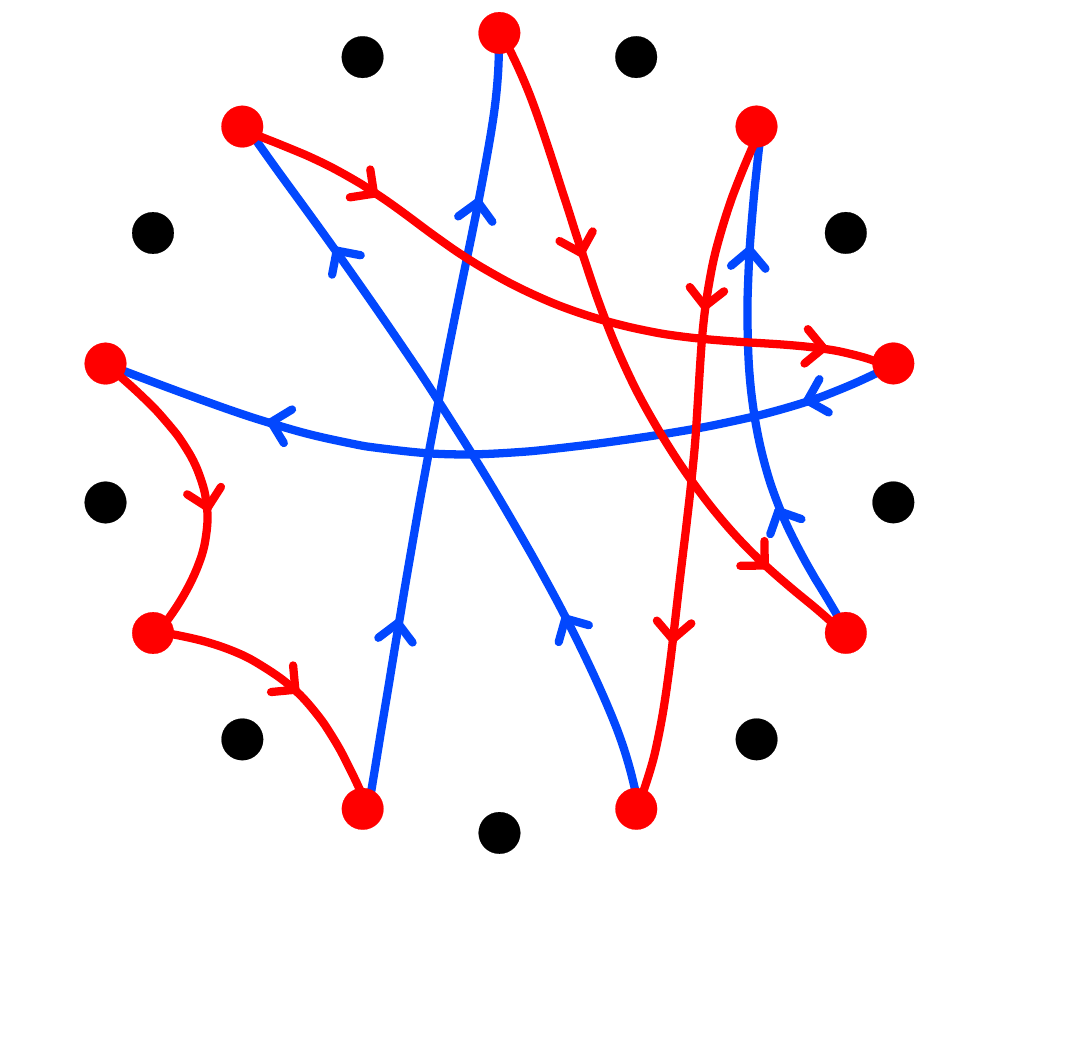
		\hspace{.5cm}\def\svgscale{0.21}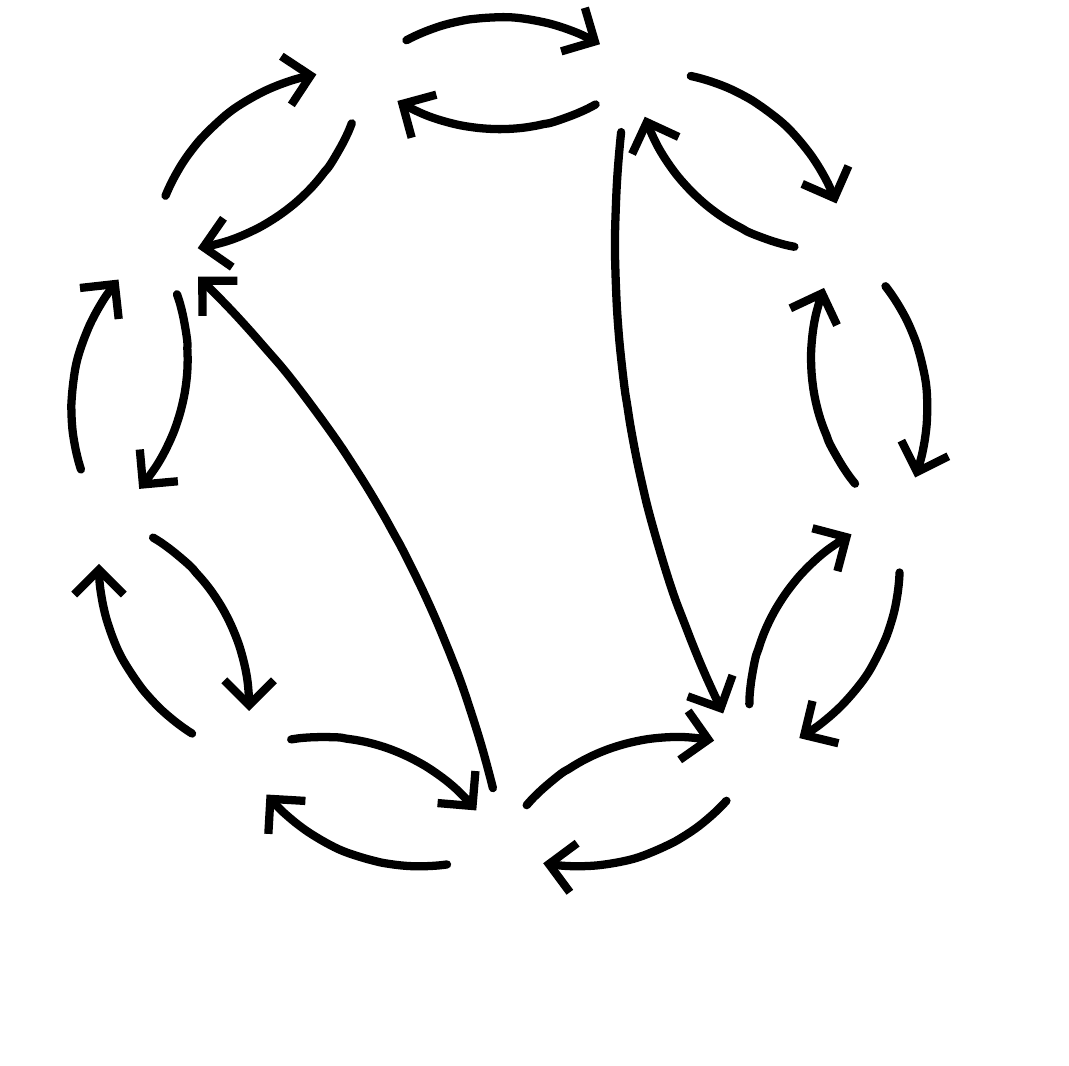
		\hspace{.5cm}\def\svgscale{0.21}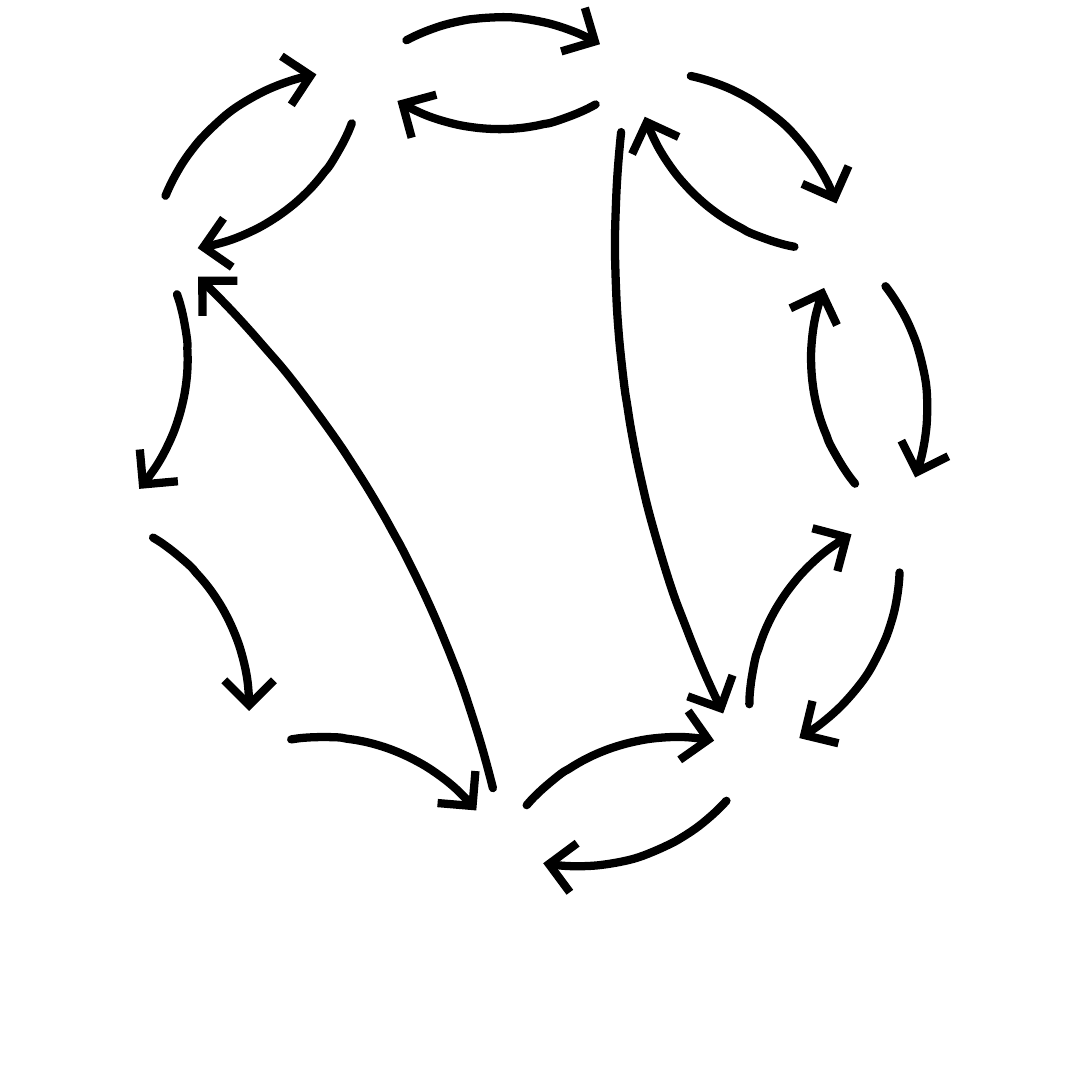
		\caption{Shown on the left is a visual representation of the permutation $458291673$ and the nonadjacent arrow-defining paths and relation numbers obtained from a thin dimer model with that permutation. The ``strands'' of $\CL(9,6)$ and $\CC(9,6)$ are in red, and the others are blue. In the middle is the quiver $Q^{\pi}_{\circ}$ and on the right is the quiver $Q^{\pi}$.}
		\label{fig:fin}
	\end{figure}

	We may calculate the cyclic sets of relations (introduced in Proposition~\ref{thm:cyc-set}) of the nonadjacent arrows ${\alpha}$ and $b$ as follows:
	\begin{align*}
		[x_2^2]=[\beta y] && [x_3^2]=[y\beta] && [\beta x^2]=[y_1^4] && [x{\beta}x]=[y_9^4] && [x^2{\beta}]=[y_8^4]
	\end{align*}
	\vspace{-.3 in}
	\begin{align*}
		[x_6]=[{\alpha}y^2] && [x_7]=[y{\alpha}y] && [x_8]=[y^2{\alpha}] && [{\alpha}x^3]=[y_5^3] && [x{\alpha}x^2]=[y_4^3] && [x^2{\alpha}x]=[y_3^3] && [x^3{\alpha}]=[y_2^3]
	\end{align*}

	In the dimer model of Example~\ref{ex:22}, we saw that the cyclic sets of relations and Grassmannian relations generated the ideal $I_\circ^\pi$ (without needing to take a cancellative closure); we will now show that this is not the case in the current example. We calculate as follows:
	\begin{align*}
		[x_1{\beta}x{\alpha}x^3]&=[(x_1{\beta}x)({\alpha}x^3)]\\
		&=[(y_9^4)(y_5^3)] && ([x{\beta}x]=[y_9^4]\text{ and }[{\alpha}x^3]=[y_5^3])\\
		&=[y_9^7]\\
		&=[y_9^2x^4] && ([x^4]=[y^5])\\
		&=[x_1y^2x^3] && ([xy]=[yx])
	\end{align*}
	We may cancel $[x]$ from the left and $[x^3]$ from the right of the resulting equation $[x{\beta}x{\alpha}x^3]=[xy^2x^3]$ to get the equation $[{\beta}x{\alpha}]=[y^2]$. Hence, the relation $[{\beta}x{\alpha}]=[y^2]$ is in the ideal $I_\circ^\pi$; on the other hand, this relation is not in the ideal generated by the cyclic sets of relations and the Grassmannian relations without the use of cancellative closure. We then see that the cancellative closure is necessary in Theorem~\ref{thm:boundary-algebra-main}. In future work, we will obtain a canonical minimal generating set for the ideal $I_\circ^\pi$ for any connected permutation $\pi$.
\end{example}
\begin{remk}
	We now make the mild assumption that $n\geq3$. This excludes a small number of degenerate examples~\cite[Remark 2.2]{ZPressland2015}. Under this assumption, we may let $Q$ be a thin dimer model \emph{with no digons} whose decorated permutation is $\pi$ by taking the dimer model of a reduced Postnikov diagram with this permutation~\cite[Proposition 2.10]{ZPressland2015}.
	Associated to this dimer model is a \emph{weakly separated collection} $\mathcal S_Q$, which associates to each vertex $v$ of $Q$ a $k$-subset $S_v\in{[n]\choose k}$.
	There is a map $\phi:\AA_Q\to\widetilde\Pi^\circ(\mathcal P_\pi)$ from the cluster algebra $\AA_Q$ to the open positroid variety $\widetilde\Pi^\circ(\mathcal P_\pi)$ mapping the initial cluster variable at a vertex $v$ of $Q$ to the Pl\"ucker coordinate given by the $S_v$. This realizes the cluster structure given in~\cite{GLZ}.
	We know from~\cite{ZPressland2015} that $\GP(\widehat B_Q)$ categorifies $\AA_Q$, where the initial seed is given by the cluster-tilting object $eA_Q$. Hence, we may say that $\GP(\widehat B_Q)$ categorifies the cluster structure on $\widetilde\Pi^\circ(\mathcal P_\pi)$, where the initial seed of the latter (given by the weakly separated collection $\mathcal S_Q$) is given by the cluster-tilting object $eA_Q$.

Say we choose a different dimer model $Q'$ with the same decorated permutation $\pi$. The categorification $\GP(\widehat B_Q)$ is given by the same category with a different initial seed, now given by the cluster-tilting object $eA_{Q'}$. This categorifies the same cluster structure on $\widetilde\Pi^\circ(\mathcal P_\pi)$, with the caveat that the initial seed is now given by the Pl\"ucker coordinates of elements of the weakly separated collection $\mathcal S_{Q'}$. Hence, changing the choice of dimer model $Q$ with a fixed decorated permutation $\pi$ corresponds to changing the choice of initial seed on a fixed cluster structure of $\widetilde\Pi^\circ(\mathcal P_\pi)$ and $\GP(\widehat B_Q)$.
\end{remk}

\bibliographystyle{plain}
\bibliography{biblio}

\begin{thebibliography}{10}

\bibitem{BBEL}
Karin Baur, Dusko Bogdanic, Ana~Garcia Elsener, and Jian-Rong Li.
\newblock Rigid indecomposable modules in {G}rassmannian cluster categories.
\newblock {\em arXiv preprint}, 2020.

\bibitem{ZZBBE}
Karin Baur, Dusko Bogdanic, and Ana Garcia~Elsener.
\newblock Cluster categories from {G}rassmannians and root combinatorics.
\newblock {\em Nagoya Math. J.}, 240:322--354, 2020.

\bibitem{ZZBBL}
Karin Baur, Dusko Bogdanic, and Jian-Rong Li.
\newblock Construction of rank 2 indecomposable modules in {G}rassmannian
  cluster categories.
\newblock In {\em Mc{K}ay correspondence, mutation and related topics},
  volume~88 of {\em Adv. Stud. Pure Math.}, pages 1--45. Math. Soc. Japan,
  Tokyo, 2023.

\bibitem{ZBKM}
Karin Baur, Alastair~D. King, and Bethany~R. Marsh.
\newblock Dimer models and cluster categories of {G}rassmannians.
\newblock {\em Proc. Lond. Math. Soc. (3)}, 113(2):213--260, 2016.

\bibitem{BS}
Jonah Berggren and Khrystyna Serhiyenko.
\newblock Consistent dimer models on surfaces with boundary.
\newblock {\em arXiv preprint}, 2023.

\bibitem{ZBocklandt2015}
Raf Bocklandt.
\newblock A dimer {ABC}.
\newblock {\em Bull. Lond. Math. Soc.}, 48(3):387--451, 2016.

\bibitem{CKP}
İlke \c{C}anak\c{c}ı, Alastair King, and Matthew Pressland.
\newblock Perfect matching modules, dimer partition functions and cluster
  characters.
\newblock {\em arXiv preprint}, 2021.

\bibitem{DLZ}
Laurent Demonet and Xueyu Luo.
\newblock Ice quivers with potential arising from once-punctured polygons and
  {C}ohen-{M}acaulay modules.
\newblock {\em Publ. Res. Inst. Math. Sci.}, 52(2):141--205, 2016.

\bibitem{FKZ}
Changjian Fu and Bernhard Keller.
\newblock On cluster algebras with coefficients and 2-{C}alabi-{Y}au
  categories.
\newblock {\em Trans. Amer. Math. Soc.}, 362(2):859--895, 2010.

\bibitem{GLZ}
Pavel Galashin and Thomas Lam.
\newblock Positroid varieties and cluster algebras.
\newblock {\em Ann. Sci. \'{E}c. Norm. Sup\'{e}r. (4)}, 56(3):859--884, 2023.

\bibitem{GLSZ}
Christof Geiss, Bernard Leclerc, and Jan Schr\"{o}er.
\newblock Partial flag varieties and preprojective algebras.
\newblock {\em Ann. Inst. Fourier (Grenoble)}, 58(3):825--876, 2008.

\bibitem{ZJKS}
Bernt~Tore Jensen, Alastair~D. King, and Xiuping Su.
\newblock A categorification of {G}rassmannian cluster algebras.
\newblock {\em Proc. Lond. Math. Soc. (3)}, 113(2):185--212, 2016.

\bibitem{KLSZ}
Allen Knutson, Thomas Lam, and David~E. Speyer.
\newblock Positroid varieties: juggling and geometry.
\newblock {\em Compos. Math.}, 149(10):1710--1752, 2013.

\bibitem{ZZL}
B.~Leclerc.
\newblock Cluster structures on strata of flag varieties.
\newblock {\em Adv. Math.}, 300:190--228, 2016.

\bibitem{ZZMS}
Greg Muller and David~E. Speyer.
\newblock The twist for positroid varieties.
\newblock {\em Proc. Lond. Math. Soc. (3)}, 115(5):1014--1071, 2017.

\bibitem{OPSZ}
Suho Oh, Alexander Postnikov, and David~E. Speyer.
\newblock Weak separation and plabic graphs.
\newblock {\em Proc. Lond. Math. Soc. (3)}, 110(3):721--754, 2015.

\bibitem{ZPostnikov}
Alexander Postnikov.
\newblock Total positivity, {G}rassmannians, and networks.
\newblock {\em arXiv preprint}, 2006.

\bibitem{ZPressland2015}
Matthew Pressland.
\newblock Internally {C}alabi-{Y}au algebras and cluster-tilting objects.
\newblock {\em Math. Z.}, 287(1-2):555--585, 2017.

\bibitem{ZPressland2020}
Matthew Pressland.
\newblock Mutation of frozen {J}acobian algebras.
\newblock {\em J. Algebra}, 546:236--273, 2020.

\bibitem{ZPressland2019}
Matthew Pressland.
\newblock Calabi-{Y}au properties of {P}ostnikov diagrams.
\newblock {\em Forum Math. Sigma}, 10:Paper No. e56, 31, 2022.

\bibitem{ZZP}
Matthew Pressland.
\newblock Quasi-coincidence of cluster structures on positroid varieties.
\newblock {\em arXiv preprint}, 2023.

\bibitem{ZScott}
Jeanne~S. Scott.
\newblock Grassmannians and cluster algebras.
\newblock {\em Proc. London Math. Soc. (3)}, 92(2):345--380, 2006.

\bibitem{SSWZ}
K.~Serhiyenko, M.~Sherman-Bennett, and L.~Williams.
\newblock Cluster structures in {S}chubert varieties in the {G}rassmannian.
\newblock {\em Proc. Lond. Math. Soc. (3)}, 119(6):1694--1744, 2019.

\end{thebibliography}

\end{document}